\documentclass[11pt, a4paper, reqno]{amsart}

\usepackage[dvipsnames]{xcolor}
\usepackage[utf8]{inputenc}
\usepackage[english]{babel}
\usepackage{amsmath}        
\usepackage{amsfonts}       
\usepackage{amsthm}         
\usepackage{subcaption}
\usepackage{float}
\usepackage[colorlinks]{hyperref}
\usepackage[nameinlink]{cleveref}
\usepackage{tikz-cd}
\usepackage{mathtools}
\usepackage{comment}
\usepackage{amssymb}
\usepackage{todonotes}
\usepackage{wrapfig}
\usepackage[font = small, textfont = it]{caption}
\usepackage{geometry}
\hypersetup{linkcolor=Cerulean, citecolor=Orange, urlcolor=DarkOrchid}

\theoremstyle{plain}
\newtheorem{thm}{Theorem}[subsection]
\theoremstyle{definition}
\newtheorem{defn}[thm]{Definition}
\newtheorem{proposition}[thm]{Proposition}
\newtheorem{corollary}[thm]{Corollary}
\newtheorem{lemma}[thm]{Lemma}
\newtheorem{rem}[thm]{Remark}
\newtheorem{example}[thm]{Example}
\newtheorem{question}{Question}

\newcommand{\im}{\operatorname{Im}}

\numberwithin{equation}{section}

\usepackage{xpatch}
\makeatletter   
\xpatchcmd{\@tocline}
{\hfil\hbox to\@pnumwidth{\@tocpagenum{#7}}\par}
{\ifnum#1<0\hfill\else\dotfill\fi\hbox to\@pnumwidth{\@tocpagenum{#7}}\par}
{}{}
\makeatother    

\makeatletter
 \def\l@subsection{\@tocline{2}{0pt}{1.15cm}{1cm}{}}
\def\l@subsubsection{\@tocline{3}{0pt}{8pc}{8pc}{}}
 \makeatother

\newcommand{\nocontentsline}[3]{}
\newcommand{\tocless}[2]{\bgroup\let\addcontentsline=\nocontentsline#1{#2}\egroup}

\makeatletter
\providecommand\@dotsep{5}
\def\listtodoname{List of Todos}
\def\listoftodos{\@starttoc{tdo}\listtodoname}
\makeatother

\date{}
\title{Factorization structures, cones, and polytopes}
\author{Roland Púček}
\address{Roland Púček, Mathematical Institute, University of Jena, Ernst-Abbe-Platz 1-2, 07743 Jena, Germany}
\email{\href{mailto:roland.pucek@uni-jena.de}{roland.pucek@uni-jena.de}}
\subjclass[2020]{52B11, 52C07, 52B05, 53-02}
\keywords{factorization structure, Segre-Veronese factorization structure, factorization curve, compatible polytope, compatible cone, cyclic polytope, generalised Gale's evenness condition}

\allowdisplaybreaks

\apptocmd{\thebibliography}{\fontsize{11}{15}\selectfont}{}{}

\begin{document}

\begin{abstract}
Factorization structures occur in toric differential and discrete geometry, and can be viewed in multiple ways, e.g., as objects determining substantial classes of explicit toric Sasaki and Kähler geometries, as special coordinates on such, or as an apex generalisation of cyclic polytopes featuring a generalised Gale's evenness condition.
This article presents a comprehensive study of this new concept called factorization structures. It establishes their structure theory and introduces their use in the geometry of cones and polytopes. 
The article explains a construction of polytopes and cones compatible with a given factorization structure, and exemplifies it for the product Segre-Veronese and Veronese factorization structures, where the latter case includes cyclic polytopes.
Further, it derives the generalised Gale's evenness condition for compatible cones, polytopes and their duals, and explicitly describes faces of these.
Factorization structures naturally provide generalised Vandermonde identities, which relate normals of any compatible polytope, and which are used to find examples of Delzant and rational Delzant polytopes compatible with the Veronese factorization structure.
The article offers a myriad of factorization structure examples, which are later characterised to be precisely factorization structures with decomposable curves, and raises the question if these encompass all factorization structures, i.e., the existence of an indecomposable factorization curve.
\end{abstract}

\maketitle

%

\thispagestyle{empty}

This paper offers an extensive exploration of a new concept called factorization structures, provides an introduction to their applications in discrete geometry, and serves as a foundational reference for future research in both discrete and differential geometry based on factorization structures. The central theme revolves around the interaction between the abstract notion of factorization structures, cones, and polytopes. This paper may be viewed as a contribution to the field of discrete geometry through this and original results described below.

Factorization structures have appeared in the literature in \cite{apostolov2015ambitoric,pucek2022extremal,brandenburg2024}.
They were first introduced as 2-dimensional factorization structures in the work of Apostolov, Calderbank, and Gauduchon \cite{apostolov2015ambitoric},
where were used to classify extremal metrics on toric 4-orbifolds with the second Betti number 2; 4-dimensional spaces with isolated non-smooth points whose automorphism group contains 2-torus.
This notable achievement served as an inspiration for further research, culminating in a thesis by the author \cite{pucek2022extremal}, where their latent potential was recognised and systematically studied within the context of toric Kähler and Sasaki geometry.
Most recently, motivated by results obtained in this article, the geometry of polytopes compatible with the Veronese factorization structure was explored in \cite{brandenburg2024}.\\

To motivate factorization structures we consider the familiar example of a cyclic polytope: the convex hull of finitely many points on the momentum curve $t \mapsto (t,t^2,\ldots,t^m)$ (see Figure \ref{fig:polytope}). Cyclic polytopes are famous for their extremal properties, which make
them key examples in various theorems (see \cite{grunbaum1967convex} for a historical account). They are particularly notable in the upper bound theorem \cite{mcmullen1970maximum}, where they exemplify polytopes with maximal number of faces for a given number of vertices, and they stand out as polytopes whose Ehrhart polynomial, counting lattice points in theirs dilates, has positive coefficients \cite{liu2005ehrhart}. They are a family of polytopes with a simple and explicit construction, yet they exhibit a
\begin{wrapfigure}[15]{r}{0.38\textwidth}%
	\centering
    \includegraphics[width=0.35\textwidth]{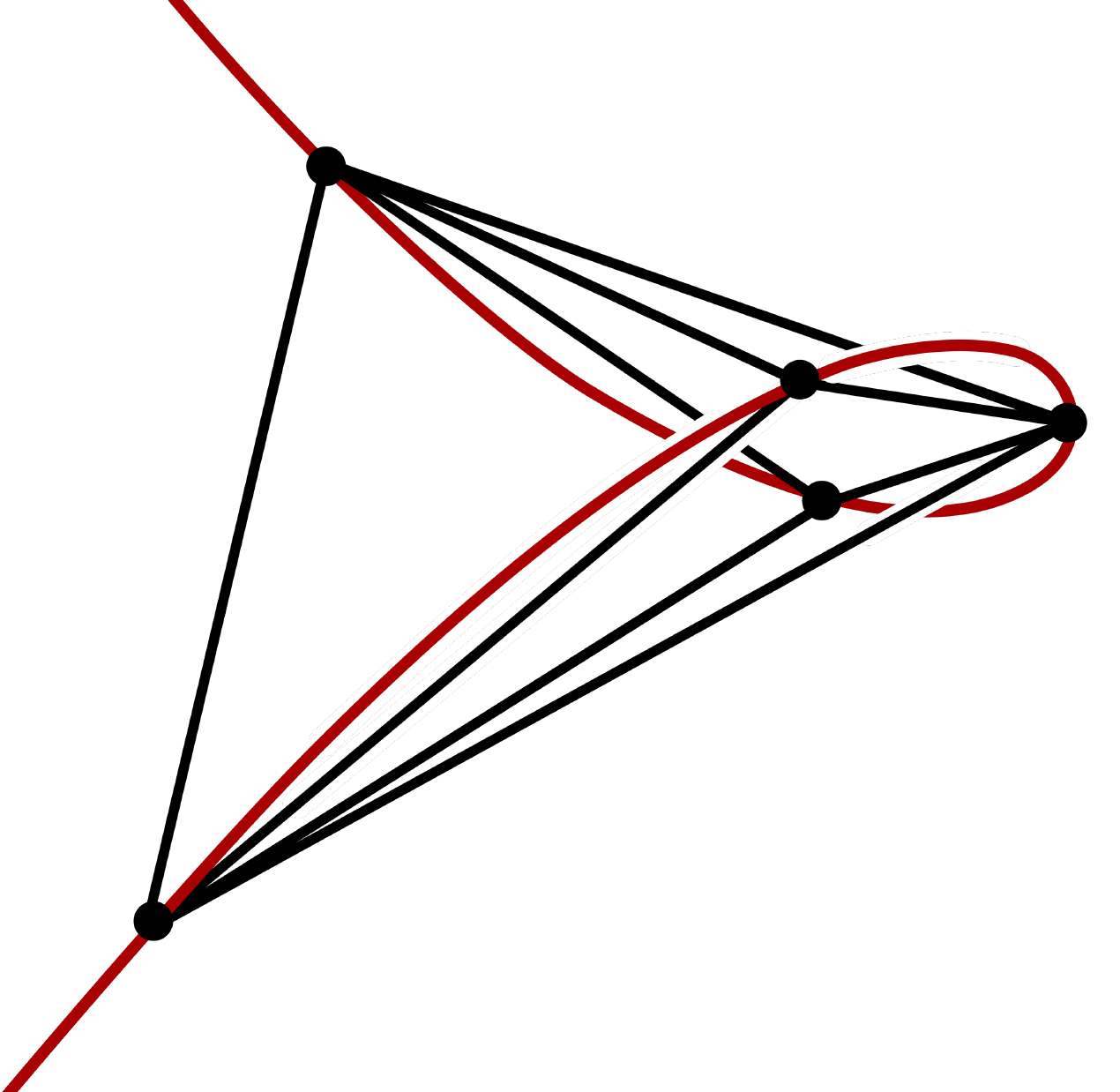}
   \caption{cyclic polytope with 5 vertices on the curve $t\mapsto(t,t^2,t^3)$}
  \label{fig:polytope}
\end{wrapfigure}
surprisingly rich and complex combinatorial structure, a beautiful property that is not easily found in other types of polytopes.
Importantly, their value solely derives from nice properties of the momentum curve, which, in fact, is the rational normal curve, a very distinguished projective curve of algebraic geometry, in a suitable affine chart.
An organic extension of cyclic polytopes would be to consider the convex hull of finitely many points lying on finitely many well-behaving projective curves in an arbitrary affine chart in a space. Such a collection of projective curves is facilitated by a factorization structure, and such a convex hull is called a polytope compatible with the factorization structure.
 
A \textit{factorization structure} of dimension $m$ is defined as a linear inclusion $\varphi:\mathfrak{h} \to V_1\otimes \cdots \otimes V_m$ of an $(m+1)$-dimensional vector space into the tensor product of $m$ 2-dimensional vector spaces satisfying
\begin{align}\label{introfs}
\dim
\left(
\varphi(\mathfrak{h})
\cap
V_1 \otimes \cdots \otimes V_{j-1} \otimes \ell \otimes V_{j+1} \otimes \cdots \otimes V_m
\right)
=
1
\end{align}
for a generic 1-dimensional subspace $\ell \subset V_j$ and any $j$. Consequently, because $\varphi$ is injective, when varying 1-dimensional spaces $\ell\subset V_j$, i.e., points of the projective line $\mathbb{P}(V_j)$, the $\varphi$-preimage of intersections \eqref{introfs} gives 1-dimensional subspaces in $\mathfrak{h}$, i.e., points of $\mathbb{P}(\mathfrak{h})$, and thus for each index $j$ we obtain a projective curve $\mathbb{P}(V_j) \to \mathbb{P}(\mathfrak{h})$, called a factorization curve.
An example of a factorization structure is the Veronese factorization structure, defined as the canonical inclusion $\varphi: S^mW \to  W^{\otimes m}$ of symmetric tensors on the 2-dimensional space $W$. All of its factorization curves coincide and are the rational normal curve,
\begin{align}
\mathbb{P}(W) &\to \mathbb{P}(S^mW)\\
\ell &\mapsto \ell \otimes \cdots \otimes \ell, \nonumber
\end{align}
which, as mentioned previously, is, in a suitable affine chart, the momentum curve. Factorization structures manifest in toric differential geometry and discrete geometry through polytopes and polyhedral cones, and are surprisingly closely related to canonical metrics in toric Kähler geometry and extremal structures in toric Sasaki and CR geometries.\par

In the context of discrete geometry, the exceptional nature of cyclic polytopes arises from the existence of a simple characterisation of hyperplanes adjacent to a cyclic polytope in terms of its vertices, called the Gale's evenness condition \cite{gale1963neighborly}. While significant, cyclic polytopes represent just a small segment of the vast landscape of polytopes and cones that are compatible with factorization structures. Remarkably, the elements within this broader class strike a balance between simplicity and complexity, mirroring the appealing characteristics of cyclic polytopes. Moreover, all these elements are equipped with a generalized Gale's evenness condition, which can be viewed as the very nature of factorization structures. Factorization structures govern the geometry of compatible polytopes and cones, providing an elegant and practical framework ideal for explicit computations. This framework offers a clear perspective on their duals, inherently involves their projective transformations, and provides an explicit description of their faces. Moreover, factorization structures offer a natural generalisation of Vandermonde identities \cite{apostolov2004hamiltonian}, which are used to grasp the interplay of a polytope or a cone with a lattice.
Notably, the efficiency of factorization structures invites attempts at computations that would otherwise be considered intricate or challenging.\par

In differential geometry, one of the main research directions is to seek canonical geometric structures, often arising as extremal points of a (energy) functional, such as the heavily studied extremal Kähler metrics \cite{calabi1982extremal, calabi1985extremal}. Finding non-trivial explicit examples of these metrics is a challenging task, and several were provided ad hoc using toric geometry (\cite{abreu1998kahler,apostolov2021cr,apostolov2003geometry,apostolov2016ambitoric,apostolov2004hamiltonian,jsg/1310388900,simanca1991kahler,simanca1992note,article}). Factorization structures offer a unifying framework that not only encompasses all known explicit extremal toric Kähler metrics but also provides new examples \cite{pucek2022extremal}. In addition, they determine extensive families of explicit toric Sasaki and Kähler geometries amenable to computations, thereby facilitating the search for these explicit canonical geometric structures. The transition between discrete and differential geometry is mediated by the momentum map of a toric geometry whose image is a Delzant polytope.\\

The main contribution of this article is not a collection of isolated results but rather the development of a cohesive and innovative theoretical framework. This framework, centred on factorization structures, their techniques, and their broad applicability, establishes a foundation for exploring new possibilities in both discrete and differential geometry. Its significance lies in its unifying power, providing tools that extend beyond ad hoc principles to systematically address complex problems.\\

In \Cref{s3}, we define cones and polytopes compatible with a given factorization structure using its canonically associated curves, whose properties, studied in \Cref{s2}, are reflected and essential in constructing these. For example, we prove in \Cref{Ver cone is simplicial} that cones compatible with the Veronese factorization structure are cones over simplicial polytopes, and that associated compatible polytopes are simple. The theory of quotient factorization structures from \Cref{quotient subsection} allows us to elegantly and geometrically describe subspaces where faces of compatible cones, polytopes, and of their duals can lie (\Cref{faces}). This, together with the generalised Gale's condition (\Cref{general Gale}), culminates in a non-trivial and powerful result: explicit description of faces. We derive generalised Vandermonde identities \eqref{V-id mod beta} and \eqref{1 identity} and use them in \Cref{s32} to find examples of rational Delzant polytopes.\par
Examples of factorization structures given in this paper are of the Segre-Veronese type (\Cref{SV def}). In particular, such a structure is determined by finitely many constant tensors fulfilling non-trivial equations, and so finding them all explicitly is a challenging task. Instead, we define a product of arbitrary factorization structures (\Cref{prod}), and use it to generate vast classes of explicit Segre-Veronese factorization structures in \Cref{s1}.
We use all the structure theory of factorization structures to characterise decomposable Segre-Veronese factorization structures, a class where defining tensors are decomposable, as iterative products of Veronese factorization structures in \Cref{classification}.
This is the first step towards the classification of factorization structures.
\Cref{main thm} characterises Segre-Veronese factorization structures as exactly those factorization structures whose factorization curves are decomposable. Such a curve can be viewed as an embedded rational normal curve (\Cref{dec char}). The open \Cref{question} asks about the existence of an indecomposable factorization curve.

\tocless{\section*{Acknowledgements}}
The author warmly thanks Marie-Charlotte Brandenburg for invaluable discussions on cones and polytopes, David M.J. Calderbank for dedicating his time and contributing to the initial exploration of this topic, and  Hendrik Süß for sharing expertise in algebraic geometry.

\addtocontents{toc}{\protect{\pdfbookmark[1]{\contentsname}{toc}}}
\renewcommand\contentsname{\vspace{.5cm} Contents}

{\hypersetup{hidelinks}
\tableofcontents}

\section{Factorization structures}\label{s1}

In this article, $V_1,\ldots,V_m$, $m\geq2$, denote real/complex 2-dimensional vector spaces. We define
\begin{align}
	V=\bigotimes_{r=1}^m V_r
	\hspace{1cm}
	\text{and}
	\hspace{1cm}
	\hat{V}_j=\bigotimes_{\substack{r=1\\r\neq j}}^m V_r,
\end{align}
and denote their duals by $V^*$ and $\hat{V}_j^*$, respectively. For a fixed $j\in\{1,\ldots,m\}$ and any 1-dimensional subspace $\ell\subset V_j$ we consider contractions $\rho_{j,v}:V^*\to\hat{V}_j^*$ parametrised by any non-zero $v\in\ell$ and defined on decomposable tensors via
\begin{align}\label{contraction defn}
v_1 \otimes \cdots \otimes v_m
\mapsto
\langle v, v_j \rangle \hspace{.1cm} v_1 \otimes \cdots \otimes v_{j-1} \otimes v_{j+1} \otimes \cdots \otimes v_m,
\end{align}
where $\langle \hspace{.1cm} , \rangle$ is the standard contraction on $V_j\otimes V_j^*$. The kernel of such a contraction is the annihilator of
\begin{align}\label{intro ref}
\Sigma_{j,\ell}:=
V_1\otimes\cdots\otimes V_{j-1}\otimes\ell\otimes V_{j+1}\otimes\cdots\otimes V_m
\end{align}
in $V^*$, which is be denoted by $\Sigma^0_{j,\ell}$, and for a fixed $\ell$ does not depend on $v$.\par
The projective space $\mathbb{P}(W)$ is viewed as the set of 1-dimensional subspaces in the vector space $W$ equipped with the Zariski topology. Often, we identify $\ell\in\mathbb{P}(W)$ with the corresponding 1-dimensional subspace of $W$, and denote the span of a non-zero vector $w\in W$ by $\langle w \rangle$. We say a condition holds for a \textit{generic} point or \textit{generically} if there exists an open non-empty subset $U\subset\mathbb{P}(W)$ such that the condition holds at each point of $U$.\\

Having the notation established we are ready to define the main object of study in this article, a factorization structure.

\begin{defn}\label[defn]{fs def}
Let $m$ be a positive integer. An injective linear map $\varphi:\mathfrak{h}\to V^*$ of a real/complex $(m+1)$-dimensional vector space $\mathfrak{h}$ into real/complex $V^*$ is called a \textit{factorization structure} of dimension $m$ if
\begin{align}\label{wfs def condition}
\dim \left( \varphi(\mathfrak{h}) \cap \Sigma_{j,\ell}^0 \right) = 1
\end{align}
holds for every $j\in\{1,\ldots,m\}$ and generic $\ell\in\mathbb{P}(V_j)$.
An isomorphism of factorization structures is the commutative diagram
\begin{center}
\begin{tikzcd}
	\mathfrak{h}_1 \arrow[d, "\varphi_1"'] \arrow[rr, "\Phi"]                      &  & \mathfrak{h}_2 \arrow[d, "\varphi_2"] \\
	V_1^*\otimes\cdots\otimes V_m^* \arrow[rr, "(\phi_1\otimes\cdots\otimes\phi_m)\sigma"] &  & W_1^*\otimes\cdots\otimes W_m^*      
\end{tikzcd},
\end{center}
where $\Phi$ and $\phi_j:V_{\sigma(j)}^*\to W_j^*$ are linear isomorphisms for all $j\in\{1,\ldots,m\}$, and $\sigma$ is a permutation of $\{1,\ldots,m\}$ viewed as the braiding map
$V_1^*\otimes\cdots\otimes V_m^*\to V_{\sigma(1)}^*\otimes\cdots\otimes V_{\sigma(m)}^*$.
\end{defn}

\begin{rem}\label[rem]{the remark}
Setting $\sigma=\text{id}$ and $\phi_j=\text{id}$, $j=1,\ldots,m$, shows that any two factorization structures with the same images are undistinguishable up to a choice of $\Phi$, which does not play a role in the defining condition \eqref{wfs def condition}. Thus, a factorization structure $\varphi$ can be identified with the subspace $\varphi(\mathfrak{h}) \subset V^*$.
\end{rem}

\begin{rem}
All results of this section hold for real and complex factorization structures. Therefore, no distinction between these is made.
\end{rem}

\subsection{Factorization structures of dimension 2}\label{fs of dim 2}
To begin the study of factorization structures we note that there is only one isomorphism class in dimension 1, and focus on the full understanding of the first non-trivial case, factorization structures of dimension 2. Although factorization structures were previously defined in \cite{apostolov2015ambitoric}, inconsistencies between the definition, notion of isomorphism, and their classification in 2 dimensions, led to \Cref{fs def}. The new definition results in the same 2-dimensional classification as in \cite{apostolov2015ambitoric}: up to isomorphism, it consits of two factorization structure, 2-dimensional Segre and Veronese factorization structures.

To restate this classification in a simplified manner and full detail we note
\begin{lemma}\label[lemma]{2d fs lemma}
An inclusion $\varphi:\mathfrak{h}\to V_1^*\otimes V_2^*$ of a 3-dimensional vector space into the tensor product of two 2-dimensional vector spaces is a 2-dimensional factorization structure.
\end{lemma}
\begin{proof}
Clearly,
\begin{align}\label{m=2 inclusion}
2\geq
\dim \left( \varphi(\mathfrak{h})\cap\ell^0\otimes V_2^* \right)
=
\dim \left( \varphi(\mathfrak{h})\cap\Sigma_{1,\ell}^0 \right)
\geq 1
\end{align}
holds for any $\ell\in\mathbb{P}(V_1)$, and similarly for intersections $\varphi(\mathfrak{h})\cap\Sigma_{2,\ell}^0$ with $\ell\in\mathbb{P}(V_2)$. Note that if \eqref{m=2 inclusion} were 2-dimensional in two distinct points $\ell, \bar{\ell}\in\mathbb{P}(V_1)$, then two 2-dimensional subspaces $\ell^0\otimes V_2^*$ and $\bar{\ell}^0\otimes V_2^*$, whose intersection is trivial, would lie in the 3-dimensional space $\varphi(\mathfrak{h})$. Therefore, the intersection $\varphi(\mathfrak{h}) \cap \Sigma_{1,\ell}^0$ is 2-dimensional at most at one point, hence is generically 1-dimensional, i.e., $\varphi$ is a factorization structure.
\end{proof}
We found that 2-dimensional factorization structures are merely linear inclusions $\varphi:\mathfrak{h}\to V_1^*\otimes V_2^*$, which we now classify up to isomorphism of factorization structures via the annihilator $\varphi(\mathfrak{h})^0\leq V_1\otimes V_2$. \par
If $\varphi(\mathfrak{h})^0$ is decomposable in $V_1\otimes V_2$, i.e., $\varphi(\mathfrak{h})^0=\gamma_1\otimes\gamma_2$ for some 1-dimensional subspaces $\gamma_j\subset V_j$, then the corresponding factorization structure, called \textit{Segre}, is of the form
\begin{align}\label{Segre 2}
\varphi(\mathfrak{h})=V_1^*\otimes\gamma_2^0+\gamma_1^0\otimes V_2^*
\hookrightarrow
V_1^* \otimes V_2^*,
\end{align}
where $\gamma_j^0\subset V_j^*$ is the annihilator of $\gamma_j$ (see \Cref{the remark}).
One easily observes that if $\tilde{\varphi}$ is another inclusion so that $\tilde{\varphi}(\mathfrak{h})^0$ is decomposable, then $\varphi$ and $\tilde{\varphi}$ are isomorphic as factorization structures. \par
Suppose now that $\varphi(\mathfrak{h})^0$ is indecomposable.
Then, any non-zero $\chi\in\varphi(\mathfrak{h})^0$, viewed as a map $\chi: V_1^* \to V_2$, is invertible, since in a basis it is represented by a 2-by-2 matrix with non-zero determinant due to the indecomposability.
The composition of isomorphisms $\text{id}\otimes\chi^{-1}:V_1\otimes V_2\to V_1\otimes V_1^*$, which maps $\chi$ to the identity automorphism of $V_1^*$, and $\omega\otimes\text{id}:V_1\otimes V_1^*\to V_1^*\otimes V_1^*$, where $\omega$ is a fixed area form on $V_1$, maps $\chi$ on an element of $\bigwedge^2V_1^*$.
Therefore, $\langle\chi\rangle$ is mapped onto $\bigwedge^2V_1^*$ under the isomorphism $\omega\otimes\chi^{-1}$ yielding the commutative diagram\\

\begin{equation}\label{diagram}
\begin{tikzcd}
            & 0 \arrow[d]                                                                            & 0 \arrow[d]                            &   \\
0 \arrow[r] & \langle\chi\rangle \arrow[r, "\omega\otimes\chi^{-1}|_{\langle\chi\rangle}"] \arrow[d] & \bigwedge^2V_1^* \arrow[r] \arrow[d]   & 0 \\
0 \arrow[r] & V_1\otimes V_2 \arrow[d, "\varphi^T"'] \arrow[r, "\omega\otimes\chi^{-1}"]             & V_1^*\otimes V_1^* \arrow[d] \arrow[r] & 0 \\
0 \arrow[r] & \mathfrak{h}^* \arrow[d] \arrow[r]                                                     & S^2V_1^* \arrow[d] \arrow[r]           & 0 \\
            & 0                                                                                      & 0                                      &  
\end{tikzcd}.
\end{equation}
Taking in account \Cref{the remark} and dualising \eqref{diagram} shows that the \textit{Veronese} factorization structure
\begin{align}\label{Veronese 2}
S^2V_1\hookrightarrow V_1\otimes V_1
\end{align}
is isomorphic in the sense of factorization structures to $\varphi:\mathfrak{h}\to V_1^*\otimes V_2^*$.\par
We note that Segre and Veronese factorization structures are not isomorphic which can be seen from the decomposability of $\varphi(\mathfrak{h})^0$ in respective cases. This classifies 2-dimensional factorization structures.

\subsection{Segre-Veronese factorization structure}
We describe a large class of factorization structures, called Segre-Veronese, which generalise Segre and Veronese factorization structures discussed in \Cref{fs of dim 2}.\par
For $i\in\{1,\ldots,m\}$ we say that the term $a_i$ in $a_1\otimes\cdots\otimes a_m$ is in the $i$th \textit{slot}. If a partition of $m$ is given, $m=d_1+\cdots+d_k$, $d_j\geq1$, slots group into $k$ groups with the $j$th group containing $d_j$ slots, $j\in\{1,\ldots,k\}$. Slots belonging to the $j$th group are referred to as \textit{grouped $j$-slots}. In fact, positions in such a tensor product can be labelled by pairs $(j,r)$, where $j\in\{1,\ldots,k\}$ and $r\in\{1,\ldots,d_j\}$. For a partition of $m$ as above and a fixed $j\in\{1,\ldots,k\}$ we define the operator
$$ins_j:
(W_j^*)^{\otimes d_j}\otimes\bigotimes_{\substack{i=1\\i\neq j}}^k (W_i^*)^{\otimes d_i}
\to
\bigotimes_{i=1}^k (W_i^*)^{\otimes d_i}$$
which acts on decomposable tensors by
\begin{align*}
\left(w_j^1\otimes\cdots\otimes w_j^{d_j}\right)
\otimes
\bigotimes_{\substack{i=1\\i\neq j}}^k
\left(w_i^1\otimes\cdots\otimes w_i^{d_i}\right)
\mapsto
\bigotimes_{i=1}^k
\left(w_i^1\otimes\cdots\otimes w_i^{d_i}\right),
\end{align*}
where $W_j$, $j=1,\ldots,k$, are vector spaces. Partitions $m=d_1+\cdots+d_p$ and $m=e_1+\cdots+e_q$ are considered to be the same if $\{d_1,\ldots,d_p\}=\{e_1,\ldots,e_q\}$, and distinct if they are not the same.

\begin{defn}\label[defn]{SV def}
For $d_1,\ldots,d_k$ a partition of an integer $m\geq2$ and $W_r$, $r=1,\ldots,k$, 2-dimensional vector spaces, let 
$\Gamma_j
\subset
\bigotimes_{r=1,r\neq j}^k(W_r^*)^{\otimes d_r}$,
$j\in\{1,\ldots,k\}$, be 1-dimensional subspaces such that
\begin{align}\label{SV image}
\sum_{j=1}^{k}
ins_j
\left(
S^{d_j}W_j^*\otimes\Gamma_j
\right)
\end{align}
has dimension $m+1$, where $S^{d_j}W_j^*\subset(W_j^*)^{\otimes d_j}$ is viewed as the subspace of symmetric tensors.
Define vector spaces $V_1,\ldots,V_m$ by
\begin{gather}\label{spaces V_j}
V_{d_1 + \cdots + d_{j-1} + 1} =
V_{d_1 + \cdots + d_{j-1} + 2} =
\cdots = V_{d_1 + \cdots + d_{j-1} + d_j} =
W_j,
\hspace{.5cm}
j=1,\ldots,k,
\end{gather}
where $d_0$ is defined to be zero.
The \textit{standard Segre-Veronese factorization structure} $\varphi: \mathfrak{h} \to V^*$ is defined to be such that $\mathfrak{h}$ is the $(m+1)$-dimensional space \eqref{SV image}, $V^* = \otimes_{j=1}^m V_j^*$ where $V_j$ is defined by \eqref{spaces V_j}, and $\varphi$ is the canonical inclusion of $\mathfrak{h}$ to $V^*$, i.e., it is
\begin{align}\label{SV inclusion}
\sum_{j=1}^{k}
ins_j
\left(
	S^{d_j}W_j^*\otimes \Gamma_j
\right)
\hookrightarrow
\bigotimes_{j=1}^k(W_j^*)^{\otimes d_j}.
\end{align}
Factorization structures corresponding to trivial partitions,
\begin{gather}
\sum_{j=1}^m ins_j \left( W_j^* \otimes \Gamma_j \right)
\hookrightarrow
\bigotimes_{j=1}^m W_j^*
\end{gather}
for $m=1+\cdots+1$, and
\begin{gather}
S^mW^* \hookrightarrow (W^*)^{\otimes m}
\end{gather}
for $m=m$, are respectively called \textit{Segre} and \textit{Veronese}.
An element of the isomorphism class of a standard Segre-Veronese factorization structure is referred to as a Segre-Veronese factorization structure.
\end{defn}
We frequently refer to the 1-dimensional spaces $\Gamma_j$ as \textit{defining tensors} of the standard Segre-Veronese factorization structure, since one does need them to define a given Segre-Veronese factorization structure, and since each $\Gamma_j$ is a linear span of a tensor.
\begin{rem}
Note that Segre and Veronese factorization structures recover \eqref{Segre 2} and \eqref{Veronese 2} when $m=2$.\par
To verify that \eqref{SV inclusion} defines a factorization structure we observe that for $i\in\{1,\ldots,k\}$ and generic $\ell\in\mathbb{P}(W_i)$ we have
\begin{align}\label{standard curves}
\varphi(\mathfrak{h})
\cap
\Sigma_{d_1+\cdots+d_{i-1}+q,\ell}^0
=
ins_i
\left(
	(\ell^0)^{\otimes d_i}
	\otimes
	\Gamma_i
\right),
\end{align}
where $\varphi(\mathfrak{h})$ is \eqref{SV image}, $q\in\{1,\ldots,d_i\}$ and $d_0$ is defined to be $0$. Note that there are at most finitely many $\ell\in\mathbb{P}(V_i)$ for which the dimension of the intersection in \eqref{standard curves} could be strictly larger than one, and, loosely speaking, this occurs when defining tensors $\Gamma_j$, $j\neq i$, decompose at the $i$th slot.
\end{rem}

Determining in general which choices of $\Gamma_j$, $j=1,\ldots,k$, give rise to a factorization structure, i.e., make \eqref{SV image} an $(m+1)$-dimensional vector space, is a challenging task.
Instead, in the following we exemplify particular choices which effortlessly guarantee the correct dimension. \par
\begin{example}\label[example]{k=2 example}
We examine the standard Segre-Veronese factorization structure for $k=2$. To this end, let $m=d_1+d_2$ be a partition, and $\Gamma_1\subset (W_2^*)^{\otimes d_2}$ and $\Gamma_2\subset (W_1^*)^{\otimes d_1}$ be 1-dimensional subspaces.
Observe that the dimension of the image of
\begin{align}\label{k=2}
S^{d_1}W_1^*\otimes\Gamma_1
+
\Gamma_2\otimes S^{d_2}W_2^*
\hookrightarrow
(W_1^*)^{\otimes d_1}\otimes (W_2^*)^{\otimes d_2}
\end{align}
is $m+1$ if and only if $\Gamma_1\subset S^{d_2}W_2^*$ and $\Gamma_2\subset S^{d_1}W_1^*$, which completely characterises choices of $\Gamma_1$ and $\Gamma_2$ leading to a factorization structure.
\end{example}

\begin{example}\label[example]{one intersection example}
For a partition $m=d_1+\cdots+d_k$ and 1-dimensional subspaces $ a^r \subset W_r^*$, $r=1,\ldots,k$, we define \textit{the product Segre-Veronese factorization structure} as the standard Segre-Veronese factorization structure such that
\begin{align}\label{product tensors}
\Gamma_j=
\bigotimes_{\substack{r=1\\r\neq j}}^k (a^r)^{\otimes d_r},
\hspace{.5cm}
j=1,\ldots,k.
\end{align}
These data ensure that any two summands of \eqref{SV image} intersect in
$\bigotimes_{r=1}^k(a^r)^{\otimes d_r}$, which implies that the dimension of \eqref{SV image} is $m+1$. Therefore, the product Segre-Veronese factorization structure is indeed a factorization structure.
The product Segre-Veronese factorization structure with partition $m=1+\cdots+1$ is called the \textit{product Segre factorization structure}.
\end{example}

\begin{example}\label[example]{decomposable example}
Motivated by the example above, a natural step is to find when decomposable $\Gamma_j$ determine a factorization structure, i.e. give rise to the correct dimension of \eqref{SV image}.
\textit{The decomposable Segre-Veronese factorization structure} is defined as the standard Segre-Veronese factorization structure such that $\Gamma_j$ are decomposable, i.e., 
$\Gamma_j
=
\bigotimes_{\substack{r=1\\r\neq j}}^k
\bigotimes_{p=1}^{d_r}
a_j^{r,p}$
for some 1-dimensional subspaces $a_j^{r,p} \subset W_r^*$, $j=1,\ldots,k$. Below, in \Cref{decomposable SV tensors}, we show that if such decomposable $\Gamma_j$, $j=1,\ldots,k$, determine a factorization structure, then it must be that
$a_j^{r,1}=
\cdots=
a_j^{r,d_r} =: a_j^r$,
and hence
\begin{align}\label{dec tensors}
\Gamma_j=
\bigotimes_{\substack{r=1\\r\neq j}}^k (a^r_j)^{\otimes d_r}
\end{align}
necessarily.
However, it is still not plain to see which choices of $a_j^r$ lead to a factorization structure. We characterise these in \Cref{classification}.
\end{example}

\begin{rem}
Observe that the isomorphism class of a fixed standard Segre-Veronese factorization structure may contain multiple standard Segre-Veronese factorization structures; e.g., apply an isomorphism which permutes grouped slots. In the decomposable case \eqref{dec tensors}, any choice of $g_r\in\text{GL}(W_r^*)$, $r=1,\ldots,k$, yields an isomorphic standard Segre-Veronese factorization structure via the operator $(g_1)^{\otimes d_1} \otimes \cdots \otimes (g_k)^{\otimes d_k}$.
\end{rem}

\subsection{Products and decomposable elements}\label{products}
In general, a complete description of 1-dimensional spaces $\Gamma_j$ determining a Segre-Veronese factorization structure is a complex task.
However, we leverage the concept of a product of factorization structures to generate extensive families of hands-on examples.
Specifically, we show that products of two factorization structures are parametrised by the points in the image of the Segre embedding of these two structures.
As it turns out in \Cref{classification}, iterated products of Veronese factorization structures completely characterise decomposable Segre-Veronese factorization structures.
We finish this subsection by presenting an example of a Segre-Veronese factorization structure whose all defining tensors are indecomposable.

\begin{defn}\label[defn]{prod}
Let $\chi:\mathfrak{g}\to W_1^*\otimes\cdots\otimes W_n^*$ and $\varphi:\mathfrak{h}\to V_1^*\otimes\cdots\otimes V_m^*$ be two factorization structures and $T\subset\chi(\mathfrak{g})$ and $S\subset\varphi(\mathfrak{h})$ any two 1-dimensional subspaces. We define the \textit{product} of $\varphi$ and $\chi$ to be the $(n+m)$-dimensional factorization structure given by the canonical inclusion
\begin{align}\label{product}
\varphi(\mathfrak{h})
\otimes
T+
S
\otimes
\chi(\mathfrak{g})
\hookrightarrow
V_1^*\otimes\cdots\otimes V_m^*\otimes W_1^*\otimes\cdots\otimes W_n^*.
\end{align}
\end{defn}

Examples of product include the 2-dimensional Segre factorization structure viewed as a product of two 1-dimensional factorization structures, Segre-Veronese factorization structure with $k=2$ from \Cref{k=2 example} viewed as a product of two Veronese factorization structures, and the product Segre-Veronese \eqref{product tensors} from \Cref{one intersection example}.
In fact, the latter is a product in multiple ways as we can see in the following

\begin{example}\label[example]{full-product ex}
Let $I:=\{1,\ldots,k_0\}\subset\{1,\ldots,k\}$, $1\leq k_0<k$, and let $I^c$ be the complement of $I$. We can rewrite the product Segre-Veronese factorization structure from \Cref{one intersection example} as
\begin{gather}\nonumber
\Bigg(
	\sum_{j\in I}
	ins_j
	\bigg(
		S^{d_j}W_j^*
		\otimes
		\bigotimes_{\substack{r\in I\\r\neq j}} (a^r)^{\otimes d_r}
	\bigg)
\Bigg)
\otimes
\bigotimes_{r\in I^c}
(a^r)^{\otimes d_r}+\\
+
\bigotimes_{r\in I}
(a^r)^{\otimes d_r}
\otimes
\Bigg(
	\sum_{j\in I^c}
	ins_j
	\bigg(
	S^{d_j}W_j^*
	\otimes
	\bigotimes_{\substack{r\in I^c\\r\neq j}} (a^r)^{\otimes d_r}
	\bigg)
\Bigg)\label{product SV is product fs}
\hookrightarrow
\bigotimes_{j\in I} (W_j^*)^{\otimes d_j} \otimes \bigotimes_{j\in I^c} (W_j^*)^{\otimes d_j},
\end{gather}
rendering it as the product of the (product Segre-Veronese) factorization structure
\begin{gather}
\sum_{j\in I}
ins_j
\bigg( S^{d_j}W_j^* \otimes	\bigotimes_{\substack{r\in I\\r\neq j}} (a^r)^{\otimes d_r} \bigg)
\hookrightarrow
\bigotimes_{j\in I} (W_j^*)^{\otimes d_j}
\end{gather}
and the (product Segre-Veronese) factorization structure
\begin{gather}
\sum_{j\in I^c}
ins_j
\bigg( S^{d_j}W_j^* \otimes \bigotimes_{\substack{r\in I^c\\r\neq j}} (a^r)^{\otimes d_r} \bigg)
\hookrightarrow
\bigotimes_{j\in I} (W_j^*)^{\otimes d_j}
\end{gather}
with $T=\bigotimes_{r\in I^c} (a^r)^{\otimes d_r}$ and $S=\bigotimes_{r\in I} (a^r)^{\otimes d_r}$.
Clearly, such a product exists for any non-trivial $I\subset\{1,\ldots,k\}$.
\end{example}

Now we illustrate how products can be used to construct new examples of factorization structures.
We fix the Segre-Veronese factorization structure \eqref{k=2} corresponding to the partition $d_1+d_2$ from \Cref{k=2 example} and the Veronese factorization structure $S^{d_3}W_3^* \hookrightarrow (W_3^*)^{\otimes d_3}$. To form a product of these two factorization structures, we choose 1-dimensional spaces $\Gamma \subset S^{d_3}W_3^*$ and $\Gamma_3$ lying in the image of \eqref{k=2}, and with respect to these choices we obtain the product
\begin{align}\label{k=3}
S^{d_1}W_1^* \otimes \Gamma_1\otimes\Gamma
+
\Gamma_2 \otimes S^{d_2}W_2^* \otimes \Gamma
+
\Gamma_3 \otimes S^{d_3}W_3^*
\hookrightarrow
\bigotimes_{j=1}^3 (W_j^*)^{\otimes d_j},
\end{align}
which is a factorization structure of the dimension $d_1+d_2+d_3$ and belongs again to the class of a Segre-Veronese factorization structures.
We could continue further and make a product of the factorization structure \eqref{k=3} with another Veronese factorization structure $S^{d_4}W_4^* \hookrightarrow (W_4^*)^{\otimes d_4}$, or form a product of two factorization structures of type \eqref{k=2} to obtain a Segre-Veronese factorization structure corresponding to a partition of length 4, i.e. $k=4$. And so on. \\

One could speculate that factorization structures, or at least Segre factorization structures, can be built up via products from atomic pieces. 
Notably, defining tensors of a product of Segre-Veronese factorization structures always decompose across slots belonging to original factors: in \eqref{product}, defining tensors are the ones of $\varphi$ tensored-from-right with $T$ and the ones of $\chi$ tensored-from-left with $S$.
The following example demonstrates that the building blocks of (Segre-Veronese) factorization structures need not be simple.

\begin{example}\label[example]{indec Segre}
We conclude this section with an example of 3-dimensional Segre factorization structure whose all defining tensors are indecomposable. Observe that only in dimension 3, the annihilator $\mathfrak{h}^0\xhookrightarrow{} V$ of a factorization structure $\mathfrak{h}\xhookrightarrow{} V^*$ has the right dimension for being a factorization structure. The Veronese $S^3W^*\xhookrightarrow{} (W^*)^{\otimes 3}$ has the annihilator
\begin{align}
W\otimes\bigwedge^2W +
ins_2\left( W\otimes\bigwedge^2W \right) +
\bigwedge^2W\otimes W
\xhookrightarrow{}
W\otimes W\otimes W,
\end{align}
a factorization structure with indecomposable defining tensors.
\end{example}

For the later use we study particular elements in a product of factorization structures.

\begin{lemma}\label[lemma]{tensors split}
Let
$\varphi(\mathfrak{h})\otimes T+
S\otimes\chi(\mathfrak{g})
\hookrightarrow
V_1^*\otimes\cdots\otimes V_m^*\otimes W_1^*\otimes\cdots\otimes W_n^*$
be a product of factorization structures.
Then
\begin{align}
I
\otimes
K  \subset
\varphi(\mathfrak{h})\otimes T+
S\otimes\chi(\mathfrak{g})
\end{align}
for some 1-dimensional subspaces $ I \subset V_1^*\otimes\cdots\otimes V_m^*$ and $ K  \subset W_1^*\otimes\cdots\otimes W_n^*$ if and only if
\begin{align}\label{split elements in product fs}
\bigg[
	I=
	S
	\text{ and }
	K \subset \chi(\mathfrak{g})
\bigg]
\hspace{.4cm}\text{or}\hspace{.4cm}
\bigg[
	K=
	T
	\text{ and }
	I \subset \varphi(\mathfrak{h})
\bigg]
\end{align}
\end{lemma}
\begin{proof}
The 'if' part is obvious. To prove the 'only if' part of the statement, let $s\in S, t\in T, \iota\in I, \kappa\in K$ be non-zero vectors. Since any element of the product factorization structure can be written as
$
\tau_1\otimes t+
s \otimes\tau_2
$
for some  $\tau_1\in\varphi(\mathfrak{h})$ and $\tau_2\in\chi(\mathfrak{g})$, we need to solve
\begin{align}\label{split elements equation}
\tau_1\otimes t+
s\otimes\tau_2=
\iota\otimes\kappa
\end{align}
for $\tau_1$ and $\tau_2$. We suppose $T \neq K$ and $S \neq I$ as the complementary situation easily gives the claim. We proceed by assuming that $\tau_1$ and $\tau_2$ solve \eqref{split elements equation}, and analyse $\tau_1$ in this equation. Note that if $\tau_1=0$, then the equation reduces to the case we excluded. Now we consider two cases; $\tau_1$ is either in the span of $S$ and $I$, or it is not. In the former case, we can write $\tau_1=as+b\iota$ for some scalars $a,b$, which transforms \eqref{split elements equation} into
\begin{align}
s \otimes (at + \tau_2) =
\iota \otimes (\kappa - bt),
\end{align}
and is true if only if $I = S$ and $K \subset \chi(\mathfrak{g})$, hence contradicting our assumptions. In the latter case $\langle \tau_1 \rangle, S$ and $I$ are linearly independent directions. By completing $\tau_1, s$ and $\iota$ into a basis of $V_1^* \otimes \cdots \otimes V_m^*$ and contracting \eqref{split elements equation} with the dual vector of $\tau_1$, we find that $t=0$, which is a contradiction. Thus, a solution exists if and only if \eqref{split elements in product fs} holds.
\end{proof}

\subsection{Motivation for definition of factorization structures rooted in discrete geometry}\label{discr geom}

The definition of 2-dimensional factorization structures was sought in the search for compactifications of ambitoric geometries \cite{apostolov2016ambitoric,apostolov2015ambitoric} by formalising the workings of hyperplane sections of the 2-dimensional Segre embedding. To provide both a presentation and motivation for this definition from a perspective rooted in discrete geometry, we turn our attention to cyclic polytopes. These were introduced by Gale \cite{gale1963neighborly} and are now a standard part of discrete geometry and combinatorics \cite{grunbaum1967convex,ziegler1993lectures}, however, their presentation could benefit from more context. In the rest of this subsection we merely outline the theory of cyclic polytopes from the viewpoint assumed later in this article.
For a detailed account and further motivation for studying applications of factorization structures in discrete geometry, see \cite{brandenburg2024}. \\

The momentum curve, $t \mapsto (t,t^2,\ldots,t^m)$, is the rational normal curve, i.e., the Veronese embedding
\begin{align}\label{rnc}
\mathbb{P}(W) &\to \mathbb{P}(S^mW^*)\\ \nonumber
\ell &\mapsto \ell^0 \otimes \cdots \otimes \ell^0,
\end{align}
in a suitable affine chart, where $W$ is a 2-dimensional vector space, $S^mW^*$ is the $(m+1)$-dimensional space of symmetric tensors on the dual of $W$, and $\ell^0 \subset W^*$ denotes the annihilator of the 1-dimensional space $\ell$. The annihilators were chosen merely for convenience and consistency with the literature. A choice of finitely many points on the momentum curve determines a cyclic polytope, as well as equally many points on the rational normal curve. As any $m$ of them are linearly independent, say parametrised by $\ell_1,\ldots,\ell_m$, they determine a hyperplane $H$. Its annihilator can be read from the contraction
\begin{align}\label{fast Gale}
\left\langle
\ell_1 \otimes \cdots \otimes \ell_m,
\ell^0 \otimes \cdots \otimes \ell^0
\right\rangle,
\end{align}
which is zero, and thus well-defined, if and only if $\ell\in\{\ell_1,\ldots,\ell_m\}$, where $\ell^0\otimes \cdots \otimes \ell^0 \subset S^mW^*$ is viewed as an element of $(W^*)^{\otimes m}$. Indeed, denoting the canonical inclusion of $S^mW^*$ into $(W^*)^{\otimes m}$ by $\varphi$, here called the Veronese factorization structure, the annihilator of $H$ is the 1-dimensional space $\varphi^t \ell_1 \otimes \cdots \otimes \ell_m$.

Expressing the contraction \eqref{fast Gale} in coordinates, i.e., using the affine chart on $S^mW^*$ in which the rational normal curve is the momentum curve and a suitable chart on its dual, we obtain a polynomial expression
\begin{gather}\label{coordinate gale}
\langle (t_1,-1) \otimes \cdots \otimes (t_m,-1), (1,t) \otimes \cdots \otimes (1,t)\rangle =
\prod_{j=1}^m (t_j-t).
\end{gather}
The proof of Gale’s evenness condition, determining whether $H$ defines a facet of the cyclic polytope, follows directly from \eqref{coordinate gale} and its geometric interpretation as the contraction of a normal vector of $H$ with a point on the momentum curve. 
A detailed proof would require introducing notation, which we omit for brevity.
This geometric approach complements standard treatments (e.g., \cite{gale1963neighborly, grunbaum1967convex, ziegler1993lectures}), which primarily rely on algebraic arguments involving the expanded form of $\prod (t_j - t)$.
By collecting derivatives of \eqref{coordinate gale}, one recovers the Vandermonde identities \cite{apostolov2004hamiltonian}.

A detailed explanation of a generalised Gale condition and the derivation of these identities can be found in \Cref{s3}.
The above analysis not only sheds light on the geometric characteristics of cyclic polytope theory but also encourages further investigation.

Dually, the annihilator of $\ell^0 \otimes \cdots \otimes \ell^0 \subset S^mW^*$ contains the $\varphi^t$-image of $\sum_{j=1}^{m} \Sigma_{j,\ell}$,
\begin{align}\nonumber
\Sigma_{j,\ell} :=
W ^{\otimes (j-1)} \otimes \ell \otimes W^{\otimes (m-j)},
\end{align} 
or, equally, the $\varphi^t$-image of any $\Sigma_{j,\ell}$, $j=1,\ldots,m$, since $\varphi^t$ projects onto symmetric tensors, where $W^{\otimes 0}$ is interpreted as an empty product. The ambiguity in $j$ occurs here because the Veronese factorization structure is the simplest and most symmetric structure, and its interpretation is carried out by factorization curves. To see if $\varphi^t \Sigma_{j,\ell}$ is a hyperplane, one computes the dimension of the intersection of $\Sigma_{j,\ell}$ with $\ker \varphi^t = (\varphi \hspace{.1cm} S^mW^*)^0$, or finds the dimension of its annihilator
\begin{align}
\left( \varphi^t \Sigma_{j,\ell} \right)^0
=
\varphi^{-1}\left( \varphi (S^mW^*) \cap ( \Sigma_{j,\ell} )^0 \right),
\end{align}
both leading to the same condition
\begin{align}\label{intro cond}
\dim
\left(
\varphi (S^mW^*) \cap ( \Sigma_{j,\ell})^0
\right)
= 1,
\end{align}
which is fulfilled for any $\ell\in\mathbb{P}(W)$. In particular, facets of the (simple) polytope dual to the cyclic polytope lie on hyperplanes $\varphi^t \Sigma_{j,\ell}$, where $\ell$ parametrise directions determined by vertices of the cyclic polytope. Furthermore, \Cref{faces} shows that its faces lie on subspaces of the form $\varphi^t \left( \Sigma_{i_1,\ell_1} \cap \cdots \cap \Sigma_{i_r,\ell_r} \right)$ for some $r\leq m$. \\

To summarise, we found that since \eqref{intro cond} holds, $\varphi^t \Sigma_{j,\ell}$ is a hyperplane as well as the annihilator of $\ell^0 \otimes \cdots \otimes \ell^0$.
Because $\varphi^t (\Sigma_{1,\ell_1}) \cap \cdots \cap \varphi^t (\Sigma_{m,\ell_m}) = \varphi^t (\ell_1 \otimes \cdots \otimes \ell_m)$, which can be verified in this case directly, the annihilator of the hyperplane given by $\ell_j^0 \otimes \cdots \otimes \ell_j^0$, $j=1,\ldots,m$, is $\varphi^t \ell_1 \otimes \cdots \otimes \ell_m$.
When particular affine charts are used, the Gale evenness condition follows and we obtain the framework of cyclic polytopes.

From this viewpoint, a generalization of the theory becomes apparent: a general inclusion satisfying an analogue of \eqref{intro cond} and general affine charts can be used to extend the cyclic polytope framework.
Remarkably, even alternative affine charts within the Veronese factorization structure yield many unexpected polytope classes beyond cyclic polytopes, as explored in \cite{brandenburg2024}.
This insight provides a contextual explanation of cyclic polytopes and their theory. \\

Wishing to preserve the clarity of computations above and maximise their use, we arrived at the definition of a factorization structure: a linear inclusion $\varphi$ of an $(m+1)$-dimensional vector space into the tensor product of $m$ 2-dimensional vector spaces such that the obvious analogue of \eqref{intro cond} holds for any $j$ and generic $\ell$. The genericity-requirement means that $\varphi^t \Sigma_{j,\ell}$ is a hyperplane only for generic $\ell$. This definition and the presentation provided apply to vector spaces over complex numbers as well. In complex case, the momentum curve can be realified, resulting in the Carathéodory curve \cite{caratheodory1911variabilitatsbereich}, whose polytopes are known to be cyclic \cite{grunbaum1967convex,ziegler1993lectures}.

Note that when $\ell_1, \ldots, \ell_m\in\mathbb{P}(W)$ vary while remaining pairwise distinct, the corresponding $m$ 1-parametric families of hyperplanes $\varphi^t\Sigma_{j,\ell_j}$, $j=1\ldots,m$, provide coordinates on a Zariski-open subset of $\mathbb{P}(S^mW)$, since their respective annihilators $(\ell_j^0)^{\otimes m}$, $j=1,\ldots,m$, remain linearly independent.
These coordinates are called \textit{separable}.
Thus, a point in this open subset of $\mathbb{P}(S^mW)$ is determined by a point in the $m$-product $\mathbb{P}(W) \times \cdots \times \mathbb{P}(W)$, showing that a projective space factors into a product of projective lines, at least locally. Hence the name factorization structures.

\subsection{Factorization structures in differential geometry}\label{dg motivation}
Most of this section is unpublished, with references provided where applicable. It gives a brief look at factorization structures in differential geometry, thereby offering another reason to study them, and outlines applications of results from this article in studying Kähler metrics.

As mentioned above, 2-dimensional factorization structures were originally explored in ambitoric compactifications \cite{apostolov2016ambitoric,apostolov2015ambitoric}, where were used to achieve a classification of extremal Kähler structures on compact toric 4-orbifolds with the second Betti number two.
As a result, many geometries and their classifications were unified under the framework of ambitoric geometry (see introduction in \cite{apostolov2016ambitoric}).
Importantly, the shape of ambitoric structures shows that factorization structures are not merely auxiliary computational tools but play an intrinsic role, they determine the Kähler structure. 

Appendix C of \cite{apostolov2015ambitoric} shows that regular ambitoric geometries can be viewed as quotients of a 5-dimensional manifold of Sasaki type by Sasaki-Reeb vector fields.
Building on this idea, \cite{apostolov2021cr} extends the approach by studying weighted extremality of quotients of Sasaki-type manifolds in general dimension.
Additionally, it finds two explicit families of separable geometries, which describes in terms of CR twists: twists of orthotoric geometry \cite{apostolov2003geometry,apostolov_hamiltonian_2006,apostolov2004hamiltonian,apostolov_hamiltonian_2008,apostolov_hamiltonian_2008-1} and twists of a Kähler product of toric Riemann surfaces.
In real dimension 4, as \cite{apostolov2020levi} and \cite{apostolov2021cr} shows, these two families recover all ambitoric geometries.

The success of factorization structures in classifying extremal 4-orbifolds and the elegance of identifying these as natural quotients of Sasaki-type geometries motivated the author's thesis \cite{pucek2022extremal}, where the two aforementioned families of separable geometries were recognised to be associated with Veronese and product Segre factorization structures. \\

More generally, an $m$-dimensional factorization structure $\varphi:\mathfrak{h}\to V_1^* \otimes \cdots \otimes V_m^*$, $\beta \in \mathfrak{h}$, and $m$ functions $A_1,\ldots,A_m$ of one variable determine the toric \textit{separable Kähler geometry}

\begin{align}\label{sep K-geom}
\nonumber
g_\beta&=
-
\sum_{j=1}^m
\left(
	\frac{\langle\partial_{x_j}\mu_\beta,\psi_j([1:x_j])\rangle}
		{A_j(x_j)}
	dx_j^2 +
	\frac{A_j(x_j)}
		{\langle\partial_{x_j}\mu_\beta,\psi_j([1:x_j])\rangle}
	\langle\partial_{x_j}\mu_\beta,dt\rangle^2
\right)\\\nonumber
\omega_\beta&=
\sum_{j=1}^mdx_j\wedge\langle\partial_{x_j}\mu_\beta,dt\rangle\\ 
J_\beta dx_j&=
-
\frac{A_j(x_j)}{\langle\partial_{x_j}\mu_\beta,\psi_j([1:x_j])\rangle}\langle\partial_{x_j}\mu_\beta,dt\rangle
\hspace{1cm}
J_\beta dt=
\sum_{j=1}^m\frac{\psi_j([1:x_j])\text{ mod }\beta}{A_j(x_j)}dx_j,
\end{align}
where $dt$ is a 1-form valued in $\mathfrak{h} / \langle \beta \rangle$, $\psi_j$, $j=1,\ldots,m$, are factorization curves (\Cref{s21}) in appropriate affine charts, and $\mu_\beta = \varphi^t x / \langle x, \varphi \beta \rangle$ with $x = (1,x_1) \otimes \cdots \otimes (1,x_m)$ is the momentum map valued in the affine chart given by $\beta$.
In particular, each $\partial_{x_j} \mu_\beta$ lies in the annihilator $\beta^0$, ensuring that the pairing $\langle \partial_{x_j} \mu_\beta, dt \rangle$ is well-defined.

While it is now possible to explicitly write down the Kähler structure for the standard Segre-Veronese factorization structure (see \Cref{SV def}), doing so would require introducing additional notation related to grouped slots, which we omit for brevity.
A detailed exposition will appear in future work.
However, we note that the Kähler structure (\ref{sep K-geom}) is obtained as the quotient of a toric separable CR geometry whose acting torus has the Lie algebra $\mathfrak{h}$, and which is equipped with special coordinates $x_1,\ldots,x_m$, called separable (see \cite{apostolov2021cr}), in which the CR structure depends on functions of one variable.
Specifically, it resembles $J_\beta dt$ from \eqref{sep K-geom} which depends on functions $\psi_j/A_j \text{ mod } \beta$, $j=1,\ldots,m,$ of one variable.\\

The advantages of separable Kähler and CR geometries are three-fold: a unified framework for many examples, each facet of the rational Delzant polytope or polyhedron of a separable geometry is described by $x_j = \text{\textit{const.}}$ for some $j$, and, unknowns in PDEs involving these geometries depend on functions of one variable. \\

Separable geometries associated with factorization structures provide a framework for vast number of toric CR and Kähler geometries which are amenable to uniform computations.
Examples appearing in the literature are the aforementioned ambitoric geometries, twists of a Kähler product of toric Riemann surfaces, and twists of orthotoric geometries, which together correspond to two simplest factorization structures: Veronese and product Segre.
The overflow of new separable geometries arises from the vast number of factorization structures.
A classification of factorization structures, would not only describe all local separable geometries in a given dimension $n$, but, as in ambitoric case, could also facilitate the classification of extremal structures on certain $n$-orbifolds.

The study of global behaviour of separable geometries includes compactifications, which are related to rational Delzant polytopes, and general boundary behaviour associated with polyhedra.
In the compact case, the image of the momentum map of a toric separable Kähler geometry is in particular a compatible polytope (\Cref{compatible polytope}) with at most $2m$ facets, where $m$ is the complex dimension of the geometry.
Any such carries separable coordinates $x_j$, $j=1,\ldots,m$, in which the facets are described by $x_j = const.$ (see the last paragraph of \Cref{discr geom} for separable coordinates).
This enables the derivation of simple necessary and sufficient conditions for compactification, similar to those in \cite{apostolov2004hamiltonian}.
A key step in compactifying is fixing an underlying rational Delzant polytope compatible with a factorization structure, constructed using Vandermonde identities (\Cref{s32}).
As part of this broader framework, understanding the effects of newly introduced operations on factorization structures, namely the product (\Cref{products}) and quotient (\Cref{quotient subsection}), on separable geometries remain to be explored.\\

Separable geometries provide a favourable setting where geometric PDEs are likely to be solved explicitly.
In particular, this applies to the problem of finding Calabi's extremal metrics, also famous for its connection with K-stability, which is governed by the extremality equation, a PDE seeking metrics whose scalar curvature is a Killing potential. 
While explicit solutions are rare and often obtained through ad-hoc methods, separable geometries offer a systematic framework for recovering known solutions and discovering new ones.

The Segre-Veronese factorization structure underpins this framework.
Solutions of the extremality equation for associated separable geometries are rational functions depending on tensors $\Gamma_j$, $j=1,\ldots,k$, determining the factorization structure and on finitely many parameters, whose number relates to degrees of involved factorization curves (see \Cref{degree fc}).
A general strategy for solving the PDE is to verify, using generalised Vandermonde identities (see \Cref{gen VI}), which solutions of compatibility conditions satisfy the PDE.
The shapes and decomposability of tensors $\Gamma_j$, $j=1,\ldots,k$, (see \Cref{tensors split} and \Cref{symmetric SV}) are crucial in obtaining useful compatibility conditions, which necessitates the classification of Segre-Veronese factorization structures.
This article achieves a partial classification by characterising decomposable Segre-Veronese factorization structures (\Cref{classification}).
Already for one of the simplest of such structures, the product Segre-Veronese factorization structure associated to a partition, known extremal metrics \cite{apostolov2021cr} are recovered for the two trivial partitions, and new solutions are obtained for any non-trivial partition.
Furthermore, analysis for decomposable Segre-Veronese factorization structures corresponding to partitions $m=d_1+\cdots+d_k$ for small $k$ indicates that the extremality equation for the associated geometries can be solved uniformly, as opposed to case-by-case approach.

\section{Structure theory}\label{s2}

The first section offered an abundance of examples of factorization structures, all of which were, notably, of Segre-Veronese type.
In fact, Segre-Veronese factorization structures are the only known examples of factorization structures.
Naturally, one can ask
\begin{enumerate}
\item[(i)] Is every factorization structure of Segre-Veronese type? 
\item[(ii)] What is the classification of Segre-Veronese factorization structures?
\end{enumerate}
These questions are the prime motivation for this section.
To address them, we develop an abstract theory of factorization structures, focusing on key aspects such as factorization curves, quotients, and complexifications.
All of these are essential in proving main results.

One of the main results is the characterisation of Segre-Veronese factorization structures: a factorization structure is a Segre-Veronese factorization structure if and only if all of its factorization curves are decomposable.
Therefore, question (i) can be formulated intrinsically as \Cref{question}, asking if every factorization curve is decomposable.

Examples of the first section make it clear that classifying defining tensors of Segre-Veronese factorization structures requires an intense effort.
However, focusing on decomposable Segre-Veronese factorization structures, we achieve their characterisation in \Cref{classification} as iterated products of Veronese factorization structures.
This result holds for a broader class of Segre-Veronese factorization structures than decomposable ones, as explained  in \Cref{classification}.

We note that results of this section have applications beyond the internal theory of factorization structures.
For example, \Cref{faces} is used to describe linear spaces determining faces of compatible cones and polytopes, which is crucial for understanding their geometric properties.
For further applications, see \Cref{discr geom} and \Cref{dg motivation}. 

\subsection{Factorization curves}\label{s21}
The defining condition of factorization structures \eqref{wfs def condition} invites us to consider generically defined curves
\begin{align}\label{rational maps defn}
\mathbb{P}(V_j) &\dashrightarrow \mathbb{P}(\mathfrak{h})\nonumber\\
\ell &\mapsto \varphi^{-1} \left( \varphi(\mathfrak{h}) \cap \Sigma_{j,\ell}^0 \right)
\end{align}
for $j=1,\ldots,m$.

For example, in the 2-dimensional Segre factorization structure $V_1^*\otimes \Gamma_1 + \Gamma_2 \otimes V_2^* \hookrightarrow V_1^* \otimes V_2^*$, we have two curves
\begin{align}
\mathbb{P}(V_1) \backslash \{ \Gamma_2 \} &\to \mathbb{P}(V_1^*\otimes \Gamma_1 + \Gamma_2 \otimes V_2^*) \nonumber \\
\ell &\mapsto (V_1^*\otimes \Gamma_1 + \Gamma_2 \otimes V_2^*) \cap \ell^0 \otimes V_2^* = \ell^0 \otimes \Gamma_1 \label{Segre lines 1}
\end{align}
and
\begin{align}
\mathbb{P}(V_2) \backslash \{ \Gamma_1 \} &\to \mathbb{P}(V_1^*\otimes \Gamma_1 + \Gamma_2 \otimes V_2^*) \nonumber \\
\ell &\mapsto (V_1^*\otimes \Gamma_1 + \Gamma_2 \otimes V_2^*) \cap  V_1^* \otimes \ell^0 = \Gamma_2 \otimes \ell^0, \label{Segre lines 2}
\end{align}
both being (generically defined) lines in $\mathbb{P}^2$.
Note that the points which are excluded from domains of these lines are exactly those where the formula \eqref{rational maps defn} does not determine a point in a projective space.
Similarly, for a general Segre factorization structure, all of its curves in the above sense are generically defined lines. \par
In the case of Veronese factorization structure $S^mW^* \hookrightarrow (W^*)^{\otimes m}$, its first curve reads
\begin{align}
\mathbb{P}(W) &\to \mathbb{P}(S^mW^*) \nonumber \\
\ell &\mapsto S^mW^* \cap \ell^0 \otimes (W^*)^{\otimes (m-1)} = (\ell^0)^{\otimes m}. \label{Veronese curve}
\end{align}
This curve is defined globally, i.e., for all $\ell \in \mathbb{P}(W)$, and is known as the rational normal curve.
Because the domain of any other curve is again $\mathbb{P}(W)$, and $S^mW^* \cap \Sigma_{j,\ell}^0 = (\ell^0)^{\otimes m}$ for any $\ell \in \mathbb{P}(W)$ and any $j=1,\ldots,m$, all curves coincide. \par
For the standard Segre-Veronese factorization structure \eqref{SV inclusion}, these curves were already found in \eqref{standard curves}, and, similarly to the example above, coincide in grouped slots. More concretely, for each $i=1,\ldots,k$ we have $d_i$ identical curves
\begin{align}\label{SV generic curves}
\mathbb{P}(W_i) &\to \mathbb{P}\left( \sum_{j=1}^{k} ins_j \left( S^{d_j}W_j^*\otimes \Gamma_j \right) \right) \nonumber \\
\ell &\mapsto ins_i \left( (\ell^0)^{\otimes d_i} \otimes \Gamma_i \right),
\end{align}
whose locus of indeterminacy is at points $\ell\in\mathbb{P}(V_j)$ for which
$\dim \left( \varphi(\mathfrak{h}) \cap \Sigma_{j,\ell}^0 \right) > 1$.
The latter happens if a defining tensor $\Gamma_j$, $j\neq i$, decomposes in grouped $i$-slots, similarly as in 2-dimensional Segre factorization structure above. \\

Now we establish a way of extending these generically defined curves into projective curves. The first step is

\begin{proposition}\label[proposition]{curves are reg}
Let $\varphi: \mathfrak{h} \to V^*$ be a real/complex factorization structure of dimension $m$. The generically defined curve \eqref{rational maps defn} is a regular map on an open and non-empty subset $W \subset \mathbb{P}(V_j)$ of degree at most $m$, i.e., it is given by homogeneous polynomials of the same degree (at most $m$) in homogeneous coordinates on $\mathbb{P}(V_j)$ and $\mathbb{P}(\mathfrak{h})$. More concretely, $W$ is an open subset of the open set
$$U=
\{\ell\in\mathbb{P}(V_j)\hspace{.1cm}|\hspace{.1cm}
\dim\left( \varphi(\mathfrak{h})\cap\Sigma_{j,\ell}^0 \right)
=1\}.$$
\end{proposition}
\begin{proof}
We describe $\varphi(\mathfrak{h}) \cap \Sigma_{j,\ell}^0$ as a solution of a linear system of $2^m-1$ equations, which depend homogeneously on $\ell$, with $2^m$ variables. Then, we apply Cramer's rule to show the claim for $\ell \mapsto \varphi(\mathfrak{h}) \cap \Sigma_{j,\ell}^0$, and thus for \eqref{rational maps defn}.

Fix a basis of $V_j^*$, $j=1,\ldots,m$, and let
$c^{a_1\cdots a_m}: V^*=V_1^*\otimes\cdots\otimes V_m^* \to \mathbb{F}$,
$a_j\in\{1,2\}$,
be the corresponding standard coordinates, where $\mathbb{F}$ denotes the field $\mathbb{R}$ or $\mathbb{C}$ depending on the factorization structure being real or complex, respectively. For $\ell\in U$, the subspace $\Sigma_{j,\ell}^0$ in $V^*$ is then described by $2^{m-1}$ independent linear equations
\begin{align}\label{sigmaeq}
	xc^{a_1\cdots a_{j-1}1a_{j+1}\cdots a_m}+
	yc^{a_1\cdots a_{j-1}2a_{j+1}\cdots a_m}
	=0,
	\text{ }
	a_i\in\{1,2\}
	\text{ for }
	i\neq j,
\end{align}
where $\ell^0=[-y:x]$. Note, these can be viewed as equations of homogeneous polynomials of degree one in $x$ and $y$ with coefficients $c^{\cdots}$'s. The $(m+1)$-dimensional subspace $\varphi(\mathfrak{h})$ in $V^*$ can be described via $2^m-(m+1)$ independent linear equations, call that system (E), which do not depend on $\ell$. Finally, the subspace $\varphi(\mathfrak{h})\cap\Sigma_{j,\ell}^0$, which is one-dimensional for a fixed generic $\ell$, is the solution to the system of $2^{m-1}+2^m-(m+1)$ linear equations, \eqref{sigmaeq} and (E). Clearly, this system has only $2^m-1$ independent equations, and these can be chosen as the system (E) together with another $m$ independent linear equations from \eqref{sigmaeq}. The latter stay independent on an open subset $V\subset\mathbb{P}(V_k)$ containing $\ell$. Thus, for $\ell\in W := U\cap V$, knowing the intersection $\varphi(\mathfrak{h})\cap\Sigma_{j,\ell}^0$ is equivalent to a system of $2^m-1$ independent linear equations, $m$ of which are homogeneous of degree one in $\ell$ and the others do not depend on $\ell$. Using Cramer's rule to solve this system (see \cite{gorodentsev2016algebra}) shows that $\varphi(\mathfrak{h})\cap\Sigma_{j,\ell}^0$ depends on $\ell$ in a homogeneous way and the degree of homogeneity is at most $m$ which, for example, is attained in the case when $\varphi(\mathfrak{h})=S^mW^*$.
\end{proof}

\begin{lemma}\label[lemma]{curve extension}
Let $U$ be an open non-empty subset of $\mathbb{P}^1$. A regular map $f:U\to\mathbb{P}^n$ extends uniquely to a regular map on $\mathbb{P}^1$.
\end{lemma}
\begin{proof}
In homogeneous coordinates, such map $f$ is given by $f([x:y])=[f_0([x:y]):\cdots:f_n([x:y])]$ where $f_j$ are homogeneous polynomials of the same degree. The expression $[f_0([x:y]):\cdots:f_n([x:y])]$ fails to define a point in $\mathbb{P}^n$ if and only if all $f_j$ vanish at $[x:y]$. However, this means that all $f_j$ have a factor in common which can be removed. Because any open non-empty set in $\mathbb{P}^1$ is $\mathbb{P}^1$ without finitely many points, $f$ extends this way to whole $\mathbb{P}^1$.
\end{proof}

Combining \Cref{curves are reg} and \Cref{curve extension} allows us to define factorization curves as extensions of \eqref{rational maps defn}.

\begin{defn}\label[defn]{curves defn}
Let $\varphi:\mathfrak{h}\to V^*$ be a real/complex factorization structure of dimension $m$. For each $j\in\{1,\ldots,m\}$ we define \textit{factorization curve} $\psi_j:\mathbb{P}(V_j)\to\mathbb{P}(\mathfrak{h})$ as the extension of the regular map generically given by \eqref{rational maps defn}.
\end{defn}

We continue with examples of generically defined curves from above.

\begin{example}\label[example]{segre lines example}
Since extensions of the generically defined curves \eqref{rational maps defn} are unique, we conclude from \eqref{Segre lines 1} and \eqref{Segre lines 2} that the 2-dimensional Segre factorization structure has two distinct factorization curves, being lines
\begin{align}\label{segre 1}
\psi_1: \mathbb{P}(V_1) &\to \mathbb{P}(V_1^* \otimes \Gamma_1 + \Gamma_2 \otimes V_2^*) \nonumber \\
\ell &\mapsto \ell^0 \otimes \Gamma_1
\end{align}
and
\begin{align}\label{segre 2}
\psi_2: \mathbb{P}(V_2) &\to \mathbb{P}(V_1^* \otimes \Gamma_1 + \Gamma_2 \otimes V_2^*) \nonumber \\
\ell &\mapsto \Gamma_2 \otimes \ell^0,
\end{align}
intersecting at one point $\Gamma_2 \otimes \Gamma_1$.
Note also that respective linear spans of images of \eqref{segre 1} and \eqref{segre 2} are $V_1^* \otimes \Gamma_1$ and $\Gamma_2 \otimes V_2^*$ (see \Cref{fig2}).
A general $m$-dimensional Segre factorization structure has $m$ distinct factorization curves, all being lines.
However, their intersections can be arbitrarily complicated.
For example, for 3-dimensional Segre factorization structure from \Cref{indec Segre}, there is no intersection between any two factorization curves/lines.
On the other extreme, in the product Segre factorization structure from \Cref{one intersection example} corresponding to the partition $m = 1 + \cdots + 1$, all factorization lines intersect at the unique point $\otimes_{r=1}^m a^r$.
One can form iterative products of 1-dimensional factorization structures to obtain an $m$-dimensional Segre factorization structure with decomposable defining tensors and with prescribed intersections of factorization lines.
\end{example}

As already found in \eqref{Veronese curve}, all factorization curves in a Veronese factorization structure coincide, $\psi_1 = \cdots = \psi_m$, being the rational normal curve of degree $m$. Such a curve has two properties we frequently use: its linear span is exactly $S^mW^*$, and any $m+1$ points on the curve are linearly independent \cite{harris2013algebraic}. 

\begin{example}\label[example]{k=2 curves example}
The Segre-Veronese factorization structure corresponding to the partition $m=d_1+d_2$ from \Cref{k=2 example}, abbreviated here as $\varphi:\mathfrak{h} \to V^*$, has two distinct factorization curves
\begin{align}
\mathbb{P}(W_1) &\to \mathbb{P}\left( S^{d_1}W_1^* \otimes \Gamma_1 + \Gamma_2 \otimes S^{d_2}W_2^* \right) \nonumber \\
\ell &\mapsto (\ell^0)^{\otimes d_1} \otimes \Gamma_1 =
\varphi(\mathfrak{h}) \cap \Sigma_{1,\ell}^0 = \cdots = 
\varphi(\mathfrak{h}) \cap \Sigma_{d_1,\ell}^0
\end{align}
and
\begin{align}
\mathbb{P}(W_2) &\to \mathbb{P}\left( S^{d_1}W_1^* \otimes \Gamma_1 + \Gamma_2 \otimes S^{d_2}W_2^* \right) \nonumber \\
\ell &\mapsto \Gamma_2 \otimes (\ell^0)^{\otimes d_2} =
\varphi(\mathfrak{h}) \cap \Sigma_{d_1+1,\ell}^0 = \cdots =
\varphi(\mathfrak{h}) \cap \Sigma_{m,\ell}^0,
\end{align}
which respectively have degrees $d_1$ and $d_2$.
Their respective linear spans are $S^{d_1}W_1^* \otimes \Gamma_2$ and $\Gamma_1 \otimes S^{d_2}W_2^*$, and both are rational normal curves within their linear span.
They intersect if and only if both $\Gamma_1$ and $\Gamma_2$ are decomposable, in which case the intersection is the unique point $\Gamma_2 \otimes \Gamma_1$.
\end{example}

\begin{example}\label[example]{3-Veronese example}
To illustrate how products of factorization structures influence intersections of factorization curves, we discuss the Segre-Veronese factorization structure \eqref{k=3}, abbreviated here as $\varphi:\mathfrak{h} \to V^*$, whose distinct curves are
\begin{align}
\mathcal{C}_1: \mathbb{P}(W_1) &\to \mathbb{P}(\mathfrak{h}) \nonumber \\
\ell &\mapsto (\ell^0)^{\otimes d_1} \otimes \Gamma_1 \otimes \Gamma,
\end{align}
and
\begin{align}
\mathcal{C}_2: \mathbb{P}(W_2) &\to \mathbb{P}(\mathfrak{h}) \nonumber \\
\ell &\mapsto \Gamma_2 \otimes (\ell^0)^{\otimes d_2} \otimes \Gamma,
\end{align}
and
\begin{align}
\mathcal{C}_3: \mathbb{P}(W_3) &\to \mathbb{P}(\mathfrak{h}) \nonumber \\
\ell &\mapsto \Gamma_3 \otimes (\ell^0)^{\otimes d_3},
\end{align}
with respective degrees $d_1, d_2$ and $d_3$, and respective linear spans $S^{d_1}W_1^* \otimes \Gamma_1 \otimes \Gamma$, $\Gamma_2 \otimes S^{d_2}W_2^* \otimes \Gamma$ and $\Gamma_3 \otimes S^{d_3}W_3^*$.
The following analysis of pairwise intersections of $\mathcal{C}_1, \mathcal{C}_2$ and $\mathcal{C}_3$ clarifies how to choose defining tensors $\Gamma_1, \Gamma_2$ and $\Gamma$ to obtain prescribed intersection properties of the curves.
Similarly to \Cref{k=2 curves example}, the curves $\mathcal{C}_1$ and $\mathcal{C}_2$ intersect if and only if both $\Gamma_1$ and $\Gamma_2$ are decomposable, in which case the intersection is the unique point $\Gamma_2 \otimes \Gamma_1 \otimes \Gamma$.
Curves $\mathcal{C}_1$ and $\mathcal{C}_3$ intersect if and only if $\Gamma$ is decomposable and $\Gamma_3 = A \otimes \Gamma_1$ for some decomposable 1-dimensional space $A \subset S^{d_1}W_1^*$, in which case the intersection is the unique point $A \otimes \Gamma_1 \otimes \Gamma$.
Finally, curves $\mathcal{C}_2$ and $\mathcal{C}_3$ intersect if and only if $\Gamma$ is decomposable and $\Gamma_3 = \Gamma_2 \otimes B$ for some decomposable 1-dimensional space $B \subset S^{d_2}W_2^*$, in which case the intersection is the unique point $\Gamma_2 \otimes B \otimes \Gamma$.
\end{example}

\begin{example}\label[example]{SVfs curves}
The standard Segre-Veronese factorization structure \eqref{SV inclusion}, abbreviated here as $\varphi: \mathfrak{h} \to V^*$, has $k$ distinct factorization curves,
\begin{align}
\mathcal{C}_i: \mathbb{P}(W_i) &\to \mathbb{P}(\mathfrak{h}) \nonumber \\
\ell &\mapsto ins_i \left( (\ell^0)^{\otimes d_i} \otimes \Gamma_i \right),
\hspace{.5cm}
i=1,\ldots,k,
\end{align}
which relate to factorization curves $\psi_1,\ldots,\psi_m$ by $\mathcal{C}_i = \psi_{d_1+\cdots+d_{i-1}+1} = \cdots = \psi_{d_1+\cdots+d_{i-1}+d_i}$, $i=1,\ldots,k$, where $d_0$ is defined to be zero.
The linear span of $\mathcal{C}_i$ is $ins_i \left( S^{d_i}W_i^* \otimes \Gamma_i \right)$, $i=1,\ldots,k$, showing that the standard Segre-Veronese factorization structure is the sum of linear spans of its factorization curves.
Finally, the degree of $\mathcal{C}_i$ is $d_i$, $i=1,\ldots,k$.
\end{example}

The following lemma plays an essential role in the study of factorization curves.

\begin{lemma}\label[lemma]{alg geom curves unique}
Let $\varphi:\mathfrak{h}\to V^*$ be a complex factorization structure, and let $\psi_i:\mathbb{P}(V_i)\to\mathbb{P}(\mathfrak{h})$ and $\psi_j:\mathbb{P}(V_j)\to\mathbb{P}(\mathfrak{h})$, $i\neq j$, be two factorization curves whose images coincide in $\infty$-many points. Then $\im \psi_i = \im\psi_j$.
\end{lemma}
\begin{proof}
Let $r\in\{i,j\}$. Then, $\im\psi_r$ is a projective variety since the image of a projective variety is closed (see \cite{shafarevich1994basic}). Clearly, $\im\psi_i \cap \im\psi_j$ is closed in $\im\psi_r$ and contains $\infty$-many points. Therefore, $\psi_r^{-1} \left( \im\psi_i \cap \im\psi_j \right)$ is closed and contains $\infty$-many points, thus equals to $\mathbb{P}(V_r)$. This is equivalent with $\im\psi_r = \im\psi_i \cap \im\psi_j$, and hence $\im\psi_i = \im\psi_j$.
\end{proof}

Its first application shows that complex factorization curves are injective (see \Cref{curves are inj}).

\begin{proposition}\label[proposition]{ell_k^0 fixed}
Let $\mathfrak{h}$ be a complex factorization structure. Then $\forall \ell\in\mathbb{P}(V_j)\hspace{.2cm}\exists T\in \hat{V}_j^*$ such that $\varphi\circ\psi_j(\ell)=\ell^0\otimes T$, where $\ell^0$ is to be viewed at $j$-th slot.
\end{proposition}
\begin{proof}
Suppose the defining polynomials of $\psi_j$ are of degree $d$. Therefore, on an open non-empty subset $U\subset\mathbb{P}(V_j)$ where $\varphi\circ\psi_j(\ell)=\varphi(\mathfrak{h})\cap\Sigma_{j,\ell}^0$, there exist $T(\ell)\in\mathbb{P}(\hat{V}_j^*)$ given by homogeneous polynomials of degree $d-1$ such that $\varphi\circ\psi_j = \ell^0\otimes T(\ell)$. By \Cref{curve extension}, the map $\ell\mapsto T(\ell)$ uniquely extends to a regular map $T:\mathbb{P}(V_j) \to \mathbb{P}(\hat{V}_j^*)$, and thus defines the curve $\mathcal{C}: \mathbb{P}(V_j) \to \mathbb{P}(V^*)$ by $\mathcal{C}(\ell)=\ell^0\otimes T(\ell)$. Since $\mathcal{C}$ and $\varphi\circ\psi_j$ agree on an open non-empty set, \Cref{alg geom curves unique} shows $\mathcal{C}=\varphi\circ\psi_j$.
\end{proof}

\begin{corollary}\label[corollary]{curves are inj}
Factorization curves in a complex factorization structure are injective. 
\end{corollary}
\begin{proof}
If
$\varphi\circ\psi_k(\ell_k)=\varphi\circ\psi_k(\tilde{\ell}_k)$,
i.e.,
$\ell_k^0\otimes T(\ell_k)=\tilde{\ell}_k^0\otimes T(\tilde{\ell}_k)$,
then $\ell_k^0=\tilde{\ell}_k^0$, and thus $\ell_k=\tilde{\ell}_k$.
\end{proof}

\subsection{Complexification}

Let $\varphi^\mathbb{C}:\mathfrak{h}\otimes\mathbb{C}\to V^*\otimes\mathbb{C} = V_1^* \otimes \mathbb{C} \otimes \cdots \otimes V_m^* \otimes \mathbb{C}$ be the complexification of a real factorization structure $\varphi:\mathfrak{h}\to V^*$, and denote 
$$
(V_1^*\otimes_\mathbb{R}\mathbb{C}) \otimes_\mathbb{C} \cdots \otimes_\mathbb{C} (V_{j-1}^*\otimes_\mathbb{R}\mathbb{C})\otimes_\mathbb{C}
L
\otimes_\mathbb{C}(V_{j+1}^*\otimes_\mathbb{R}\mathbb{C})\otimes_\mathbb{C}\cdots\otimes_\mathbb{C}(V_m^*\otimes_\mathbb{R}\mathbb{C})
$$
by $\mathbb{C}\Sigma_{j,L}^0$ for any complex 1-dimensional subspace $L\subset V_j^*\otimes\mathbb{C}$. Such a complexification is called a \textit{complexified factorization structure}.

\begin{proposition}
A map $\varphi:\mathfrak{h}\to V^*$ is a real factorization structure if and only if its complexification $\varphi^\mathbb{C}:\mathfrak{h}\otimes\mathbb{C}\to V^*\otimes \mathbb{C}$ is a complex factorization structure.
\end{proposition}
Before we prove it, we remark on a general property which will be used multiple times and also in the proof.
\begin{rem}\label[rem]{Schubert rem}
In general, for a real/complex factorization structure, the set 
\begin{align}
	U_d:=\{\ell\in\mathbb{P}(V_j):\dim\left( \varphi(\mathfrak{h})\cap\Sigma_{j,\ell}^0 \right)\geq d\}
\end{align}
is the preimage of 
\begin{align}\label{schubert}
	\mathcal{U}_d=\{\Lambda\in Gr(2^{m-1},V^*)\hspace{.1cm}|\hspace{.1cm}
	\dim\left( \varphi(\mathfrak{h})\cap\Lambda \right)\geq d\}
\end{align}
under the regular map $\mathbb{P}(V_j)\to Gr(2^{m-1},V^*)$ defined by $\ell\mapsto\Sigma_{j,\ell}^0$, and hence is closed since \eqref{schubert} is a (closed) Schubert variety (see \cite{eisenbud20163264,harris2013algebraic,pucek2022extremal}).
\end{rem}
\begin{proof}
On an open non-empty subset of $\mathbb{P}(V_j)$ we have
\begin{align}\label{real-cpx identity curves}
1=
\dim \left( \varphi \circ \psi_j(\ell) \otimes \mathbb{C} \right)&=
\dim \left( (\varphi(\mathfrak{h})\cap\Sigma_{j,\ell}^0)\otimes\mathbb{C} \right)=
\dim \left( \varphi(\mathfrak{h})\otimes\mathbb{C}\cap\mathbb{C}\Sigma_{j,\ell\otimes\mathbb{C}}^0 \right)\nonumber\\
&=
\dim \left( \varphi^\mathbb{C}(\mathfrak{h}\otimes\mathbb{C})\cap\mathbb{C}\Sigma_{j,\ell\otimes\mathbb{C}}^0 \right).
\end{align}
If we define
\begin{align}\label{closed1}
Q_d=
\{L\in\mathbb{P}(V_j\otimes\mathbb{C})\hspace{.1cm}|\hspace{.1cm}
\dim\left( \varphi^\mathbb{C}(\mathfrak{h}\otimes\mathbb{C})\cap\mathbb{C}\Sigma_{j,L}^0 \right) \geq d\},
\hspace{.5cm}
d=1,2,
\end{align}
then \eqref{real-cpx identity curves} shows that for $\varphi$ a real factorization structure, $Q_1$ contains $\infty$-many points. Since $Q_1$ is closed, we have $Q_1 = \mathbb{P}(V_j \otimes \mathbb{C})$. Furthermore,
\begin{align}
Q_1 \backslash Q_2 =
\{L\in\mathbb{P}(V_j\otimes\mathbb{C})\hspace{.1cm}|\hspace{.1cm}
\dim\left( \varphi^\mathbb{C}(\mathfrak{h}\otimes\mathbb{C})\cap\mathbb{C}\Sigma_{j,L}^0 \right) = 1\}
\end{align}
is non-empty by \eqref{real-cpx identity curves}, and open since $Q_2$ is closed. Thus, $\varphi^\mathbb{C}$ is a complex factorization structure.\par
On the other hand, for $\varphi^\mathbb{C}$ a complex factorization structure, $Q_1 \backslash Q_2$ is open and non-empty, and thus intersects 
$\{\ell\otimes\mathbb{C}\in\mathbb{P}(V_j\otimes\mathbb{C})\hspace{.1cm}|\hspace{.1cm}\ell\in\mathbb{P}(V_j)\}$
in $\infty$-many points. Equalities \eqref{real-cpx identity curves} at these intersection points show that the closed set
\begin{align}\label{real part}
\{\ell\in\mathbb{P}(V_j)\hspace{.1cm}|\hspace{.1cm}
\dim
\left(
\varphi(\mathfrak{h})\cap\Sigma_{j,\ell_k}^0
\right)
\geq 1\}
\end{align}
is infinite, and thus is the whole $\mathbb{P}(V_j)$. They also show that the open subset where the dimension is 1 is non-empty, which shows that $\varphi$ is a real factorization structure.
\end{proof}

We remark on complexified factorization curves. Unravelling the definition of complexification shows that complexifying the curve
$\psi_j: \mathbb{P}(V_j) \to \mathbb{P}(\mathfrak{h})$,
$$[x^1:x^2] \mapsto \psi_j([x^1:x^2])=\left[ f_1([x^1:x^2]):\cdots:f_{m+1}([x^1:x^2]) \right],$$
means to regard its defining polynomials $f_r$ as complex polynomials, i.e., the complexified curve
$\mathcal{C}: \mathbb{P}(V_j \otimes \mathbb{C}) \to \mathbb{P}(\mathfrak{h} \otimes \mathbb{C})$ is
$$[z^1:z^2] \mapsto \left[ f_1([z^1:z^2]):\cdots:f_{m+1}([z^1:z^2]) \right].$$
Using \Cref{alg geom curves unique} we conclude that $\mathcal{C}$ and the factorization curve $\psi_j^\mathbb{C}$,
given by
$$\varphi^\mathbb{C} \circ \psi_j^\mathbb{C}(L) =
\varphi(\mathfrak{h})\otimes\mathbb{C}\cap\mathbb{C}\Sigma_{j,L}^0,$$
coincide, since by \eqref{real-cpx identity curves} they agree at $\infty$-many points.

\subsection{Degree}\label{degree fc}

Most claims of this subsection hold for projective curves in general, but we stay focused on factorization curves as defined in \Cref{curves defn}. The tautological section
$\tau: \mathbb{P}(\mathfrak{h})\to \mathcal{O}_\mathfrak{h}(1) \otimes \mathfrak{h}$ assigns to each class $[z]\in\mathbb{P}(\mathfrak{h})$ the canonical inclusion of the corresponding 1-dimensional vector space $\langle z \rangle$ into $\mathfrak{h}$ viewed as an element of $\langle z \rangle^*\otimes \mathfrak{h}$, where $\mathcal{O}_\mathfrak{h}(1)$ is the the dual of the tautological line bundle. By pulling $\tau$ back via a factorization curve $\psi_j$ we can view the curve as a section of $\mathcal{O}_{V_j}(e_j)\otimes\mathfrak{h}$,
\begin{equation}\label{curve pullback}
\begin{tikzcd}
\mathcal{O}_{V_j}(e_j)\otimes\mathfrak{h} \arrow[r] & \mathcal{O}_\mathfrak{h}(1)\otimes\mathfrak{h} \\
\mathbb{P}(V_j) \arrow[r, "\psi_j"] \arrow[u, "\psi_j^*\tau"]       & \mathbb{P}(\mathfrak{h}) \arrow[u, "\tau"'],  
\end{tikzcd}
\end{equation}
where $e_j\in\mathbb{Z}$ is determined via the isomorphism
$(\psi_j)^*\mathcal{O}_{\mathfrak{h}}(1)
\cong
\mathcal{O}_{V_j}(e_j)$, using the classification of line bundles over projective spaces. On the other hand, choosing a basis for $\mathfrak{h}$ allows us to view the section $\psi_j^*\tau$ as $\dim \mathfrak{h}$ global sections of $\mathcal{O}_{V_j}(e_j)$. Such sections are homogeneous polynomials of degree $e_j$ which together recover $\psi_j$.\\

\begin{defn}\label[defn]{degree defn}
Let $\mathfrak{h}$ be a real/complex factorization structure. The degree $\deg\psi_j$ of a factorization curve $\psi_j$ is defined to be $\deg\psi_j=e_j$, where $e_j$ is such that $(\psi_j)^*\mathcal{O}_{\mathfrak{h}}(1)
\cong
\mathcal{O}_{V_j}(e_j)$.
\end{defn}

One can consult \Cref{segre lines example} - \Cref{SVfs curves} for examples of degrees. In these examples we used a notion of degree intuitively, which, as we now see, agrees with the one from \Cref{degree defn}.

\begin{rem}\label[rem]{complexified degree}
Since the complexified factorization curve $\psi_j^\mathbb{C}$ and the real factorization curve $\psi_j$ are given by the same polynomial expressions, their degrees agree, i.e., $\deg \psi_j^\mathbb{C} = \deg \psi_j$.
\end{rem}

\begin{rem}
Over the complex numbers, the degree of a curve $\psi_j$ is the same as the number of points, counted with multiplicities, of the intersection of $\text{Im }\psi_j$ with a generic hyperplane. Moreover, since the condition for a hyperplane to be tangent to the curve is closed, a generic hyperplane intersects $\text{Im }\psi_j$ in exactly $\deg \psi_j$ points, each with multiplicity one.
\end{rem}

\subsection{Decomposability}

\begin{defn}\label[defn]{equivalent}
Two factorization curves $\psi_i$ and $\psi_j$ are \textit{equivalent}, $\psi_i\sim\psi_j$, if they have the same image. A factorization curve $\psi_j$ is called \textit{decomposable} if its equivalence class has cardinality $\deg\psi_j$.
\end{defn}

We illustrate this definition on several examples.
Clearly, factorization curves of degree 1, i.e., factorization lines, are decomposable.
Therefore, every factorization curve of an $m$-dimensional Segre factorization structure is decomposable (see \Cref{segre lines example}). \par
In \eqref{Veronese curve} and below \Cref{segre lines example} we learnt that all $m$ factorization curves of the Veronese factorization structure $S^mW^* \hookrightarrow (W^*)^{\otimes m}$ coincide, and that its degree is $m$.
Therefore, each factorization curve $\psi_1,\ldots,\psi_m$ is decomposable, being $\psi_1(\ell) = \cdots = \psi_m(\ell) = (\ell^0)^{\otimes m}$.
Note that the values of this curve are genuine decomposable tensors, which is precisely the motivation behind the definition of a decomposable factorization curve. \par
More generally, one can observe directly in \Cref{SVfs curves} that factorization curves of a standard Segre-Veronese factorization structure are decomposable. Note that in this case, decomposability means that the variable part of such a curve, say $\ell \mapsto ins_i \left( (\ell^0)^{\otimes d_i} \otimes \Gamma_i \right)$, is a decomposable tensors, although $\Gamma_i$ does not have to be. \\

To characterise decomposability via decomposable tensors as above we make use of

\begin{lemma}\label[lemma]{decomp help}
Let $\varphi:\mathfrak{h}\to V^*$ be a complex factorization structure. The following are equivalent.
\begin{enumerate}
\item[(i)] Factorization curves $\psi_i$ and $\psi_j$ are equivalent.
\item[(ii)] There exists an invertible projective transformation $\Phi_{ji}:\mathbb{P}(V_i)\to\mathbb{P}(V_j)$ such that $\psi_i(\ell)=\psi_j(\Phi_{ji}(\ell))$.
\end{enumerate}
\end{lemma}
\begin{proof}
Clearly, $(ii)$ implies $(i)$. For the other implication we show that factorization curves are birational onto their images which makes $\psi_j^{-1}\circ\psi_i$ into a birational map $\mathbb{P}(V_i) \dashrightarrow \mathbb{P}(V_j)$. This uniquely extends to a biregular map between projective lines, and thus must be an invertible projective transformation.\par
Let $r\in\{i,j\}$. A factorization curve $\psi_r$ is in particular an injective regular map, and hence a dominant rational map onto its image which is a projective variety. Therefore, it induces an inclusion of function fields
$\psi_r^*: \mathbb{C}(\text{Im }\psi_r) \xhookrightarrow[]{} \mathbb{C}(\mathbb{P}(V_r))$.
Additionally, since the generic fibre of $\psi_r$ is finite, being a singleton, $\psi_r^*$ expresses $\mathbb{C}(\mathbb{P}(V_r))$ as a finite degree extension of $\mathbb{C}(\text{Im }\psi_r)$, the degree being the size of a generic fibre (\cite{harris2013algebraic}). Thus, since $\mathbb{C}(\mathbb{P}(V_r))=\mathbb{C}[x]$, $\psi_r^*$ is an isomorphism, and its inverse yields a dominant rational map $\Psi_r: \text{Im }\psi_r \dashrightarrow \mathbb{P}(V_r)$, the rational inverse to $\psi_r$.\par
Since $\psi_i$ and $\psi_j$ are equivalent, we have well defined dominant rational map
$\Psi_j \circ \psi_i: \mathbb{P}(V_i) \dashrightarrow \mathbb{P}(V_j)$ with the dominant rational inverse $\Psi_i \circ \psi_j$. By \Cref{curve extension}, both extend to regular maps $\Phi_{ji}$ and $\Phi_{ij}$, respectively. Furthermore, the construction implies that $\Phi_{ij} \circ \Phi_{ji}$ agrees with the identity $\text{id}_{\mathbb{P}(V_i)}$ on an open dense subset, and by \Cref{alg geom curves unique}
$\Phi_{ij} \circ \Phi_{ji} = \text{id}_{\mathbb{P}(V_i)}$. Similarly,
$\Phi_{ji} \circ \Phi_{ij} = \text{id}_{\mathbb{P}(V_j)}$. Thus, $\Phi_{ij}$ as well as $\Phi_{ji}$ are invertible regular maps, i.e., projective transformations.
\end{proof}

The following statement confirms our intuition behind decomposability of a curve  from the above examples.

\begin{thm}\label{dec char}
Let $\varphi:\mathfrak{h}\to V^*$ be a complex factorization structure of dimension $m$.
A factorization curve $\psi_j$ is decomposable if and only if there exists
$S \subset \{1,\ldots,m\}$ of cardinality $\deg \psi_j$ and a 1-dimensional subspace
$\Gamma_j \subset \bigotimes_{\substack{r=1\\r\notin S}}^m V_r^*$
such that for each $r\in S$ there exists invertible projective transformation $g_r\in\text{Hom}(\mathbb{P}(V_j^*),\mathbb{P}(V_r^*))$ such that for all $\ell\in\mathbb{P}(V_j)$ the tensor $\varphi\circ\psi_j(\ell)$ has $g_r\ell^0$ in the $r$th slot, and $\Gamma_j$ elsewhere. Clearly, $j\in S$ and $g_j=\text{id}_{\mathbb{P}(V_j^*)}$.
\end{thm}
\begin{rem}\label[rem]{up to permutation of slots}
If $S$ consists of the first $r_0$ indices, $1\leq r_0\leq m$, then we can write
\begin{align}\label{dec char eq}
\varphi\circ\psi_j(\ell)=
\left( \bigotimes_{r=1}^{r_0}
g_r\ell^0 \right)
\otimes
\Gamma_j,
\end{align}
whose linear span is $((\otimes_{r=1}^{r_0} g_r) \cdot S^{r_0} V_j^*) \otimes \Gamma_j$, where $\cdot$ represent the action of the operator $\otimes_{r=1}^{r_0} g_r$.
The general case differs from \eqref{dec char eq} by permutation of slots.
\end{rem}
\begin{proof}
The if part is obvious. For the other implication note that since $\psi_j$ is decomposable, it is $\sim$-equivalent with $\deg \psi_j$ curves indexed by $S \subset \{1,\ldots,m\}$. \Cref{decomp help} gives the existence of invertible projective transformations $G_r$, $r\in S$, of projective lines satisfying $\psi_j(\ell)=\psi_r(G_r(\ell))$. Since
$(G_r^t)^{-1}\ell^0 = (G_r\ell)^0$ we define $g_r=(G_r^t)^{-1}$, where $\cdot^t$ is the transpose. Finally, \Cref{ell_k^0 fixed} shows that $g_r\ell^0$ is at the $r$th slot of $\varphi\circ\psi_j(\ell)$ for each $\ell$, and gives the existence of $\Gamma_j$. Note, $\Gamma_j$ cannot depend on $\ell$ as it would contradict the degree. This proves the claim.
\end{proof}

Decomposability is preserved under complexification.

\begin{thm}\label{real decomp}
A factorization curve in a real factorization structure is decomposable if and only if its complexification in the complexified factorization structure is decomposable. Furthermore, the obvious real counterpart of the characterisation from \Cref{dec char} holds for real decomposable curves, as well.
\end{thm}
\begin{proof}
Let $\psi_i$ and $\psi_j$ be equivalent factorization curves in a real factorization structure. Using equalities in \eqref{real-cpx identity curves} we find that for any $\ell\in\mathbb{P}(V_i)$ there exists $\ell'\in\mathbb{P}(V_j)$ such that
\begin{align}
\psi_i^\mathbb{C}(\ell \otimes \mathbb{C})
=
\psi_i(\ell) \otimes \mathbb{C}
=
\psi_j(\ell') \otimes \mathbb{C}
=
\psi_j^\mathbb{C}(\ell' \otimes \mathbb{C}).
\end{align}
Since factorization curves are injective, $\psi_i^\mathbb{C}$ and $\psi_j^\mathbb{C}$ coincide in $\infty$-many points, and hence \Cref{alg geom curves unique} implies they are equivalent. Furthermore, since degree is preserved in a complexification (see \Cref{complexified degree}), we conclude that the complexification of a decomposable curve is a decomposable curve.\par
Suppose now $\psi_j^\mathbb{C}$ is decomposable and of degree $d$. Using \Cref{dec char} we find
\begin{gather}\label{expr}
\varphi\circ\psi_j(\ell)\otimes\mathbb{C}=
\varphi^\mathbb{C}\circ\psi_j^\mathbb{C}(\ell\otimes\mathbb{C})=
g_1\ell^0 \otimes \cdots \otimes g_d\ell^0 \otimes \langle t \rangle \otimes \mathbb{C},
\end{gather}
up to a permutation of slots as in \Cref{up to permutation of slots}. The expression \eqref{expr} clearly shows that $\psi_j$ is equivalent with $d$ curves, and hence decomposable.
\end{proof}

\begin{corollary}\label[corollary]{fc rnc}
Let $\psi_j$ be a decomposable factorization curve, and $1\leq r \leq \deg \psi_j +1$. Then, any $r$ pairwise distinct points on $\psi_j$ are linearly independent.
\end{corollary}
\begin{proof}
Since $\varphi \circ \psi_j$ is of the form \eqref{dec char eq}, it is, up to an isomorphism, a rational normal curve within its span. Therefore, any $r$ pairwise distinct points on $\varphi \circ \psi_j$ are linearly independent, and since $\varphi$ is injective, the same is true for $\psi_j$.
\end{proof}

\subsection{Quotient factorization structure}\label{quotient subsection}
To arrive at the main results of this section, we need to prove that a naturally defined quotient of a factorization structure is itself a factorization structure.
Doing so requires working with seemingly weaker structures called weak factorization structures.
These are employed because proving that their quotients are weak factorization structures is a more manageable task.
Subsequently, through an inductive argument, it is shown that weak factorization structures are, indeed, genuine factorization structures, thereby establishing the well-behaved nature of their quotients.\\

Let $\varphi(\mathfrak{h})\subset V^*$ be a factorization structure of dimension $m$, and $v$ a non-zero vector on a generic line $\lambda\in\mathbb{P}(V_i)$ such that
$\dim \left( \varphi(\mathfrak{h}) \cap \Sigma_{i,\lambda}^0 \right) = 1$.
The inclusion of short exact sequences\newline
\begin{equation}\label{quotient}
\begin{tikzcd}
0 \arrow[r] & {\varphi^{-1}(\varphi(\mathfrak{h})\cap\Sigma_{i,\lambda}^0)} \arrow[r] \arrow[d, hook] & \mathfrak{h} \arrow[r, "{P_{i,v}}"] \arrow[d, "\varphi", hook] & {\mathfrak{h}_{i,\lambda}} \arrow[r] \arrow[d, "{\varphi_{i,v}}"] & 0 \\
0 \arrow[r] & {\Sigma_{i,\lambda}^0} \arrow[r]                                                        & V^* \arrow[r, "{\rho_{i,v}}"]                                  & \hat{V}_i^* \arrow[r]                                             & 0
\end{tikzcd}
\end{equation}
defines the inclusion $\varphi_{i,v}:\mathfrak{h}_{i,\lambda}\to\hat{V}_i^*$ of the $m$-dimensional (quotient) vector space $\mathfrak{h}_{i,\lambda}$ into the tensor product of $m-1$ 2-dimensional vector spaces, which will be shown to be a factorization structure; the quotient of $\varphi:\mathfrak{h}\to V^*$ with respect to the choice $i\in\{1,\ldots,m\}$ and $\lambda\in\mathbb{P}(V_i)$.

\begin{rem}
We remark that for non-zero $v,w\in\lambda$, $\varphi_{i,v}$ and $\varphi_{i,w}$ have the same image. Thus, if they were factorization structures, as will be shown by the end of this subsection, they would be isomorphic (see \Cref{the remark}).
\end{rem}

Clearly, the inclusion $\varphi_{i,v}:\mathfrak{h}_{i,\lambda}\to\hat{V}_i^*$ is a factorization structure if and only if the intersections
\begin{gather}
\varphi_{i,v}(\mathfrak{h}_{i,v}) \cap V_1^* \otimes \cdots \otimes V_{j-1}^* \otimes \ell^0 \otimes V_{j+1}^*
\otimes \cdots \otimes V_{i-1}^* \otimes V_{i+1}^* \otimes \cdots \otimes V_m^* = \nonumber\\
\rho_{i,v} \left(\varphi(\mathfrak{h})\right) \cap \rho_{i,v} \Sigma_{j,\ell}^0
\label{intersections 1d?}
\end{gather}
are 1-dimensional for every $j\in\{1,\ldots,m\}\backslash\{i\}$ and generic $\ell\in\mathbb{P}(V_j)$. As stated, it is difficult to prove.
To make it tractable we use the aforementioned new object, weak factorization structure, defined by requiring the intersection in \eqref{wfs def condition} to be at least 1-dimensional instead of being exactly 1-dimensional.

\begin{defn}\label[defn]{wfs def}
A linear inclusion $\varphi:\mathfrak{h}\to V^*$ of an $(m+1)$-dimensional vector space $\mathfrak{h}$ is called a \textit{weak factorization structure} of dimension $m$ if for every $j\in\{1,\ldots,m\}$ and generic $\ell\in\mathbb{P}(V_j)$
\begin{align}\label{wfs defining}
\dim \left( \varphi(\mathfrak{h}) \cap \Sigma_{j,\ell}^0 \right) \geq 1
\end{align}
holds. Isomorphisms are defined in the same way as for factorization structures.
\end{defn}

As the following claim proves, a weak factorization structure satisfies the defining condition \eqref{wfs defining} not only for generic $\ell \in \mathbb{P}(V_j)$, but in fact for all $\ell \in \mathbb{P}(V_j)$.

\begin{lemma}\label[lemma]{d_k open-constant}
Let $\varphi:\mathfrak{h} \to V^*$ be a weak factorization structure.
Then, for every $j\in\{1,\cdots,m\}$, the condition \eqref{wfs defining} holds on the entire $\mathbb{P}(V_j)$, and the map $\mathbb{P}(V_j)\to\mathbb{Z}$, given by $\ell \mapsto \dim \left( \varphi(\mathfrak{h}) \cap \Sigma_{j,\ell}^0 \right)$, attains its minimal value on an open non-empty subset of $\mathbb{P}(V_j)$.
\end{lemma}
\begin{proof}
For $m=2$ this was solved directly in \Cref{fs of dim 2}.\newline
Suppose $m\geq3$. Let
\begin{align}
	U_d:=\{\ell\in\mathbb{P}(V_j):\dim\left( \varphi(\mathfrak{h})\cap\Sigma_{j,\ell}^0 \right)\geq d\}
\end{align}
be the closed sets as in \Cref{Schubert rem}.\par
The set $U_1$ is open and non-empty by definition of weak factorization structure, so $U_1=\mathbb{P}(V_k)$, and hence the condition \eqref{wfs defining} holds on the whole $\mathbb{P}(V_j)$ as claimed. Let
\begin{align}
	U^d:=U_d\backslash U_{d+1}=\{\ell\in\mathbb{P}(V_j):|\varphi(\mathfrak{h})\cap\Sigma_{j,\ell}^0|=d\}.
\end{align}
The set $U^1=U_1\backslash U_2=\mathbb{P}(V_j)\backslash U_2$ is open as $U_2$ is closed. Thus, if there exists $\ell\in\mathbb{P}(V_j)$ such that
\begin{align}\label{generic 1}
\dim \left(\varphi(\mathfrak{h})\cap\Sigma_{j,\ell}^0\right) =1,
\end{align}
i.e., $U^1\neq\emptyset$, then \eqref{generic 1} holds generically in $\ell$ as claimed.\par
However, if the set $U^1$ is empty, then $\mathbb{P}(V_j)=U_1=U_2$. Now, similarly as before, if the open set $U^2$ is non-empty, then $\dim \left(\varphi(\mathfrak{h})\cap\Sigma_{j,\ell}^0\right) =2$ generically in $\ell$.\newline
Since $\mathfrak{h}$ is a weak factorization structure this process yields the claim before $d$ exceeds $\dim(\mathfrak{h})=m+1$.
\end{proof}

Now we are ready to investigative quotients of (weak) factorization structures.

\begin{lemma}\label[lemma]{crucial lemma}
Let $\varphi:\mathfrak{h}\to V^*$ be a weak factorization structure of dimension $m$. Then, for every $i\in\{1,\ldots,m\}$ there exists an open non-empty $A_i\subset\mathbb{P}(V_i)$ such that for all $\lambda\in A_i$, every non-zero $v\in\lambda$, and every $j\in\{1,\ldots,m\}\backslash\{i\}$ there is an open non-empty $U_j\subset\mathbb{P}(V_j)$ such that for all $\ell\in U_j$:
\begin{gather}\label{dimension result}
\dim
\left(
\rho_{i,v} \left( \varphi(\mathfrak{h}) \cap \Sigma_{j,\ell}^0 \right)
\right)
\geq 1.
\end{gather}
\end{lemma}
\begin{proof}
By contradiction. Suppose there is $i\in\{1,\ldots,m\}$ such that for any open $A_i\subset\mathbb{P}(V_i)$ there is $\lambda\in A_i$, $0\neq v\in\lambda$ and $j\neq i$ such that for any open $U_j\subset\mathbb{P}(V_j)$ there exists $\ell$ with
\begin{align}\label{contradiction quotient}
\varphi(\mathfrak{h})\cap \Sigma_{j,\ell}^0
\subset
\ker \rho_{i,v}
=
\Sigma_{i,\lambda}^0.
\end{align}
There are two observations to be made. First, for such a fixed $i, A_i, \lambda$ and $j$, by suitably varying $U_j$ we find an infinite set $S_\lambda\subset\mathbb{P}(V_j)$ such that any $\ell \in S_\lambda$ satisfies \eqref{contradiction quotient}.
Secondly, replacing $A_i$ with the open set $A_i\backslash\{\lambda\}$ gives an index $j_0\neq i$, $\bar{\lambda}\in\mathbb{P}(V_i)$ s.t. $\bar{\lambda}\neq\lambda$, and the corresponding infinite set $S_{\bar{\lambda}}\subset\mathbb{P}(V_{j_0})$. Clearly, indices $j$ and $j_0$ might be different, but the freedom in choosing an open non-empty set not containing $\lambda$ allows to find two infinite sets $S_\lambda, S_{\bar{\lambda}}\subset\mathbb{P}(V_j)$ corresponding to distinct $\lambda,\bar{\lambda}\in\mathbb{P}(V_i)$ such that
\begin{align}
\varphi(\mathfrak{h})\cap \Sigma_{j,\ell}^0
\subset
\Sigma_{i,\lambda}^0,
\hspace{.2cm}
\ell \in S_\lambda,
\end{align}
and
\begin{align}
\varphi(\mathfrak{h})\cap \Sigma_{j,\ell}^0
\subset
\Sigma_{i,\bar{\lambda}}^0,
\hspace{.2cm}
\ell \in S_{\bar{\lambda}}.
\end{align}
Let $j$ be as above and $\mathcal{V}\subset\mathbb{P}(V_j)$ be the open non-empty set where $\ell \mapsto \dim \left( \varphi(\mathfrak{h}) \cap \Sigma_{j,\ell}^0 \right)$ attains its minimal value $d$ (see \Cref{d_k open-constant}). Then the closed set (see \Cref{Schubert rem})
\begin{align}\label{preimage}
\left\{
	\ell \in \mathcal{V} \hspace{.1cm}|\hspace{.1cm} \varphi(\mathfrak{h}) \cap \Sigma_{j,\ell}^0 \subset \Sigma_{i,\lambda}^0
\right\}
=
\left\{
	\ell \in \mathcal{V} \hspace{.1cm}|\hspace{.1cm} \dim \left( \varphi(\mathfrak{h}) \cap \Sigma_{j,\ell}^0 \cap \Sigma_{i,\lambda}^0 \right) \geq d
\right\}
\end{align}
contains the set $S_\lambda$, and thus equals to $\mathcal{V}$. Clearly, the same argument works for $\bar{\lambda}$. Thus,
\begin{align}
\varphi(\mathfrak{h}) \cap \Sigma_{j,\ell}^0 \subset \Sigma_{i,\lambda}^0
\hspace{1cm}\text{and}\hspace{1cm}
\varphi(\mathfrak{h}) \cap \Sigma_{j,\ell}^0 \subset \Sigma_{i,\bar{\lambda}}^0
\end{align}
for $\ell\in \mathcal{V}$, i.e.,
$\varphi(\mathfrak{h}) \cap \Sigma_{j,\ell}^0 \subset \Sigma_{i,\lambda}^0  \cap \Sigma_{i,\bar{\lambda}}^0 = 0$
as $\lambda\neq\bar{\lambda}$. Therefore $\dim \left(\varphi(\mathfrak{h}) \cap \Sigma_{j,\ell}^0 \right) =0$ generically, which contradicts $\varphi$ being a weak factorization structure.
\end{proof}

\begin{corollary}\label[corollary]{cor fs}
Let $\varphi:\mathfrak{h}\to V^*$ be a factorization structure of dimension $m$. Then, for every $i\in\{1,\ldots,m\}$ there exists an open non-empty $A_i\subset\mathbb{P}(V_i)$ such that for all $\lambda\in A$ and any $j\in\{1,\ldots,m\}\backslash\{i\}$ there is an open non-empty $\bar{U}_j\subset\mathbb{P}(V_j)$ such that for all $\ell\in \bar{U}_j$
\begin{gather}
\dim \left( \varphi(\mathfrak{h})\cap\Sigma_{j,\ell}^0 \cap \Sigma_{i,\lambda}^0 \right)=0
\end{gather}
\end{corollary}
\begin{proof}
Rank-nullity theorem together with \Cref{crucial lemma} give an open non-empty $U_j\subset \mathbb{P}(V_j)$ such that for $\ell \in U_j$
\begin{align}
\dim
\left(
\varphi(\mathfrak{h}) \cap \Sigma_{j,\ell}^0
\right)
-
\dim
\left(
\varphi(\mathfrak{h}) \cap \Sigma_{j,\ell}^0 \cap \Sigma_{i,\lambda}^0
\right)
=
\dim
\left(
\rho_{i,v} \left( \varphi(\mathfrak{h}) \cap \Sigma_{j,\ell}^0 \right)
\right)
\geq 1.
\end{align}
Intersecting $U_j$ with the open non-empty set where $\dim \left(\varphi(\mathfrak{h}) \cap \Sigma_{j,\ell}^0 \right) =1$ gives an open non-empty set $\bar{U}_j$ where the claim holds.
\end{proof}

We define the quotient of a weak factorization structure $\varphi: \mathfrak{h} \to V^*$ with respect to $i \in \{1,\ldots,m\}$ and $v \in \lambda$, $\lambda \in \mathbb{P}(V_j)$, by \eqref{quotient} to be the linear inclusion $\varphi_{i,v}: \mathfrak{h}_{i,\lambda} \to \hat{V}_i^*$.

\begin{proposition}\label[proposition]{contracted wfs is wfs}
Let $\varphi:\mathfrak{h}\to V^*$ be a weak factorization structure of dimension $m$, $i \in \{1,\ldots,m\}$, $A_i$ be as in \Cref{crucial lemma}, and $v$ a non-zero vector on $\lambda \in A_i$. Then, the quotient $\varphi_{i,v}$ of a weak factorization structure $\varphi$ with respect to $i$ and $v$ is a weak factorization structure.
\end{proposition}
\begin{proof}
We need to check if the intersection \eqref{intersections 1d?} are at least 1-dimensional generically in every slot. This follows by combining \Cref{crucial lemma} with the set-theoretical inclusion
\begin{gather}\label{function on intersection}
\rho_{i,v} \left( \varphi(\mathfrak{h}) \cap \Sigma_{j,\ell}^0 \right)
\subset
\rho_{i,v} \left(\varphi(\mathfrak{h})\right) \cap \rho_{i,v} \Sigma_{j,\ell}^0.
\end{gather}
\end{proof}

The following sufficient condition for a weak factorization structure to be a factorization structure is used to prove the main theorem of this section.

\begin{lemma}\label[lemma]{induction for qfs}
Let $\varphi:\mathfrak{h}\to V^*$ be a weak factorization structure of dimension $m\geq 3$.
Suppose that for every $i \in \{1,\ldots,m\}$ there exists distinct $\lambda_1^i, \lambda_2^i \in \mathbb{P}(V_i)$ and non-zero $v_1^i \in \lambda_1^i, v_2^i \in \lambda_2^i$ such that the quotient weak factorization structures $\varphi_{i,v_1^i}$ and $\varphi_{i,v_2^i}$ are factorization structures.
Then $\varphi$ is a factorization structure.
\end{lemma}
\begin{proof}
To ease the notation we proceed with $v=v_1^i$.
Using respectively that $\varphi_{i,v}$ is a factorization structure, the inclusion \eqref{function on intersection}, and the rank-nullity theorem for the contraction $\rho_{i,v}$, we find that for a generic $\ell\in\mathbb{P}(V_j)$, $j\neq i$,
\begin{gather}
1 =
\dim \left( 
		\rho_{i,v} \left(\varphi(\mathfrak{h})\right) \cap \rho_{i,v} \Sigma_{j,\ell}^0
	\right) \geq
\dim \left(
		\rho_{i,v} \left( \varphi (\mathfrak{h})
						\cap
						\Sigma_{j,\ell}^0
					\right)
	 \right)=\\
\dim \left(
		\varphi (\mathfrak{h})
		\cap
		\Sigma_{j,\ell}^0\right)-\label{here}\\
\dim \left(
		\varphi (\mathfrak{h})
		\cap
		\Sigma_{j,\ell}^0
		\cap
		\Sigma_{i,\lambda}^0
	\right).\label{here0}
\end{gather}
Observe that if \eqref{here0} were zero, then $\varphi$ being a weak factorization structure implies that \eqref{here} is 1, and thus proving that $\varphi$ satisfies the defining equation of a factorization structure for $j\neq i$. To show that \eqref{here0} is zero we consider contractions $\rho_{q,v_1^q}$ and $\rho_{q,v_2^q}$ for $q\neq j$ and $q\neq i$. Now, as $\varphi_{q,v_1^q}$ is a factorization structure, \Cref{cor fs} implies that the right hand side of
\begin{align}
\rho_{q,v_1^q} \left( \varphi (\mathfrak{h})
					\cap
					\Sigma_{j,\ell}^0
					\cap
					\Sigma_{i,\lambda}^0
			\right)
\subset
\rho_{q,v_1^q}\varphi (\mathfrak{h})
\cap
\rho_{q,v_1^q}\Sigma_{j,\ell}^0
\cap
\rho_{q,v_1^q}\Sigma_{i,\lambda}^0
\end{align}
is zero for appropriate generic choices of $\lambda$ and $\ell$ as explained in \Cref{cor fs}. Thus, for these values of $\ell$ and $\lambda$,
\begin{align}
\varphi (\mathfrak{h})
\cap
\Sigma_{j,\ell}^0
\cap
\Sigma_{i,\lambda}^0
\subset
\ker \rho_{q,v_1^q} = \Sigma_{q,\lambda_1^q}^0.
\end{align}
Repeating this process with $\varphi_{q,v_2^q}$ yields generic values of $\ell$ and $\lambda$ for which
\begin{align}
\varphi (\mathfrak{h})
\cap
\Sigma_{j,\ell}^0
\cap
\Sigma_{i,\lambda}^0
\subset
\Sigma_{q,\lambda_1^q}^0 \cap \Sigma_{q,\lambda_2^q}^0=0
\end{align}
as required.
We showed that for $j \neq i$ and $j \neq q$, $\dim \left( \varphi(\mathfrak{h}) \cap \Sigma_{j,\ell}^0 \right) = 1$ generically in $\ell$.
To prove the rest one permutes the roles of $i,j$ and $q$.
\end{proof}

\begin{thm}\label{wfs is fs}
Every weak factorization structure is a factorization structure.
\end{thm}
\begin{proof}
Induction on dimension $m$ of a weak factorization structure. For $m=2$, the classification of factorization structures of dimension 2 in \Cref{fs of dim 2} shows that any weak factorization structure is a factorization structure (see \Cref{2d fs lemma}).\par
Suppose now that the claim holds for weak factorization structures of dimension $m\geq2$ and let $\varphi$ be a weak factorization structure of dimension $m+1$. Using \Cref{contracted wfs is wfs} and the induction hypothesis, we find that for any $i \in \{1,\ldots,m\}$, $\lambda \in A_i$ and $v \in \lambda$, the quotients $\varphi_{i,v}$ of $\varphi$ are factorization structures. \Cref{induction for qfs} concludes that weak factorization structures of dimension $m+1$ are factorization structures.
\end{proof}
Finally,
\begin{thm}\label{fs quotient thm}
Let $\varphi:\mathfrak{h}\to V^*$ be a factorization structure of dimension $m$.
Then, for every $i \in \{1,\ldots,m\}$ there exists an open non-empty $A_i \subset \mathbb{P}(V_i)$ such that for every $\lambda \in A_i$ and non-zero $v \in \lambda$, the quotient $\varphi_{i,v}$ is a factorization structure.
Furthermore, $A_i$ is as in \Cref{crucial lemma}.
\end{thm}
\begin{proof}
A factorization structure is in particular a weak factorization structure. \Cref{contracted wfs is wfs} shows that for every $i \in \{1,\ldots,m\}$ there exists an open non-empty $A_i \subset \mathbb{P}(V_i)$ such that for every $\lambda \in A_i$ and non-zero $v \in \lambda$, the quotient $\varphi_{i,v}$ is a weak factorization structure. In turn, by \Cref{wfs is fs}, $\varphi_{i,v}$ is a factorization structure, thus proving the claim.
\end{proof}

\begin{rem}\label[rem]{quotient as contraction}
Note that the image $\varphi_{i,v} \mathfrak{h}_{i,\lambda}$ of the quotient factorization structure can be computed by \eqref{quotient} as the contraction $\rho_{i,v} \varphi(\mathfrak{h})$.
This fact will be used is subsequent subsections freely.
\end{rem}

Finally, we are ready to describe the behaviour of factorization curves and their degrees in quotient spaces.\\

Let $\psi_j^{i,\lambda}: \mathbb{P}(V_j)\to\mathbb{P}(\mathfrak{h}_{i,\lambda})$
be a factorization curve in the quotient factorization structure \eqref{quotient}, thus generically given by
\begin{align}\label{qfc0}
\varphi_{i,v}\circ\psi_j^{i,\lambda}(\ell)
 = 
\varphi_{i,v}(\mathfrak{h}_{i,\lambda}) \cap \rho_{i,v}\Sigma_{j,\ell}^0.
\end{align}
\Cref{cor fs} shows that
$\varphi(\mathfrak{h}) \cap \Sigma_{j,\ell}^0$
does not lie in
$\ker \rho_{i,v} = \Sigma_{i,\lambda}^0$
for generic $\ell\in\mathbb{P}(V_j)$, hence
\begin{equation}\label{qfc}
\rho_{i,v}
\left( \varphi(\mathfrak{h}) \cap \Sigma_{j,\ell}^0 \right)
=
\varphi_{i,v}(\mathfrak{h}_{i,\lambda}) \cap \rho_{i,v}\Sigma_{j,\ell}^0
\end{equation}
holds generically. Combining \eqref{qfc0} and \eqref{qfc} we generically find
\begin{align}
\rho_{j,v} \circ \varphi \circ \psi_j (\ell)
 = 
 \varphi_{i,v}\circ\psi_j^{i,\lambda}(\ell),
\end{align}
which, using the commutativity of \eqref{quotient} and taking the $\varphi_{i,v}$-preimage, results in the generic equality
\begin{align}
P_{i,v} \circ \psi_j(\ell) = \psi_j^{i,\lambda}(\ell).
\end{align}
\Cref{alg geom curves unique} implies that the unique extension of $P_{i,v} \circ \psi_j$, ensured by \Cref{curve extension}, and $\psi_j^{i,\lambda}$ coincide. Additionally, the injectivity of factorization curves (\Cref{curves are inj}) shows that the curve $\psi_j$ intersects $\ker P_{i,v}=\psi_i(\lambda)$ at most once. In this case, we have
\begin{align}\label{iso on curves}
P_{i,v} \circ \psi_j(\ell) = \psi_j^{i,\lambda}(\ell), \text{ where } \ell\neq\lambda,
\end{align}
otherwise the equality holds everywhere. Finally, since $\psi_j$ and $\psi_j^{i,\lambda}$ are injective, $P_{i,v}$ restricted to $\text{Im }\psi_j \backslash \{\psi_i(\lambda)\}$ is bijective.

\begin{thm}\label{degree thm}
Let $\varphi$ be a complex factorization structure, $i\neq j$, and fix $\lambda\in\mathbb{P}(V_i)$. Then
\begin{align}
\deg \psi_j^{i,\lambda}
=
\begin{cases}
\deg \psi_j -1, & \text{ if  } \hspace{.1cm} \psi_i(\lambda) \in \im \psi_j\\
\deg \psi_j, & \text{ otherwise.}
\end{cases}
\end{align}
\end{thm}
\begin{proof}
Note that there is at most one point on $\im \psi_j^{i,\lambda}$ which does not lie in the image of the restriction of $P_{i,v}$ to $\text{Im }\psi_j \backslash \{\psi_i(\lambda)\}$, depending on $\psi_i(\lambda)$ being in $\im \psi_j$. To compute $\deg \psi_j^{i,\lambda}$ we consider a generic hyperplane $H$ in $\mathbb{P}(\mathfrak{h}_{i,\lambda})$ which does not intersect $\psi_j^{i,\lambda}$ in this point, if such a point occurs, otherwise we consider any generic hyperplane $H$. Since every hyperplane in $\mathbb{P}(\mathfrak{h}_{i,\lambda})$ corresponds to a hyperplane in $\mathbb{P}(\mathfrak{h})$ through $\psi_i(\lambda)$, Bertini's theorem \cite{harris2013algebraic} applied to
$P_{i,v}: \im \psi_j \to \mathbb{P}(\mathfrak{h}_{i,\lambda})$ and $H$ shows that $\im \psi_j$ intersects $(P_{i,v})^{-1}(H)$ transversally, and therefore by Bézout's theorem they intersect in $\deg \psi_j$ points. By the choice of $H$, and because $P_{i,v}$ is bijective on $\text{Im }\psi_j \backslash \{\psi_i(\lambda)\}$, the intersection points of $H$ and $\psi_j^{i,\lambda}$ bijectively correspond to intersection points of $(P_{i,v})^{-1}(H)$ and $\psi_j$. Thus, degree remains the same, unless $\psi_i(\lambda) \in \im \psi_j$, in which case it drops by one.
\end{proof}

\subsection{Decomposable curves and Segre-Veronese factorization structures}\label{all fs are SV}
This subsection proves in \Cref{main thm} one of the main results of this article: if all factorization curves in a factorization structure $\varphi$ are decomposable, then $\varphi$ is a Segre-Veronese factorization structure. Consequently, if every factorization curve is decomposable, then all factorization structures are of Segre-Veronese type. We therefore ask

\begin{question}\label[question]{question}
Are all factorization curves decomposable?
\end{question}

\begin{rem}
The obstacle in proving that every factorization curve is decomposable is the validity of the following implication: There exists a quotient of a given factorization structure $\varphi$ such that if two curves are not equivalent in $\varphi$, then their corresponding quotient curves are not equivalent. If it were true, a simple argument, by contracting the whole factorization structure into 2-dimensional Segre while keeping track of degrees, would prove that curves are decomposable. Alternatively, an inductive argument, similar to the one in the proof of \Cref{main thm}, would do the job as well.
\end{rem}

\Cref{dec char}, \Cref{up to permutation of slots} and \Cref{real decomp} show that for a real/complex factorization structure $\varphi: \mathfrak{h} \to V^*$ with all curves decomposable, $\varphi(\mathfrak{h})$ contains, up to an isomorphism of factorization structures, the image of a Segre-Veronese factorization structure as a subspace, whose preimage under $\varphi$ is denoted here by $\mathcal{SV}$.

\begin{thm}\label{main thm}\leavevmode
\begin{enumerate}
\item[(i)] Suppose that any factorization structure of dimension $m-1$ with all curves decomposable is of Segre-Veronese type. Then, any factorization structure of dimension $m$ with all curves decomposable is of Segre-Veronese type. 
\item[(ii)] Any factorization structure with all curves decomposable is of Segre-Veronese type. In particular, such a factorization structure is the sum of spans of its factorization curves.
\end{enumerate}
\end{thm}
In the following proof, the set of 1-dimensional spaces corresponding to a factorization curve is called a curve as well.
\begin{proof}\leavevmode
\begin{enumerate}
\item[(i)]
The goal is to show $\mathcal{SV}=\mathfrak{h}$. Let $\mathfrak{h}_{i,v}$ be a quotient factorization structure as in \eqref{quotient} which, by the assumption, it is of Segre-Veronese type, where $v$ is a non-zero vector on a generic $\lambda\in\mathbb{P}(V_i)$. Note that $\ker P_{i,v} \subset \mathcal{SV}$, and that if $P_{i,v}\big|_\mathcal{SV}: \mathcal{SV} \to \mathfrak{h}_{i,v}$ were surjective, then the 3rd and 1st isomorphism theorems for vector spaces give the claim
\begin{equation}
\mathfrak{h}/\mathcal{SV}=
\frac{\mathfrak{h}/\ker P_{i,v}}{\mathcal{SV}/\ker P_{i,v}}=
\mathfrak{h}_{i,v}/\mathfrak{h}_{i,v}=
0.
\end{equation}
It remains to show $P_{i,v}\big|_\mathcal{SV}$ is surjective, or equivalently
$\rho_{i,v}\big|_{\varphi(\mathcal{SV})}: \varphi(\mathcal{SV}) \to \varphi_{i,v}(\mathfrak{h}_{i,v})$
is surjective, where $\varphi_{i,v}(\mathfrak{h}_{i,v})$ is, up to an isomorphism, of the form \eqref{SV image}.\par
Note that \eqref{iso on curves} shows that $\rho_{i,v}$ is an isomorphism from $\varphi\circ\psi_j(\ell)$ to $\varphi_{i,v}\circ\psi_j^{i,\lambda}(\ell)$ for every $j\in\{1,\ldots,m\}\backslash\{i\}$ and every $\ell\in\mathbb{P}(V_j)$ such that $\varphi\circ\psi_j(\ell)\neq\varphi\circ\psi_i(\lambda)$. Now, since $\varphi_{i,v}(\mathfrak{h}_{i,v})$ is the span of its factorization curves \eqref{standard curves}, which are rational normal curves within their spans, any point in $\varphi_{i,v}(\mathfrak{h}_{i,v})$ can be written as a linear combination of a basis lying on these curves; there is a large freedom for such a choice of basis. We choose such a basis $P_1,\ldots,P_m$ so that none of these vectors lie on the lines where $\rho_{i,v}$ is not an isomorphism in the above sense. We lift this basis to vectors lying on curves in $\mathcal{SV}$ via the corresponding restriction of $\rho_{i,v}$, resulting in $m$ linearly independent vectors, thus proving surjectivity. Additionally, the $m$-dimensional space spanned by these is linearly independent from $\varphi \circ \psi_i(\lambda)$, hence providing another proof that $\mathcal{SV}=\mathfrak{h}$.
\item[(ii)]
Induction with respect to the dimension of factorization structure. The base case $m=2$ follows from the classification (\Cref{fs of dim 2}). The rest follows from part (i) of this theorem.
See also \Cref{SVfs curves}.
\end{enumerate}
\par

\end{proof}

In \Cref{k=2 example} we observed that defining tensors $\Gamma_1$ and $\Gamma_2$ of a standard Segre-Veronese factorization structure corresponding to a partition $m=d_1+d_2$ are symmetric. This observation generalises as follows.
\begin{proposition}\label[proposition]{symmetric SV}
If $\{\Gamma_j\}_{j=1}^k$ define a standard Segre-Veronese factorization structure with a partition $m=d_1+\cdots+d_k$, then $\Gamma_j
\subset
\bigotimes_{\substack{i=1\\i\neq j}}^k
S^{d_i}W_i^*$, $j=1,\ldots,k$.
\end{proposition}
\begin{proof}
Induction on $k$. The case $k=1$ is trivial and $k=2$ was solved in \eqref{k=2} above.

Suppose the claim holds for $k\geq2$, fix $j\in\{1,\ldots,k+1\}$ and set $d_0=0$.
The idea is to form $d_j$ quotients iteratively in $(d_1+\cdots+d_{j-1}+1)$-st slot in such a way that \Cref{fs quotient thm} is applicable, i.e., so that each quotient is a factorization structure.
This contracts grouped $j$-slots, leaving $\Gamma_j$ behind as we will see.
Note that after the first quotient, the $(d_1+\cdots+d_{j-1}+2)$-nd slot becomes $(d_1+\cdots+d_{j-1}+1)$-st slot, and so on.
Clearly in each step, one can choose $v$ and $\lambda$ as in \Cref{fs quotient thm} so that the corresponding quotient is a factorization structure.
While a complete tracking of indices, $v$'s and $\lambda$'s is possible, it contributes no essential understanding and significantly complicates the presentation, and will therefore be bypassed.
We denote the composition of all $d_j$ quotient maps by $\rho$ (see \Cref{quotient as contraction}) and apply it on
\begin{align}\label{Gammas are symmetric}
\varphi(\mathfrak{h})=
\sum_{i=1}^{k+1}
ins_i
\left(
	S^{d_i}W_i^*
	\otimes
	\Gamma_i
\right)
\end{align}
which results in
\begin{align}\label{Gammas are symmetric contracted}
\Gamma_j
+
\sum_{\substack{i=1\\i\neq j}}^{k+1}
ins_i
\left(
	S^{d_i}W_i^*
	\otimes
	\rho \Gamma_i
\right).
\end{align}
By \Cref{main thm},
\begin{align}
\Gamma_j
\in
\sum_{\substack{i=1\\i\neq j}}^{k+1}
ins_i
\left(
	S^{d_i}W_i^*
	\otimes
	\rho \Gamma_i
\right)
\end{align}
which together with the induction hypothesis,
\begin{align}
\rho \Gamma_i
\subset
\bigotimes_{\substack{b=1\\b\neq i,j}}^{k+1}
S^{d_b}W_b^*,
\hspace{.2cm}
i\in\{1,\ldots,k+1\}\backslash\{j\},
\end{align}
give the claim.
\end{proof}
\begin{corollary}\label[corollary]{decomposable SV tensors}
If a Segre-Veronese factorization structure of dimension $m=d_1+\ldots+d_k$ is determined by 1-dimensional spaces
\begin{align}
\Gamma_j
=
\bigotimes_{\substack{r=1\\r\neq j}}^k
\bigotimes_{p=1}^{d_r}
a_j^{r,p},
\hspace{.2cm}
a_j^{r,p} \subset W_r^*,
\end{align}
then $a_j^{r,1}=
\cdots=
a_j^{r,d_r} =: a_j^r$,
i.e.,
\begin{align}
\Gamma_j=
\bigotimes_{\substack{r=1\\r\neq j}}^k (a^r_j)^{\otimes d_r}.
\end{align}
\end{corollary}

The following lemma shows that $d_1,\ldots,d_k$ in the partition $m = d_1 + \cdots + d_k$ corresponding to a Segre-Veronese factorization structure are invariants.

\begin{lemma}\label[lemma]{SV iso class} $ $
\begin{itemize}
\item[(i)] Any two orderings of $d_1,\ldots,d_k$ in the partition of $m$ give isomorphic standard Segre-Veronese factorization structures provided $\Gamma_1,\ldots,\Gamma_k$ are fixed. Furthermore, the isomorphism is given by permuting grouped slots.
\item[(ii)] Standard Segre-Veronese factorization structures corresponding to distinct partitions cannot be isomorphic for any choice of  $\Gamma_1,\ldots,\Gamma_k$.
\end{itemize}
\end{lemma}
\begin{proof}$ $
\begin{itemize}
\item[(i)] These are isomorphic via the braiding map $\sigma$ (see \Cref{fs def}) which permutes groups of slots corresponding to the partition.
\item[(ii)] Positive integers $d_1,\ldots,d_k$ determining a factorization structure can be viewed as degrees of factorization curves (see \Cref{SVfs curves}). The claim is proved by observing that these are invariant under isomorphisms of factorization structures.
\end{itemize}
\end{proof}

\begin{rem}\label[rem]{partition assignment}
This lemma ensures a well-defined assignment from isomorphism classes of Segre-Veronese factorization structures onto finite subsets of positive integers $\{d_1,\ldots,d_k\}$. Observe that this map classifies product Segre-Veronese factorization structures. In the following subsection we use 
the map to describe the classification of decomposable Segre-Veronese factorization structures.
\end{rem}

\subsection{Characterisation of decomposable factorization structures}\label{classification}
This subsection proves another important result of this article.
It uses products of factorization structures and the correspondence from \Cref{correspondence} below to characterise decomposable Segre-Veronese factorization structures in \Cref{decom fs thm} as iterative products of Veronese factorization structures.
This structural result is the first step towards the classification of Segre-Veronese factorization structures and, in addition, it is allows to solve the extremality equation for associated separable Kähler geometries uniformly, as outlined in \Cref{dg motivation}.

\begin{rem}\label[rem]{correspondence}
We explain a correspondence between isomorphism classes of decomposable Segre-Veronese factorization structures and pairs consisting of an isomorphism class of a decomposable Segre factorization structure and a set of positive integers.\par
Starting with a decomposable Segre-Veronese factorization structure $\varphi$ of dimension $m$, we have a partition $m=d_1+\cdots+d_k$ and 1-dimensional subspaces
\begin{align}\label{dec SV tensors}
\Gamma_j=
\bigotimes_{\substack{r=1\\r\neq j}}^k (a_j^r)^{\otimes d_r}.
\end{align}
Then, for each $r\in\{1,\ldots,k\}$ we take, in grouped $r$-slots, $d_r-1$ (order and choice independent) factorization structure quotients as in \Cref{fs quotient thm} of $\varphi$ to get a decomposable Segre factorization structure of dimension $k$
\begin{align}\label{Segre correspondece tensors}
\sum_{j=1}^{k}
\text{ins}_j
\left(
W_j^* \otimes
\bigotimes_{\substack{r=1\\r\neq j}}^k a_j^r
\right),
\end{align}
while remembering the partition $\{d_1,\ldots,d_k\}$ (see also proof of \Cref{symmetric SV} and \Cref{quotient as contraction} for more details on quotients). The isomorphism class of \eqref{Segre correspondece tensors} together with $\{d_1,\ldots,d_k\}$ form the pair from the correspondence. Observe that every factorization structure isomorphic with $\varphi$ corresponds to the same pair. This gives a well-defined assignment.\par
Conversely, starting with a decomposable Segre factorization structure, say \eqref{Segre correspondece tensors}, we use the set $\{d_1,\ldots,d_k\}$ and the Veronese embedding $W_r^*\to (W_r^*)^{\otimes d_r}$, $v\mapsto v^{\otimes d_j}$, in each slot $r\in\{1,\ldots,k\}$ to obtain an inclusion of vector spaces
\begin{align}\label{Segre correspondece}
\sum_{j=1}^{k}
\text{ins}_j \left( S^{d_j} W_j^*\otimes \Gamma_j \right)
\hookrightarrow
\bigotimes_{j=1}^k S^{d_j}W_j^*
\hookrightarrow
\bigotimes_{j=1}^k (W_j^*)^{\otimes d_j},
\end{align}
where $\Gamma_j$ are now as in \eqref{dec SV tensors}. The image has dimension $m+1$ as can be seen by taking consecutive factorization structure quotients, whose kernel is 1-dimensional, as above. This determines an assignment on isomorphism classes which is inverse to the one describe above.
\end{rem}

\begin{defn}\label[defn]{full-product defn}
A decomposable Segre factorization structure of dimension $m$, $m\geq2$, admits a \textit{full-product in $j$th slot} if it equals to
\begin{align}\label{full-product}
\text{ins}_j
\left(
q\otimes Q
+
V_j^* \otimes \Gamma_j
\right),
\end{align}
where $Q$ is a decomposable Segre factorization structure of dimension $m-1$ which admits a full-product in $r$th slot for some $r\in\{1,\ldots,m\}\backslash\{j\}$, $V_j^*$ is a 2-dimensional vector space, and $q\subset V_j^*$ and $\Gamma_j\subset Q$ are 1-dimensional subspaces. A (decomposable Segre) factorization structure of dimension 1, being a 2-dimensional vector space, admits a full-product by definition. We say that a decomposable Segre factorization structure admits a \textit{full-product} if it admits a full-product in $j$th slot for some $j$.
\end{defn}

\begin{rem}
A decomposable Segre-Veronese factorization structure admitting a full-product can be defined similarly by substituting $V_j^*$ in \eqref{full-product} by $S^{d_j}W_J^*$. However, we use the correspondence from \Cref{correspondence} and say that a decomposable Segre-Veronese factorization structure admits a full-product if its corresponding decomposable Segre factorization structure admits a full-product. Note that these two ways of defining full-product are equivalent.
\end{rem}

We remark that in \Cref{full-product defn} we use without loss of generality the identification of a factorization structure with its image (see \Cref{the remark}).\par
One can consult \Cref{full-product example} for an example of a full-product decomposable Segre-Veronese.

\begin{example}\label[example]{full-product example}
We illustrate how a decomposable Segre admitting a full-product is built using inductive products of 1-dimensional factorization structures and outline the shape of defining tensors.\par
A 2-dimensional (decomposable) Segre factorization structure is a product of two 1-dimensional factorization structures, and hence admits full-product in both slots. Forming a product of this 2-dimensional Segre with a 1-dimensional factorization structure yields
\begin{align}\label{3fs}
\left( V_1^* \otimes \gamma_2 + \gamma_1 \otimes V_2^* \right)
\otimes \gamma +
\lambda \otimes V_3^*,
\end{align}
where $\lambda \subset V_1^* \otimes \gamma_2 + \gamma_1 \otimes V_2^*$ is assumed to be decomposable, and hence by \Cref{tensors split} either $\lambda=a\otimes \gamma_2$ or $\lambda=\gamma_1 \otimes b$ for some 1-dimensional $a\subset V_1^*$ or $b\subset V_2^*$. We note that regardless of the choice of $\lambda$, \eqref{3fs} admits full-products in at least two distinct slots; one being the 3rd slot and the other is the 1st or 2nd depending on the choice of $\lambda$. Note, if $\lambda=\gamma_1\otimes\gamma_2$, then the full-product exists in all slots which recovers the product Segre factorization structure. Forming yet another product
\begin{align}\label{4fs}
\left(\left( V_1^* \otimes \gamma_2 + \gamma_1 \otimes V_2^* \right)
\otimes \gamma +
\lambda \otimes V_3^*\right)
\otimes \delta
+\pi\otimes V_4^*
\end{align}
to make the pattern more visible, we see that again for any admissible choices of $\lambda$ and $\pi$, \eqref{4fs} admits at least two and at most four full-products. Observe in \eqref{4fs}, that three summands belong into $\Sigma_{4,\delta^0}^0$. More importantly, for any choice, $\pi$ decomposes so that another three summands belong to $\Sigma_{r,\tau^0}^0$ for some $r\in\{1,2,3\}$ and $\tau\in\mathbb{P}(V_r)$.
\end{example}

In general, it is plain to see that $m-1$ summands in a full-product \eqref{full-product} lie in $\Sigma_{j,q^0}^0$. The following lemma shows that there are another $m-1$ summands with a similar property. This helps us to establish the main claim of this subsection that decomposable Segre-Veronese factorization structures are given by full-products.

\begin{lemma}\label[lemma]{share a slot lemma}
Let a decomposable Segre factorization structure of dimension $m$ admit a full-product in the $j$th slot. Then, there exist $r\neq j$, and $\lambda\in\mathbb{P}(V_r)$, such that for all $i \neq r:$
\begin{align}
\text{ins}_i(V_i^*\otimes\langle\Gamma_i\rangle)
\subset
\Sigma_{r,\lambda}^0.
\end{align}
\end{lemma}
\begin{proof}
Induction on the dimension of a factorization structure. The base case, $m=2$, is obvious. Suppose the statement holds in dimension $m-1$, $m\geq3$, and write a decomposable Segre factorization structure $\varphi(\mathfrak{h})$ of dimension $m$ which admits a full-product in the $j$th slot as
\begin{align}\label{original fs}
\varphi(\mathfrak{h})=
\text{ins}_j
\left(
q\otimes Q
+
V_j^* \otimes \Gamma_j
\right),
\end{align}
where $\Gamma_j \subset Q$, $q\subset V_j^*$ and $Q$ is a decomposable Segre factorization structure of dimension $m-1$, which admits a full-product in $r$th slots for some $r\in\{1,\ldots,m\}\backslash\{j\}$. In particular,
\begin{align}\label{Q}
Q=
\text{ins}_r
\left(
	\lambda \otimes P
	+
	V_r^*\otimes \pi
\right),
\end{align}
for some decomposable $\pi \subset P$ and $\lambda \subset V_r^*$, where $P$ is a decomposable Segre factorization of dimension $m-2$ which admits a full-product.\par
Since \eqref{original fs} is a decomposable factorization structure, in particular $\Gamma_j$ is decomposable, \Cref{tensors split} implies that 
either
\begin{align}\label{former}
\Gamma_j =  \text{ins}_r \left (\lambda \otimes \Lambda \right) 
\hspace{.5cm}\text{for some }
\Lambda \subset P,
\end{align}
or
\begin{align}\label{latter}
\Gamma_j = \text{ins}_r \left( \Pi \otimes \pi \right) 
\hspace{.5cm}\text{for some }
\Pi \subset V_r^*.
\end{align}\par
In \eqref{former} case, it is clear that
\begin{align}
\text{ins}_j\left(V_j^*\otimes\Gamma_j\right)
=
\text{ins}_j\left(V_j^*\otimes \text{ins}_r \left( \lambda \otimes \Lambda \right) \right) 
\leq
\varphi(\mathfrak{h})
\end{align}
and another $m-2$ summands sitting in
\begin{align}
\text{ins}_j \left( q\otimes \text{ins}_r\left(\lambda\otimes P\right)\right) \leq \varphi(\mathfrak{h})
\end{align}
lie in $\Sigma_{r,\lambda^0}^0$.\par
For \eqref{latter} case we use the induction hypothesis stating that $\text{ins}_r\left(V_r^*\otimes \pi\right) \leq Q$ and another $m-3$ summands in $\text{ins}_r\left(\lambda\otimes P\right) \leq Q$ lie in $\Sigma_{i,\mu}^0$ for some $i\in\{1,\ldots,m-1\}$ and $\mu\in\mathbb{P}(V_i)$. Clearly, this gives the claim.
\end{proof}

\begin{lemma}\label[lemma]{Segre induction}
Let every decomposable Segre factorization structure of dimension $m-1$ admits a full-product. Then every decomposable Segre factorization structure of dimension $m$ admits a full-product.
\end{lemma}
We note that the proof of this lemma uses less assumptions than required in the statement. However, the stronger statement suffices for proving our end-result, and reveals a rigidity of decomposable Segre factorization structures as the proof is by induction. A similar situation occurred in \Cref{induction for qfs} when proving that every weak factorization structure is a factorization structure.
\begin{proof}
Note that if we would know that $m-1$ summands in $\varphi(\mathfrak{h})$ lie in $\Sigma_{j,q^0}^0$ for some $j\in\{1,\ldots,m\}$, then
\begin{align}\label{product?}
\varphi(\mathfrak{h})=
\text{ins}_j\left( q \otimes Q + V_j^* \otimes \Gamma_j \right)
\end{align}
must be a product. Indeed, a factorization structure quotient of $\varphi(\mathfrak{h})$ in $j$th slot gives an $m$-dimensional vector space, $Q+\Gamma_j$, and \Cref{main thm} shows $\Gamma_j \subset Q$. Now we would use the induction hypothesis for $Q$, saying that $(m-1)$-dimensional decomposable Segre factorization structures admit a full-product, which shows that \eqref{product?} admits a full-product. Thus we are left to prove that there exists an index $j$ and $q\in\mathbb{P}(V_j^*)$ such that $m-1$ summands belong to $\Sigma_{j,q^0}^0$.\par
Fix any $j\in\{1,\ldots,m\}$, and let $\tilde{Q}$ be the image of the quotient factorization structure of $\varphi(\mathfrak{h})$ in $j$th slot with respect to some $v$ and $\lambda$ (see \Cref{fs quotient thm} and \Cref{quotient as contraction}). We have $\Gamma_j \subset \tilde{Q}$, and by assumptions $\tilde{Q}$ is full-product. In particular
\begin{align}
\tilde{Q}=
\text{ins}_r
\left(
	\lambda \otimes P
	+
	V_r^*\otimes \pi
\right),
\end{align}
and \Cref{tensors split} implies that $\Gamma_j$ decomposes either as \eqref{former} or \eqref{latter}. We proceed similarly to the lemma above. In the former case it is immediate that $m-1$ summands in $\varphi(\mathfrak{h})$ lie in $\Sigma_{r,\lambda^0}^0$, while for the latter case we apply \Cref{share a slot lemma} to $\tilde{Q}$ and conclude the proof.
\end{proof}

Finally, the following theorem completely characterises decomposable Segre-Veronese factorization structures.

\begin{thm}\label{decom fs thm}
Every decomposable Segre-Veronese factorization structure admits a full-product.
\end{thm}
\begin{proof}
We use the correspondence from \Cref{correspondence} and its compatibility with full-products to reduce the statement to: Every decomposable Segre factorization structure admits a full-product. We prove this claim by induction on dimension. The base case holds trivially as any Segre factorization structure of dimension 2 is a full-product. For the induction hypothesis suppose that every decomposable Segre factorization structure of dimension $m-1$ admits a full-product, $m\geq3$. \Cref{Segre induction} gives the claim.
\end{proof}

\begin{rem}
Observe that the number of ways in which a decomposable Segre-Veronese factorization structure is a full-product is an invariant. \Cref{share a slot lemma} implies that there are always at least two ways. Factorization structures corresponding to only two full-products are the most complicated ones. The other extreme, when a full-product exists in each slot, corresponds to the product Segre-Veronese factorization structure. \Cref{full-product example} gives a recipe how to build decomposable Segre(-Veronese) factorization structures with prescribed number of full-products.
\end{rem}

\begin{rem}
We remark that ideas from this subsection can be directly adapted to more general factorization structures, whose defining tensors are of the form
\begin{align}
\Gamma_j
=
\bigotimes_{\substack{i=1\\ i\neq j}}^k
\gamma_i,
\end{align}
where $\gamma_i \in S^{d_i}W^*_i$.
\end{rem}

\section{Compatible cones and polytopes}\label{s3}

This section studies convex polyhedral cones $\sigma$ whose projectivised normals $n_1,\ldots,n_r$ lie on factorization curves of a fixed factorization structure.
These cones, along with their duals and sections, are called compatible with the factorization structure, and their construction is given in \Cref{s30}.
The impact of factorization structures on polyhedral geometry is demonstrated in \Cref{Ver cone is simplicial}, which proves that polytopes compatible with the Veronese factorization structure are simple. \par
Facets and faces of $\sigma$ lie on the annihilators $n_1^0,\ldots,n_r^0$ and their intersections, respectively.
Remarkably, the compatibility ensures that these admit elegant and explicit descriptions within the framework of factorization structures, as shown in \Cref{faces} - one of the main results of this section.
Its proof relies on quotients of factorization structures, a technically demanding achievement from \Cref{s2}. \par

Building on this, \Cref{s31} constructs cones compatible with the product Segre-Veronese factorization structure.
Its important outcome is a generalisation of Gale's evenness condition (\Cref{Gale}), which characterises facet-defining hyperplanes of such cones.
Combined with \Cref{faces}, this enables explicit description of all facet-determining linear spaces.
Moreover, we observe that the generalised Gale's evenness condition can be adapted for general compatible cones and polytopes. \par
Finally, we reinterpret Vandermonde identities via the Veronese factorization structure, providing a blueprint for extending such identities to arbitrary factorization structures. These results yield explicit examples of Delzant and rational Delzant compatible polytopes, paving the road for their construction in general. \\

After we recall basics of cones, we define compatible cones and polytopes, and provide examples. Here, polytopes are always compact and convex.\par
A \textit{convex polyhedral cone} $\sigma$ in an $(m+1)$-dimensional vector space $\mathfrak{h}^*$ generated by vectors $v_1,\ldots,v_r$ is the set of their non-negative linear combinations. For the rest of this subsection we assume that none of $v_j$ is in the relative interior of the cone. Geometrically, $\sigma$ contains convex combinations, hence is piecewise linear, and thus can equivalently be viewed as the intersection of closed half-spaces. Dually, the latter correspond to rays in $\mathfrak{h}$ and, in fact, these generate the \textit{dual cone}
\begin{equation}
\sigma^\vee=\{v\in \mathfrak{h}\hspace{.1cm}|\hspace{.1cm} \langle \alpha,v \rangle \geq 0,\hspace{.1cm} \forall \alpha\in \sigma\}
\end{equation}
of $\sigma$. Therefore, $\sigma^\vee$ is a convex polyhedral cone as well. One the other hand, the rays determined by $v_1,\ldots,v_r\in\mathfrak{h}^*$ give rise to closed half-spaces in $\mathfrak{h}$ which intersect in $\sigma^\vee$, and hence $(\sigma^\vee)^\vee=\sigma$. Observe that an oriented hyperplane $H_v\subset\mathfrak{h}^*$ given by $v\in\mathfrak{h}$ is a \textit{supporting hyperplane} of $\sigma$, i.e., $\langle v, \sigma\rangle\geq0$, if and only if $v\in\sigma^\vee$. Finally, every $v\in\sigma^\vee$ determines a \textit{face} $H_v\cap\sigma$ of $\sigma$; 1-dimensional faces are called \textit{extremal rays} and codimension one faces are \textit{facets}. For example, rays generated by $v_1,\ldots,v_r$ are extremal rays of $\sigma$, and determine facets of $\sigma^\vee$. A hyperplane supporting $\sigma$ which gives rise to a facet is called \textit{facet-supporting hyperplane}.\par
This work is exclusively concerned with cones whose affine hyperplane sections are polytopes. As \Cref{pointed lemma} shows, such cones are \textit{pointed}, i.e., contain no non-trivial subspace. Note that if a cone $\sigma\subset\mathfrak{h}^*$ has strictly less than $\dim \mathfrak{h}$ facets, the half-spaces defining $\sigma$ intersect in a non-trivial subspace. Thus, a pointed cone has at least as many facets as the ambient dimension is.
\begin{lemma}$ $ \label[lemma]{pointed lemma}
Let $\sigma$ be a convex polyhedral cone in a vector space $\mathfrak{h}^*$. Then,
\begin{enumerate}
\item[(i)]
$\sigma$ is pointed if and only if $\dim\sigma^\vee=\dim\mathfrak{h}$, where $\dim\sigma^\vee$ is the dimension of the smallest linear subspace of $\mathfrak{h}$ containing $\sigma^\vee$.
\item[(ii)]
$\sigma$ is pointed  if and only if it admits an affine hyperplane section given by $\epsilon\in\mathfrak{h}$, called an \textit{affine chart}, such that the set $\sigma\cap\{\epsilon=1\}$ is a polytope. All such affine charts are given by $\epsilon\in\text{Int}(\sigma^\vee)$.
\end{enumerate}
\end{lemma}
\begin{proof}$ $
\begin{enumerate}
\item[(i)]
If $\dim\sigma^\vee < \dim \mathfrak{h}$, then $\sigma^\vee$ is contained in a proper subspace $U$ of $\mathfrak{h}$, and $0\neq U^0\subset \sigma$, i.e., $\sigma$ is not pointed. On the other hand, if the largest linear subspace $\sigma\cap\{-\sigma\}$ in $\sigma$ is non-trivial, then supporting hyperplanes of $\sigma$ lie in the annihilator $(\sigma\cap\{-\sigma\})^0$, and hence $\dim\sigma^\vee < \dim \mathfrak{h}$. 
\item[(ii)]
Let $v_1,\ldots,v_r$ be generators of $\sigma$. Note that $\epsilon$ from the interior of $\sigma^\vee$ supports $0\in\sigma$ and has the generators of $\sigma$ on its positive side. Hence the convex hull of $v_j/\langle v_j, \epsilon \rangle$, $j=1,\ldots,r$, is the convex polytope $\sigma\cap\{\epsilon=1\}$. For the other implication, if $\sigma\cap\{\epsilon=1\}$ is a convex polytope, in particular a bounded set, then the affine hyperplane $\epsilon=1$ intersects every ray generated by $v_1,\ldots,v_r$ transversally. Therefore,  $\langle\epsilon,v_j\rangle>0$, $j=1,\ldots,r$, and hence $\sigma$ cannot contain a non-trivial linear subspace.
\end{enumerate}
\end{proof}

\subsection{Compatibility in general}\label{s30}
\begin{defn}\label[defn]{compatible polytope}
A full-dimensional and pointed convex polyhedral cone in $\mathfrak{h}$ is called \textit{compatible} with a factorization structure $\varphi : \mathfrak{h} \to V^*$ if its projectivised edges lie on factorization curves of $\varphi$. A convex polytope is called \textit{compatible} with a factorization structure $\varphi$ if it is a section of a cone $\sigma$ whose dual $\sigma^\vee$ is compatible with $\varphi$.
\end{defn}
To rephrase, a convex polytope is compatible with a factorization structure if it is full-dimensional, and is a section a pointed convex polyhedral cone whose projectivised normals lie on factorization curves.\\

We exemplify cones and polytopes compatible with 2-dimensional factorization structures, originally found in \cite{apostolov2015ambitoric}.
To keep our cartoons uncomplicated we discuss projectivised versions of these cones/polytopes, but the reader is strongly encouraged to work out 3-dimensional polyhedral geometry according to \Cref{compatible polytope}. For more details see \cite{brandenburg2024}.

\begin{example}\label[example]{segre polyotpes}
\Cref{2b} displays the images of two factorization lines/curves $\psi_1$ and $\psi_2$ in 2-dimensional projective space $\mathbb{P}(V_1^* \otimes \Gamma_1 + \Gamma_2 \otimes V_2^*)$, associated with 2-dimensional Segre factorization structure (\Cref{segre lines example}), together with their intersection point $\Gamma_2 \otimes \Gamma_1$ and a choice of points $a_i,b_i \in \im \psi_i$, $i=1,2$. \par
Under the projective duality, \Cref{2a} shows lines and points arrangement in $\mathbb{P}(( V_1^* \otimes \Gamma_1 + \Gamma_2 \otimes V_2^* )^*)$: the line $(\Gamma_2 \otimes \Gamma_1)^0$, dual to the point $\Gamma_2 \otimes \Gamma_1$, with two marked points $(\im \psi_1)^0$ and $(\im \psi_2)^0$, dual to the lines $\im \psi_1$ and $\im \psi_2$, and lines $(a_i)^0, (b_i)^0$, dual to $a_i,b_i$, passing through the point $(\im \psi_i)^0$, $i=1,2$. \par
The blue region is the projectivisation of a 2-dimensional polytope, whose de-projectivisation is a polytope compatible with 2-dimensional Segre factorization structure, since its projectivised normals $a_i,b_i$, $i=1,2$, lie on factorization curves.
The ambiguity in the definition of de-projectivisation comes from the fact that these are really meant to be 3-dimensional pictures of planes, lines and of a cone section rather than their projectivisations.\\

\begin{figure}[h]
\centering
\begin{subfigure}{0.49\textwidth}
	\centering
	\begin{tikzpicture}
	\node at (0,0) {\includegraphics[width=9cm]{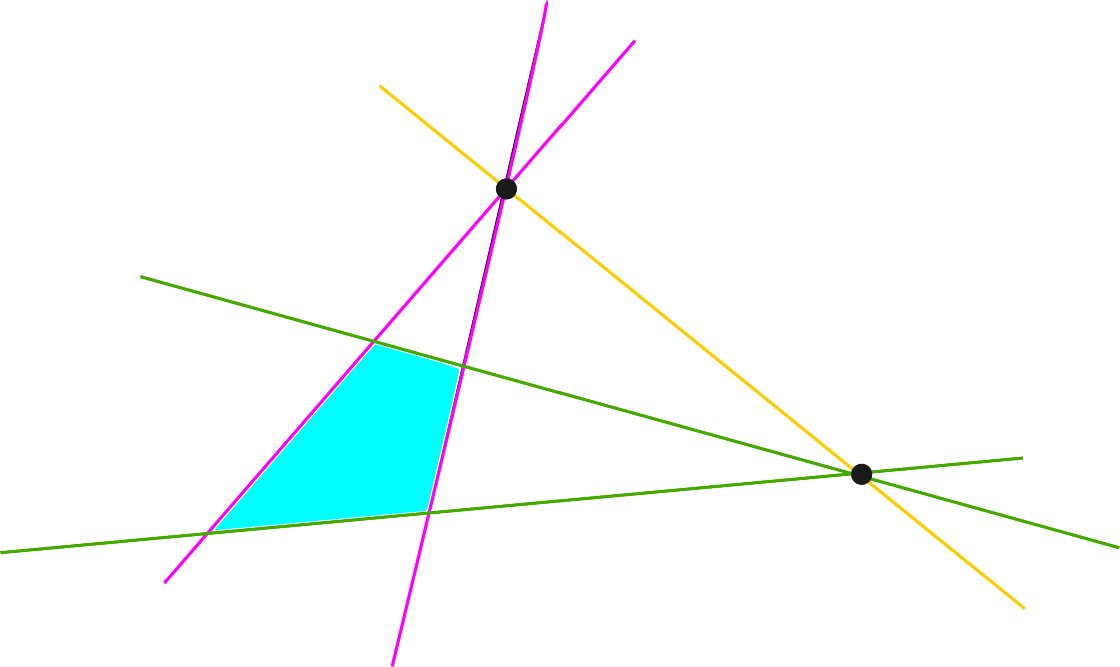}};
	\node at (1.5,.5) {$(\Gamma_2 \otimes \Gamma_1)^0$};
	\node at (0,3) {$(a_1)^0$};
	\node at (1,2.6) {$(b_1)^0$};
	\node at (-1.4,1) {$(\im \psi_1)^0$};
	\node [rotate=69] at (2,-2) {$(\im \psi_2)^0$};
	\node at (-3.7,.7) {$(b_2)^0$};
	\node at (-4.1,-1.3) {$(a_2)^0$};
	\end{tikzpicture}
	\caption{projectivised compatible polytope}
	\label{2a}
\end{subfigure}
\begin{subfigure}{0.49\textwidth}
\hspace{1cm}
	\begin{tikzpicture}
	\node at (0,0) {\includegraphics[width=5.5cm]{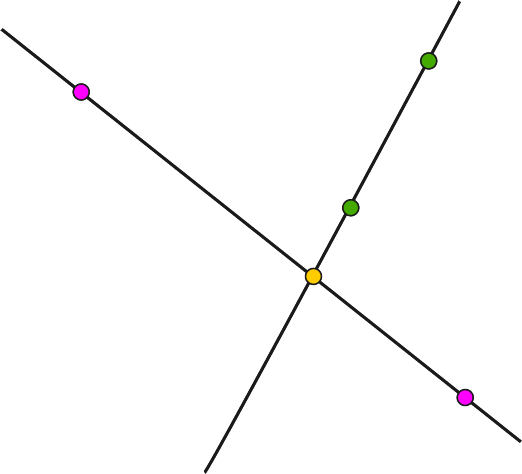}};
	\node at (2.8,2.5) {$\im \psi_2$};
	\node at (1.4,2) {$b_2$};
	\node at (.6,.5) {$a_2$};
	\node at (-2.1,2.5) {$\im \psi_1$};
	\node at (-2,1.1) {$a_1$};
	\node at (2,-2.1) {$b_1$};
	\node at (-.5,-.5) {$\Gamma_2 \otimes \Gamma_1$};
	\end{tikzpicture}
\hspace{2cm}
\caption{Segre factorization structure}
\label{2b}
\end{subfigure}
\caption{} \label{fig2}
\end{figure}

If fact, every 2-dimensional polytope with 4 edges is compatible with a 2-dimensional Segre factorization structure.
Indeed, viewing the polytope as an affine section of a cone $\sigma$, the dual cone $\sigma^\vee$ has four extremal rays lying on four 1-dimensional spaces, which are the annihilators of the planes determining facets of $\sigma$.
These four 1-dimensional spaces determine two planes $\Pi_1$ and $\Pi_2$ in three possible ways.
In all three cases, $\Pi_1 + \Pi_2$ is the ambient 3-dimensional space, and, after the projectivisation, we obtain 2-dimensional projective space with two (intersecting) lines.
To complete the argument we need to find a linear isomorphism $\Phi: \Pi_1+\Pi_2 \to V_1^* \otimes \Gamma_1 + \Gamma_2 \otimes V_2^*$ (see the definition of isomorphism in \Cref{fs def}) sending the distinguished planes to the distinguished planes, which is trivial.
We found that projectivised normals of $\sigma$ lie on factorization curves, therefore the polytope is compatible.
Said differently, a Segre factorization structure can be fit onto the vector space $\Pi_1 + \Pi_2$ so that the polytope is compatible with it.
\end{example}

\begin{example}
\Cref{3b} illustrates the quadric $\im \psi$ with four marked points $a_1,\ldots,a_4$ in the 2-dimensional projective space $\mathbb{P}(S^2W^*)$ associated to the 2-dimensional Veronese factorization structure.
The quadric represents the factorization curve $\psi_1 = \psi_2 =: \psi$, being the rational normal curve of degree 2, a quadric and a conic too.
Dually, \Cref{3a} shows lines $(a_1)^0,\ldots,(a_4)^0$ dual to points $a_1,\ldots,a_4$, which are tangent to the dual quadric $(\im \psi)^*$.
The orange region is the projectivisation of a 2-dimensional polytope which is compatible with 2-dimensional Veronese factorization structure.

\begin{figure}[h]
\centering
\begin{subfigure}{0.49\textwidth}
	\centering
	\begin{tikzpicture}
	\node at (0,0) {\includegraphics[width=7cm]{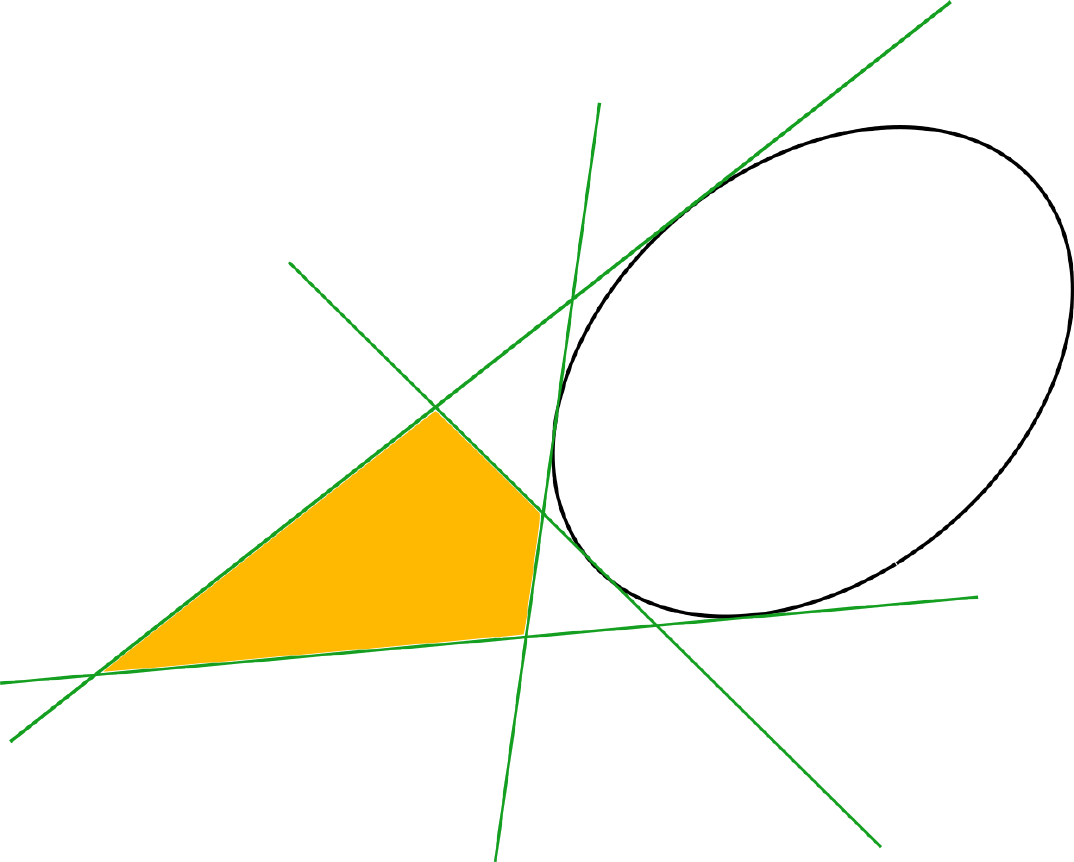}};
	\node at (2.5,1.2) {$(\im \psi)^*$};
	\node at (.5,2.6) {$(a_2)^0$};
	\node at (-2,1.5) {$(a_3)^0$};
	\node at (3.2,3) {$(a_1)^0$};
	\node at (-4.1,-1.5) {$(a_4)^0$};
	\end{tikzpicture}
	\caption{projectivised compatible polytope}
	\label{3a}
\end{subfigure}
\begin{subfigure}{0.49\textwidth}
\hspace{2cm}
	\begin{tikzpicture}
	\node at (0,0) {\includegraphics[width=3cm]{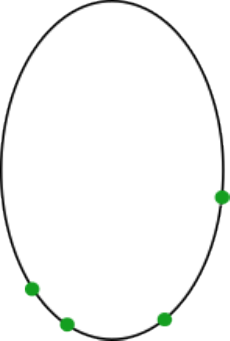}};
	\node at (1.5,2) {$\im \psi$};
	\node at (1.8,-.5) {$a_4$};
	\node at (1,-2.2) {$a_3$};
	\node at (-.6,-2.4) {$a_2$};
	\node at (-1.5,-1.6) {$a_1$};
	\end{tikzpicture}
\caption{Veronese factorization structure}
\label{3b}
\end{subfigure}
\caption{}
\end{figure}
In fact, every 2-dimensional polytope with 4 edges is compatible with the Veronese factorization structure.
Indeed, proceeding as in \Cref{segre polyotpes}, we view the polytope as a section of a cone $\sigma$, which, through the extremal rays of $\sigma^\vee$, gives four lines in 3-dimensional vector space, and hence four points in $\mathbb{P}^2$. As any five points in general position determine the rational normal curve of degree 2, there is a 1-parametric family of such curves (conics) fitting these four points, and hence for any member of this family, the cone $\sigma$ has its projectivised normals on factorization curves. Therefore, the polytope is compatible with a Veronese factorization structure.
\end{example}

A strategy for constructing compatible cones and polytopes is to start with finitely many points on factorization curves and de-projectivise them, which determines a cone $\sigma^\vee$ whose dual $\sigma$ is a compatible cone by construction, provided it is full-dimensional and pointed. To clarify how to do this rigorously we continue with the following general observations, which, once restricted to our setting, provide a construction of compatible cones and polytopes. \\

Note that a cone $\sigma\subset\mathfrak{h}^*$ with $n$ facets determines $n$ points in $\mathbb{P}(\mathfrak{h})$ by projectivising extremal rays of $\sigma^\vee$. However, not every choice of an affine chart realises points in $\mathbb{P}(\mathfrak{h})$ as generators of extremal rays of a cone. In general, we have the following.\par 
A finite collection of points $p_1,\ldots,p_n\in\mathbb{P}(\mathfrak{h})$ belongs to the domain of the affine chart given by $\epsilon\in \mathfrak{h}^*$ if and only if $\epsilon$ does not belong into the proper and closed set $\cup_{j=1}^n (p_j)^0$, where $(p_j)^0\subset \mathfrak{h}^*$ is the annihilator of $p_j\subset \mathfrak{h}$. The set $\cup_{j=1}^n (p_j)^0$ is a hyperplane arrangement in $\mathfrak{h}^*$ splitting it into a union of full-dimensional convex polyhedral cones which, in general, do not have $n$ facets. Observe that for $\sigma$ such a cone, bounded by $(p_{i_1})^0,\ldots,(p_{i_r})^0$, $r\leq n$, all lines $p_1,\ldots,p_n$ contain rays $p_1^+,\ldots,p_n^+$ which belong to $\sigma^\vee$, since all functionals $\alpha\in\sigma$ evaluate non-negatively on them, but the only extremal rays of $\sigma^\vee$ are $p_{i_1}^+,\ldots,p_{i_r}^+$.
By construction, the projectivised normals of $\sigma$ are $p_{i_1},\ldots,p_{i_r} \in \mathbb{P}(\mathfrak{h})$.

\begin{corollary}\label[corollary]{cor_examples}
Let $p_1,\ldots,p_n\in\mathbb{P}(\mathfrak{h})$, $n \geq \dim\mathfrak{h}$. If there exist $\epsilon\in \mathfrak{h}^*$ such that the image of $p_1,\ldots,p_n$ in the affine chart $\epsilon$ generate extremal rays of a full-dimensional cone $\sigma^\vee$, then its dual cone $\sigma$ is a full-dimensional and pointed cone with $n$ facets. The interior $\text{Int}(\sigma^\vee)$ parametrises affine hyperplane sections intersecting $\sigma$ in a convex polytope with $n$ facets.
\end{corollary}

As we are also interested in compatible polytopes which are simple polytopes, we recall relevant definitions and reflect them into associated cones. An $m$-dimensional polytope is \textit{simple} if exactly $m$ facets are incident with each of its vertices, and \textit{simplicial} if its every facet is a simplex.
Observe that an $m$-dimensional simple polytope arises as a section of a full-dimensional and pointed cone $\sigma$ whose extremal rays are intersections of exactly $m$ facets.
Dually, each facet of $\sigma^\vee$ contains exactly $m$ extremal rays, and these are linearly independent.
Equivalently, every compact slice of $\sigma^\vee$ has simplices as faces, thus $\sigma^\vee$ is a cone over a simplicial polytope.
Thus, to determine if a cone is a cone over a simplicial polytope means to know how many edges lie on facets.

\begin{thm}\label{Ver cone is simplicial}
A cone compatible with the Veronese factorization structure is a cone over a simplicial polytope.
\end{thm}
\begin{proof}
Say that the $n$ extremal rays of our cone lie on 1-dimensional spaces
\begin{align}\label{Ver is simplicial}
\psi(t_i),\hspace{.2cm} i=1,\ldots,n,
\end{align}
where $\psi=\psi_1=\cdots=\psi_m$ denotes the factorization curve, and  $n\geq m+1$ since the cone is full-dimensional (see \Cref{fc rnc}). We need to show that a facet-supporting hyperplane, which is generally defined by $m$ extremal rays and which in our case lie on 1-dimensional spaces \eqref{Ver is simplicial}, does not contain any other extremal rays. This is clearly true since any hyperplane intersects the degree $m$ curve $\psi$ in at most $m$ points. Thus, it is a cone over a simplicial polytope.
\end{proof}

In the rest of this subsection we describe hyperplanes and higher codimension spaces where facets and faces of a compatible cone and its dual lie. They have a particularly nice form characterised as $\varphi^t$-images of intersections of spaces $\Sigma_{j,\ell}$, see \Cref{faces} below.\par
We start with finding the hyperplane in $\mathfrak{h}^*$ corresponding to (a projectivised normal) $\psi_j(\ell) \in \mathbb{P}(\mathfrak{h})$. Note that since (see \Cref{ell_k^0 fixed})
\begin{align}\label{fs prop example}
\varphi\circ\psi_j(\ell)\subset \varphi(\mathfrak{h}) \cap \Sigma_{j,\ell}^0,
\end{align}
we have
\begin{align}
0=
\langle \Sigma_{j,\ell}, \varphi\circ\psi_j(\ell) \rangle =
\langle \varphi^t \Sigma_{j,\ell}, \psi_j(\ell) \rangle,
\end{align}
and if $v\in\mathfrak{h}$ annihilates $\varphi^t \Sigma_{j,\ell}\subset\mathfrak{h}^*$, then
\begin{align}
0=
\langle \varphi^t \Sigma_{j,\ell}, v \rangle =
\langle \Sigma_{j,\ell}, \varphi v \rangle,
\end{align}
i.e., $\varphi(v)\in\varphi(\mathfrak{h})\cap\Sigma_{j,\ell}^0$. Therefore, $\varphi^t \Sigma_{j,\ell}\subset\mathfrak{h}^*$ is a hyperplane with the annihilator $\psi_j(\ell) \subset \mathfrak{h}$ if and only if $\dim \left( \varphi(\mathfrak{h})\cap\Sigma_{j,\ell}^0 \right) = 1$. In other words, $\varphi^t\Sigma_{j,\ell}$ is a hyperplane if and only if $\psi_j(\ell)$ does not lie on any other curve. In general, we have

\begin{thm}\label{faces}
For $r\in\{1,\ldots,m\}$ and pairwise distinct $i_1,\ldots,i_r\in\{1,\ldots,m\}$, the space
\begin{align}\label{codim r space}
\varphi^t \left(  \Sigma_{i_1,\ell_1} \cap \cdots \cap \Sigma_{i_r,\ell_r} \right)
\end{align}
is of codimension $r$ for generic choices of $\ell_j \in \mathbb{P}(V_{i_j})$, $j=1,\ldots,r$. Furthermore, for $\ell_j$, $j=1,\ldots,r$, such that \eqref{codim r space} is of codimension $r$, if hyperplanes $\varphi^t ( \Sigma_{i_1,\ell_1} ), \cdots, \varphi^t ( \Sigma_{i_r,\ell_r} )$ are independent, then
\begin{align}\label{face equation}
\varphi^t \left(  \Sigma_{i_1,\ell_1} \cap \cdots \cap \Sigma_{i_r,\ell_r} \right) =
\varphi^t ( \Sigma_{i_1,\ell_1} ) \cap \cdots \cap \varphi^t ( \Sigma_{i_r,\ell_r} ).
\end{align}
\end{thm}
In particular, as \Cref{fc rnc} shows, such hyperplanes are always independent in case of Veronese factorization structure and generically independent for product Segre-Veronese factorization structure.
\begin{proof}
Set-theoretically, we always have that the space \eqref{codim r space} lies in the intersection of the hyperplanes. The independence of hyperplanes implies that they intersect in a codimension $r$ space. Thus, showing that \eqref{codim r space} is of codimension $r$ proves \eqref{face equation}. To this end, we prove that its annihilator,
\begin{align}\label{larger}
\varphi
\left(
\left(
\varphi^t \left(  \Sigma_{i_1,\ell_1} \cap \cdots \cap \Sigma_{i_r,\ell_r} \right)
\right)^0
\right)
=
\varphi(\mathfrak{h}) \cap
\left( \Sigma_{i_1,\ell_1}^0 + \cdots + \Sigma_{i_r,\ell_r}^0 \right),
\end{align}
is $r$-dimensional.\par
By combining
\begin{align}\nonumber
\dim
\left(
\varphi(\mathfrak{h}) +
\Sigma_{i_1,\ell_1}^0 + \cdots + \Sigma_{i_r,\ell_r}^0
\right)
=&
\dim
\left(
\varphi(\mathfrak{h}) + \Sigma_{i_2,\ell_2}^0 + \cdots + \Sigma_{i_r,\ell_r}^0
\right)
+
\dim
\left( \Sigma_{i_1,\ell_1}^0 \right)\\
&-
\dim
\left(
\left( \varphi(\mathfrak{h}) + \Sigma_{i_2,\ell_2}^0 + \cdots + \Sigma_{i_r,\ell_r}^0 \right) \cap \Sigma_{i_1,\ell_1}^0
\right),
\end{align}
and (rank-nullity theorem)
\begin{align}\nonumber
\dim
\left(
\rho_{i_1,\ell_{i_1}}
\left(
\varphi(\mathfrak{h}) + \Sigma_{i_2,\ell_2}^0 + \cdots + \Sigma_{i_r,\ell_r}^0
\right)
\right)
=&
\dim
\left(
\varphi(\mathfrak{h}) + \Sigma_{i_2,\ell_2}^0 + \cdots + \Sigma_{i_r,\ell_r}^0
\right)\\
&-
\dim
\left(
\left( \varphi(\mathfrak{h}) + \Sigma_{i_2,\ell_2}^0 + \cdots + \Sigma_{i_r,\ell_r}^0 \right) \cap \Sigma_{i_1,\ell_1}^0
\right),
\end{align}
we arrive at
\begin{align}
\dim
\left(
\varphi(\mathfrak{h}) +
\Sigma_{i_1,\ell_1}^0 + \cdots + \Sigma_{i_r,\ell_r}^0
\right)
=&
\dim \left( \Sigma_{i_1,\ell_1}^0 \right)+
\dim
\left(
\rho_{i_1,\ell_{i_1}}
\left(
\varphi(\mathfrak{h}) + \Sigma_{i_2,\ell_2}^0 + \cdots + \Sigma_{i_r,\ell_r}^0
\right)
\right),
\end{align}
where $\rho_{i_1,\ell_{i_1}}$ represents the contraction $\rho_{i_1,v}$ for some/any $v \in \ell_{i_1}$.
Similar abbreviations are used in the following.
Repeating the above $(r-2)$-times yields
\begin{align}\nonumber
\dim
\left(
\varphi(\mathfrak{h}) +
\Sigma_{i_1,\ell_1}^0 + \cdots + \Sigma_{i_r,\ell_r}^0
\right)
=&
\dim \left( \Sigma_{i_1,\ell_1}^0 \right)
+
\sum_{j=1}^{r-2}
\dim
\left(
\rho_{i_j,\ell_j} \circ \cdots \circ \rho_{i_1,\ell_1} \Sigma_{i_{j+1},\ell_{j+1}}^0
\right)\\\label{id:final}
&+
\dim
\left(
\rho_{i_{r-1},\ell_{r-1}} \circ \cdots \circ \rho_{i_1,\ell_1}
\left(
\varphi(\mathfrak{h}) + \Sigma_{r,\ell_r}^0
\right)
\right),
\end{align}
which is valid for $\varphi(\mathfrak{h})=0$ too. Finally, inserting \eqref{id:final} with $\varphi(\mathfrak{h})=0$ and \eqref{id:final} itself into
\begin{align}\nonumber
\dim
\left(
\varphi(\mathfrak{h}) +
\Sigma_{i_1,\ell_1}^0 + \cdots + \Sigma_{i_r,\ell_r}^0
\right)
=&
\dim \left( \varphi(\mathfrak{h}) \right)
+
\dim
\left(
\Sigma_{i_1,\ell_1}^0 + \cdots + \Sigma_{i_r,\ell_r}^0
\right)\\\label{id:default}
&-
\dim
\left(
\varphi(\mathfrak{h}) \cap
\left( \Sigma_{i_1,\ell_1}^0 + \cdots + \Sigma_{i_r,\ell_r}^0 \right)
\right).
\end{align}
provides the final formula
\begin{gather}
\dim
\left(
\rho_{i_{r-1},\ell_{r-1}} \circ \cdots \circ \rho_{i_1,\ell_1}
\left(
\varphi(\mathfrak{h})
\right)
\right) \label{first line}\\
-
\dim
\left(
\rho_{i_{r-1},\ell_{r-1}} \circ \cdots \circ \rho_{i_1,\ell_1}
\left(
\varphi(\mathfrak{h})
\right) 
\cap
\rho_{i_{r-1},\ell_{r-1}} \circ \cdots \circ \rho_{i_1,\ell_1}
\left(
\Sigma_{r,\ell_r}^0
\right) 
\right)
=\label{second line} \\
\dim \left( \varphi(\mathfrak{h}) \right)
-
\dim
\left(
\varphi(\mathfrak{h}) \cap
\left( \Sigma_{i_1,\ell_1}^0 + \cdots + \Sigma_{i_r,\ell_r}^0 \right)
\right).
\end{gather}
Now, we use \Cref{fs quotient thm} to find an open non-empty $A_{i_1} \subset \mathbb{P}(V_{i_1})$ where the quotient $\varphi_{i_1,v}$ of $\varphi$ is a factorization structure for any non-zero $v \in \lambda$, $\lambda \in A_{i_1}$, then to find an open non-empty $A_{i_2} \subset \mathbb{P}(V_{i_2})$ where the quotient $(\varphi_{i_1,v})_{i_2,w}$ of $\varphi_{i_1,v}$ is a factorization structure for any $w \in \mu$, $\mu \in A_{i_2}$, etc.
Thus, for $(\ell_1,\ldots,l_r) \in A_{i_1} \times \cdots \times A_{i_r}$, i.e., generic $\ell_j \in \mathbb{P}(V_{i_j})$, $j=1,\ldots,r$, we find that \eqref{first line} is $m+1-(r-1)$, and \eqref{second line} is 1. Thus,
\begin{align}
\dim
\left(
\varphi(\mathfrak{h}) \cap
\left( \Sigma_{i_1,\ell_1}^0 + \cdots + \Sigma_{i_r,\ell_r}^0 \right)
\right)
=
r
\end{align}
as claimed.
\end{proof}

The above suggests 

\begin{corollary}\label[corollary]{face structure}
Generically, $\varphi^t \ell_1\otimes\cdots\otimes\ell_m \subset \mathfrak{h}^*$ determines a hyperplane through $\psi_j(\ell_j)$, $j=1,\ldots,m$.
\end{corollary}

\begin{proof}
\Cref{faces} shows that $\varphi^t \ell_1\otimes\cdots\otimes\ell_m$ is generically 1-dimensional. Since $\varphi\circ\psi_j(\ell_j)$ has $\ell_j^0$ at the $j$-th slot, the computation
\begin{equation}
\left\langle  \varphi^t \ell_1\otimes\cdots\otimes\ell_m, \psi_j(\ell_j) \right\rangle =
\left\langle  \ell_1\otimes\cdots\otimes\ell_m, \varphi \circ \psi_j(\ell_j) \right\rangle =
0
\end{equation}
gives the claim.
\end{proof}

Note that if a cone $\sigma^\vee$ has extremal rays lying on $\psi_j(\tau_{ji})$, then extremal rays of $\sigma$ lie on $\varphi^t \ell_1\otimes\cdots\otimes\ell_m$ for some $\ell_r\in\{\tau_{ji}\}_{j,i}$.

\subsection{Cones compatible with the product Segre-Veronese factorization structure}\label{s31}

This section demonstrates the construction of compatible cones through an example of a cone compatible with the product Segre-Veronese factorization structure. Furthermore, a condition for determining its facets is given which generalises the Gale evenness condition from cyclic polytopes to this broader setting.
More details can be found in \cite{brandenburg2024} where an extensive theory for the case of the Veronese factorization structure was already developed. \\

Recall the $m$-dimensional product Segre-Veronese factorization structure 
\begin{align}\label{1}
\mathfrak{h}:=
\sum_{j=1}^{k}
ins_j
\left(
S^{d_j}W_j^*\otimes\Gamma_j
\right)
\xhookrightarrow{\hspace{.5cm} \varphi \hspace{.5cm}}
\bigotimes_{j=1}^k (W_j^*)^{\otimes d_j}
=: V^*,
\end{align}
where
\begin{align}\label{2}
\Gamma_j=
\bigotimes_{\substack{r=1\\r\neq j}}^k (a^r)^{\otimes d_r}
\end{align}
for some 1-dimensional subspaces $a^r\subset W_r^*$, $r=1,\ldots,k$, $d_1+\cdots+d_k = m$, and $\otimes_{r=1}^k (a^r)^{\otimes d_r}$ is called \textit{the intersection point}.
Its distinct factorization curves are given by
\begin{align}\label{fc_examples}
\varphi\circ\psi_j(\ell)=
ins_j
\left(
(\ell^0)^{\otimes d_j} \otimes \Gamma_j
\right),
\end{align}
$j=1,\ldots,k$ (see \Cref{segre lines example} - \Cref{SVfs curves} for more details).

Consider points on factorization curves of the product Segre-Veronese factorization structure,
\begin{align}\label{edges}
\psi_j(\tau_{ji}) \in \mathbb{P}(\mathfrak{h}),
\hspace{.2cm}
i=1,\ldots,c_j,
\hspace{.2cm}
j=1,\ldots,k,
\end{align}
for some $c_j$'s, where $\tau_{ji}$'s are pair-wise distinct, $\tau_{ji} \in \mathbb{P}(W_j)$. We fix a chart, and declare \eqref{edges} in this chart as generators of the cone, which is therefore pointed. Assuming that these generators generate extremal rays of the cone, a necessary and sufficient condition for its full-dimensionality follows from \Cref{fc rnc}. The cone is full-dimensional if and only if there exists $j_0\in\{1,\ldots,k\}$ such that $c_{j_0}>d_{j_0}$ and $c_j\geq d_j$ for $j\neq j_0$. However, as noted above, not every affine chart realises \eqref{edges} as generators of extremal rays. Regardless, the condition for determining facets is applicable as shown below. \\

We express images of \eqref{edges} in an affine chart chosen below. To do so we fix dual bases of $W_r$ and $W^*_r$ such that the first basis vector $\underline{a}^r$ of $W_r^*$ lies on $a^r$ (see \eqref{1} and \eqref{2} for notation). These provide coordinates $t_{ji}$ for $\tau_{ji}$, $t_{ji}=\tau_{ji}/\langle \tau_{ji}, \underline{a}^j \rangle$, and a basis $\epsilon_0,\epsilon_{ji}$, $i=1,\ldots,d_j$, $j=1,\ldots,k$, of $\mathfrak{h}$, uniquely characterised by 
\begin{align}\label{basis}
\varphi\epsilon_0 &= \otimes_{r=1}^k(\underline{a}^r)^{\otimes d_r},\nonumber\\
\varphi\epsilon_{ji} &= \text{ins}_j \left( \varepsilon_{ji} \otimes \bigotimes_{\substack{r=1\\r\neq j}}^k (\underline{a}^r)^{\otimes d_r} \right),
\end{align}
where $\varepsilon_{ji}$ together with $(\underline{a}^j)^{\otimes d_j}$ denote the standard basis for symmetric tensors $S^{d_j}W_j^*$.
Indeed, $\epsilon_0$ together with $\epsilon_{ji}$, $i=1,\ldots,d_j$, form a basis of $ins_j \left( S^{d_j} W_j^* \otimes \bigotimes_{\substack{r=1 \\ r \neq j}}^k (a^r)^{\otimes d_r} \right)$, (see also \Cref{SVfs curves} and \Cref{one intersection example}). \par
We define the affine chart $\epsilon\in\mathfrak{h}^*$ by $\epsilon = \varphi^t \varepsilon \in \mathfrak{h}^*$, where
\begin{align}\label{epsilon}
\varepsilon:=
\sum_{j=1}^k
\text{ins}_j \left( (0,-1)^{\otimes d_j} \otimes (1,0)^{\otimes (m-d_j)} \right) \in V.
\end{align}
The only point of any $\psi_j$, $j=1,\ldots,k$, which is not in this chart is the intersection point. Note that this remains true if for any $j\in\{1,\ldots,k\}$, the $(1,0)^{\otimes (m-d_j)}$-part of \eqref{epsilon} is replaced by a tensor from
$\otimes_{\substack{i=1\\ i\neq j}}^k \left( W_i \right)^{\otimes d_i}$
which does not belong to the annihilator of
$(1,0)^{\otimes (m-d_j)} \in \otimes_{\substack{i=1\\ i\neq j}}^k \left( W_i^* \right)^{\otimes d_i}$.

In this chart and coordinates, $\psi_j([1:x])$ can be found explicitly,
\begin{align}
\frac{\varphi \circ \psi_j([1:x])}{\langle \varphi \circ \psi_j([1:x]), \varepsilon \rangle} =
\text{ins}_j
\left( (x,-1)^{\otimes d_j} \otimes (1,0)^{\otimes (m-d_j)} \right)
\end{align}
and thus
\begin{align}\label{psi in a chart}
\frac{\psi_j([1:x])}{\langle \psi_j([1:x]), \epsilon \rangle} =
x^{d_j}\epsilon_0 +
\sum_{i=1}^{d_j} (-1)^i x^{d_j-i}\epsilon_{ji}.
\end{align}
Evaluating \eqref{psi in a chart} at $x=t_{ji}$, we obtain images of \eqref{edges} in the chart $\epsilon$, i.e., the vectors generating the cone $\sigma^\vee$,
\begin{gather}\label{cone}
\sigma^\vee =
\text{cone} \left( \left(t_{ji}^{d_j}, -t_{ji}^{d_j-1},\ldots,(-1)^{d_j-1} t_{ji}, (-1)^{d_j} \right) \bigg| \hspace{.1cm} i=1,\ldots,c_i, \hspace{.2cm} j=1,\ldots,k \right).
\end{gather}

To describe facet-supporting hyperplanes of $\sigma^\vee$, we note that each such is in particular a hyperplane through $m$ linearly independent extremal rays. First, we classify hyperplanes through $m$ linearly independent 1-dimensional spaces lying on factorization curves, which allows us to see which collections of points from \eqref{edges} give rise to a hyperplane, and then we derive a condition for deciding which of these are facet-supporting hyperplanes.

\begin{proposition}\label[proposition]{hyperplanes}
Let $S$ be a set of $m$ linearly independent points lying on factorization curves of the product Segre-Veronese factorization structure.
Then, one of the following two is satisfied.
\begin{enumerate}
\item For each $j=1,\ldots,k$, the cardinality $|S \cap \im \psi_j|$ is exactly $d_j$. Then, parametrising the points as $\psi_j([1:x_{ji}])$, $i=1,\ldots,d_j$, $j=1,\ldots,k$, we obtain normal vectors of the associated hyperplane,
\begin{align}\label{normals}
c \cdot \varphi^t \left( \otimes_{j=1}^k\otimes_{i=1}^{d_j} (1,x_{ji}) \right),
\end{align}
$c \in \mathbb{R}\backslash\{0\}$. Additionally, the hyperplane does not contain the intersection point.
\item There exists $i \in \{1,\ldots,k\}$ such that $|S \cap \im \psi_i| = d_i+1$.
Then, there is $r \in \{1,\ldots,k\} \backslash \{i\}$ such that, when the intersection point is excluded, $\im \psi_r$ contains exactly $d_r-1$ of the points, say labelled as $\psi_r([1:x_{rq}])$, $q=1,\ldots,d_r-1$.
The associated hyperplane contains curves $\psi_j$, $j\neq r$, and hence their span \eqref{chunk}. Its normal vectors can be written as
\begin{align}\label{normal big}
c \cdot \varphi^t
\left( \bigotimes_{j=1}^{r-1} \bigotimes_{i=1}^{d_j} (1,x_{ji}) \otimes (0,1) \otimes \bigotimes_{q=1}^{d_r-1} (1,x_{rq}) \otimes \bigotimes_{j=r+1}^k \bigotimes_{i=1}^{d_j} (1,x_{ji}) \right)
\end{align}
$c \in \mathbb{R}\backslash\{0\}$, where $x_{ji}$, $i=1,\ldots,d_j$, $j\neq r$, are any such that $\psi_j([1:x_{ji}])$ are distinct.
In particular, any mutually distinct choices of these $\psi_j([1:x_{ji}])$ give the same hyperplane.
\end{enumerate}
\end{proposition}
\begin{proof}
\Cref{fc rnc} shows that there are at most $d_j+1$ linearly independent directions on $\psi_j$, $j=1,\ldots,k$, and note that the shape of the product Segre-Veronese factorization structure \eqref{1} associated to the partition $m=d_1+\cdots+d_k$ is such that its summands mutually intersect at a unique single direction (see also \Cref{SVfs curves}).
Therefore, for a fixed distribution of $m$ independent directions on factorization curves, every $\psi_j$ must contain at least $d_j-1$ of these directions.
Indeed, since factorization curves $\psi_j$, $j\neq i$, span together $(m+1-d_i)$-dimensional space, the curve $\psi_i$ cannot contain strictly less than $d_i-1$ points.
Now we cover all possible cases: there exists a curve $\psi_i$ carrying $d_i+1$ independent directions, each curve $\psi_i$ carries exactly $d_i$ independent directions, and there exists a curve $\psi_i$ carrying $d_i-1$ independent directions. \par
First, we start with $d_i+1$ points on $\psi_i$ for some $i\in\{1,\ldots,k\}$, which leaves us with choosing $m-d_i-1$ independent directions on factorization curves $\psi_j$, $j\neq i$, each now retaining exactly $d_j$ dimensions.
Then, the only way how to ensure $m$ independent directions is to fix an index $r \neq i$ and $d_r-1$ independent directions on $\psi_r$, and $d_j$ independent directions on $\psi_j$ otherwise.
Note that the latter distribution of points defines the hyperplane containing
\begin{align}\label{chunk}
\sum_{\substack{j=1 \\ j\neq r}}^{k}
ins_j
\left(
S^{d_j}W_j^*\otimes
\bigotimes_{\substack{q=1\\r\neq j}}^k (1,0)^{\otimes d_q}
\right)
\end{align}
and $d_r-1$ points on $\psi_r$ as above.
Observe that it is the same hyperplane as the one given in \eqref{normal big}. \par
Secondly, we consider $d_j$ points on the curve $\psi_j$ for each $j=1,\ldots,k$, say $\psi_j([1:x_{ji}])$, $i=1,\ldots,d_j$, $j=1,\ldots,k$, which provides $m$ independent directions, and \Cref{faces} and \Cref{face structure} show that its normal is of the form \eqref{normals}. Note that this case excludes the intersection point as one $m$ independent directions. \par
Finally, it is easy to observe that having a curve $\psi_j$, with $d_j-1$ independent directions forces existence of a curve $\psi_i$, $i\neq j$ with $d_i+1$ independent directions, which was solved in the first case.
\end{proof}

To proceed further, note that since a cone consists of positive combinations of its generators, a hyperplane is a (facet-)supporting hyperplane if and only if it has all the generators on its positive side.
Using this we find facet-supporting hyperplanes of the cone $\sigma^\vee$ from \eqref{cone} by separately considering the two types of hyperplanes from \Cref{hyperplanes}.
Let $\{x_{ji}\}_{i=1}^{d_j} \subset \{t_{jr}\}_{r=1}^{c_j}$, $j=1,\ldots,k$, be such that the corresponding hyperplane through $\psi_j([1:x_{ji}])$ has the normal \eqref{normals}.
We compute its contraction with a general point on $\psi_j/\langle \psi_j, \epsilon \rangle$,
\begin{gather}
\nonumber
c\left \langle \varphi^t \otimes_{j=1}^k\otimes_{i=1}^{d_j} (1,x_{ji}), \hspace{.1cm}
\frac{\psi_j([1:x])}{\langle \psi_j([1:x]), \epsilon \rangle} \right \rangle =\\\nonumber
c\left\langle \otimes_{j=1}^k\otimes_{i=1}^{d_j} (1,x_{ji}), \hspace{.1cm}
\text{ins}_j \left( (x,-1)^{\otimes d_j} \otimes (1,0)^{\otimes (m-d_j)} \right) \right\rangle =\\
c\prod_{i=1}^{d_j} \left\langle (1,x_{ji}), (x,-1) \right \rangle = 
c\prod_{i=1}^{d_j} (x-x_{ji}) = \nonumber\\
c\sum_{i=0}^{d_j}(-1)^{i} x^{d_j-i} \sigma_{i}(x_{j1},\ldots,x_{jd_j}).\label{the polynomial}
\end{gather}
The expression \eqref{the polynomial} is a polynomial $p_j$ in $x$ which vanishes at $\{x_{ji}\}_{i=1}^{d_j} \subset \{t_{jr}\}_{r=1}^{c_j}$. We conclude
\begin{proposition}\label[proposition]{general Gale}
The value of the polynomial \eqref{the polynomial} on points $t_{jr}$, $r=1, \ldots, c_j$ is zero or has the same sign if and only if any two elements of the set $\{t_{jr}\}_{r=1}^{c_j} \backslash \{x_{ji}\}_{i=1}^{d_j}$ are separated by an even number of elements from $\{x_{ji}\}_{i=1}^{d_j}$ in the sequence $t_{jr}$, $r=1,\ldots,c_i$.
\end{proposition}

In the Veronese factorization structure case, hyperplanes \eqref{normals} are the only class of hyperplanes through $m$ independent points \eqref{edges}, and \Cref{general Gale} recovers the Gale's evenness condition. 
However, to understand if the associated compatible polytopes are cyclic requires further analysis (for details see \cite{brandenburg2024}). \par
Collecting our previous results together yields a condition for facet-supporting hyperplanes of $\sigma^\vee$.

\begin{thm}\label{Gale}
Let $\sigma^\vee$ be the cone compatible with the product Segre-Veronese factorization structure $\varphi$ generated by images of \eqref{edges} in the affine chart $\varepsilon = \varphi^t \epsilon$, $\epsilon$ as in \eqref{epsilon}, i.e., $\sigma^\vee$ is \eqref{cone}. Its $m$ linearly independent generators determine a facet-supporting hyperplane if and only if one of the following holds:
\begin{enumerate}
\item The hyperplane does not contain the intersection point, and, when the $m$ independent generators are labelled as in \Cref{hyperplanes} (1), for each $j=1,\ldots,k$ the value of polynomials $p_j$ from \eqref{the polynomial} is zero or has the constant sign (constant also with respect to $j$) on $t_{ji}$, $i=1,\ldots,d_j$.
\item The hyperplane contains the intersection point, and, in the notation of \Cref{hyperplanes} (2), the value of the polynomial
\begin{align}\label{rest of general Gale}
-\delta_j^r
c
\prod_{q=1}^{d_r-1} \left\langle (1,x_{rq}), (x,-1) \right \rangle =
-\delta_j^r
c
\sum_{i=0}^{d_r-1}(-1)^{i} x^{d_r-1-i} \sigma_{i}(x_{r1},\ldots,x_{rd_{r-1}}),
\end{align}
is zero or has a constant sign on $t_{rq}$, $q=1,\ldots,d_r-1$, where $\delta_j^r$ is the Kronecker symbol.
\end{enumerate}
\end{thm}
\begin{proof}
\Cref{hyperplanes} establishes that there are only two types of hyperplanes, those containing the intersection point and those which do not contain it.
The contraction for hyperplanes of type (1) is computed in \eqref{the polynomial}.
Since a hyperplane is facet-supporting if and only if the entire cone lies on its one side, the contraction must have constant sign, or be zero on points inside the hyperplane.
We are left to compute contractions with hyperplanes of type (2).
To this end, simply observe that because any hyperplane from \Cref{hyperplanes} (2) annihilates curves $\psi_j$, $j\neq r$, the only non-trivial contraction is against points on $\psi_r$ which reads \eqref{rest of general Gale}.
\end{proof}

The above theorem can be formulated in the spirit of \Cref{general Gale}, and further worded in combinatorics, but this is beyond the scope of this article. For more details in the case of the Veronese factorization structure, see \cite{brandenburg2024}.

\begin{proposition}
In the case of the Veronese factorization structure, rays generated by images of 1-dimensional spaces \eqref{edges} in the chart $\epsilon$ \eqref{epsilon} are extremal rays of the corresponding cone.
\end{proposition}
\begin{proof}
Fix such a ray $\rho$.
\Cref{general Gale} implies the existence of a facet-supporting hyperplane given by $m$ independent direction, with $\rho$ lying on one of them.
Recall from \Cref{Ver cone is simplicial} that the cone is a cone over a simplicial polytope, and thus no more than $m$ directions can lie on the hyperplane.
Therefore, $\rho$ cannot be written as a non-negative combination of other rays, and hence is extremal.
\end{proof}

\begin{rem}
We wish to remark that computations \eqref{the polynomial} and \eqref{rest of general Gale} apply in finding facet-supporting hyperplanes for cones/polytopes compatible with a general factorization structure as well. For a general affine chart, a similar computation works, however, the contraction \eqref{the polynomial} is a genuine rational function in this case.
\end{rem}

Finally, as the discussion below \Cref{cor_examples} and the corollary itself explain, if we find an affine chart in which \eqref{edges} generate a full-dimensional cone (over a simplicial polytope), then compatible (simple) polytopes are parametrised by the interior of this cone and realised as sections of its dual. Additionally, if images of \eqref{edges} in the affine chart generate extremal rays of this cone, then compatible polytopes have $c_1+\cdots+c_k$ facets.

\begin{rem}
Observe that choosing $\epsilon = (0,1)^{\otimes m}$ in case of the Veronese factorization structure, whose factorization curve is denoted here by $\psi$, results in vectors
\begin{gather}
\frac{\psi([1:x])}{\langle \psi([1:x]), \epsilon \rangle}
\end{gather}
with the coordinate expression $((-x)^m, (-x)^{m-1}, \ldots, -x, 1)$. Thus, cone generators lie on \textit{the momentum curve}, and their convex hull is by definition a \textit{cyclic polytope} \cite{gale1963neighborly}.
\end{rem}

\subsection{Compatible rational Delzant polytopes}\label{s32}

In this subsection, we find compatible polytopes which are rational Delzant.

\begin{defn}\label[defn]{Delzant}
Let $\mathfrak{t}$ be an $m$-dimensional real vector space. A \textit{rational Delzant polytope} in $\mathfrak{t}^*$ is a simple compact convex polytope
\begin{align}
\Delta=
\left\{
x\in\mathfrak{t}^*
\hspace{.2cm}|\hspace{.2cm}
L_j(x) \geq 0, j=1,\ldots,n
\right\}
\end{align}
with
\begin{align}
L_j(x) = \langle u_j, x \rangle + \lambda_j
\end{align}
for some $\lambda_1,\ldots,\lambda_n\in\mathbb{R}$ such that $u_1,\ldots,u_n$ belong to a lattice $\Lambda\subset\mathfrak{t}$. It is called \textit{integral} or simply \textit{Delzant polytope} if for each vertex $v$ of $\Delta$, the set
$\{u_j \hspace{.1cm}|\hspace{.1cm} L_j(v) = 0\}$
is a basis of $\Lambda$. The set $\{L_j\}_{j=1}^n$ is understood as the minimal set of affine functionals defining $\Delta$.
\end{defn}
\begin{rem}
(Rational) Delzant polytopes occur in the context of toric geometry (see \cite{guillemin1982convexity, atiyah1982convexity, delzant1988hamiltoniens, audin2012topology,lerman1997hamiltonian}). On one hand, the image of the momentum map of a toric symplectic (orbifold) manifold is a (rational) Delzant polytope. On the other hand, (generalised) Delzant construction produces such a toric geometry out of any (rational) Delzant polytope. The condition on normals to form a lattice-basis ensures smoothness of the resulting toric space.
\end{rem}

The (inward-pointing) normals of a compatible polytope given as a section of a cone $\sigma$ by $\beta\in\sigma^\vee$, $\sigma^\vee$ having extremal rays generated by \eqref{edges} in the chart $\epsilon$ (see \eqref{epsilon}), are
\begin{align}
C_{ji}
\frac{\psi_j([1:t_{ji}])}{\langle \psi_j([1:t_{ji}]), \epsilon \rangle}\text{ mod }\beta,
\end{align}
where $C_{ji}$ are any positive constants. We wish to find such scales $C_{ji}$ or a chart $\beta$ for which the polytope is rational Delzant. To do so, we use Vandermonde identities which follow from properties of factorization structures. In general we have

\begin{rem}\label[rem]{gen VI}
For a general factorization structure of dimension $m+1$ and pair-wise distinct $x_1,\ldots,x_{m+1}\in\mathbb{R}$ we denote
$x = \text{span}\left\{ \otimes_{r=1}^{m+1} (1,x_r) \right\} \in \mathbb{P}(V)$
and find
\begin{gather}
\partial_{x_i}
\frac{\varphi^t x}{\langle \varphi^tx, \beta \rangle}
\in
\beta^0 \subset \mathfrak{h}^*.
\end{gather}
Differentiating the identity
\begin{align}
\left \langle \frac{\varphi^t x}{\langle \varphi^tx, \beta \rangle}, \hspace{.1cm}
\frac{\psi_j([1:x_j])}{\langle \psi_j([1:x_j]), \epsilon \rangle} \right \rangle
=0
\end{align}
shows
\begin{gather}\label{V-id mod beta}
\left \langle \partial_{x_i} \frac{\varphi^t x}{\langle \varphi^tx, \beta \rangle}, \hspace{.1cm}
\frac{\psi_j([1:x_j])}{\langle \psi_j([1:x_j]), \epsilon \rangle} \text{ mod }\beta \right \rangle
=0
\end{gather}
for $i\neq j$.
\end{rem}

In particular, for Veronese factorization structure and $\beta=(1,0,\ldots,0)$ the expression \eqref{V-id mod beta} yields
\begin{gather}\label{pre-Vandermonde}
\begin{bmatrix}
\partial_{x_1} \sigma_1 & \cdots & \partial_{x_1} \sigma_{m+1}\\
\vdots & \ddots & \vdots \\
\partial_{x_{m+1}} \sigma_1 & \cdots & \partial_{x_{m+1}} \sigma_{m+1}
\end{bmatrix}
\begin{bmatrix}
- x_1^{m} & \cdots & - x_m^{m}\\
\vdots & \ddots & \vdots\\
(-1)^{m+1} & \cdots & (-1)^{m+1}
\end{bmatrix}
=
- \text{diag}\{ \Delta_1,\ldots, \Delta_{m+1} \},
\end{gather}
where $\sigma_j=\sigma_j(x_1,\ldots,x_{m+1})$ is the $j$th elementary symmetric polynomial, $\sigma_0=1$, and $\Delta_j = \prod_{\substack{r=1\\r\neq j}}^{m+1} \left( x_j - x_r \right)$. Since in the ring of $m\times m$ matrices, a left inverse is also a right inverse, \eqref{pre-Vandermonde} gives Vandermonde identities
\begin{gather}
\sum_{r=1}^{m+1}
\frac{(-1)^{j-1}(x_r)^{{m+1}-j} \partial_{x_r}\sigma_i}{\Delta_r}
=
\delta_{ij},
\hspace{1cm}
i,j=1,\ldots,{m+1},
\end{gather}
which for $i=1$ read
\begin{gather}\label{1 identity}
\sum_{r=1}^{m+1}
\frac{1}{\Delta_r}
\begin{bmatrix}
(x_r)^{m}\\
-(x_r)^{m-1}\\
\vdots\\
(-1)^{m}
\end{bmatrix}
=
\begin{bmatrix}
1\\
0\\
\vdots\\
0
\end{bmatrix}.
\end{gather}
Note that for a general $\beta$ and a general factorization structure we obtain generalised Vandermonde identities.

\begin{lemma}\label[lemma]{common lattice}
Let $v_1,\ldots,v_m$ be a basis of $\mathfrak{t}$ and $v_{m+1},\ldots,v_{m+\ell}\in\mathfrak{t}$. Then, each $v_j$, $j=m+1,\ldots,m+\ell$, can be expressed as a rational linear combination of the basis if and only if $\{v_1,\ldots,v_{m+\ell}\}$ belong to a common full-rank lattice.
\end{lemma}
\begin{proof}
By assumptions, for any $j=1,\ldots,\ell$ we have
$$v_{m+j} = \sum_{r=1}^m \frac{\alpha_j^r}{\beta_j^r} v_r,$$
where $\alpha_j^r,\beta_j^r\in\mathbb{Z}$ and $\beta_j^r \neq 0$. The lattice generated by
$$\frac{v_r}{\text{lcm}\{\beta_1^r,\ldots,\beta_\ell^r\}}, \hspace{.2cm} r=1,\ldots, m,$$
contains all $v_1,\ldots,v_{m+\ell}$ as claimed.\par
For the other part, we show that if $\beta$ is an element of the common full-rank lattice, e.g., $\beta=v_j$ for some $j\in\{1,\ldots,m+\ell\}$, then it is a rational linear combination of the basis $v_1,\ldots,v_m$. To this end, we choose a lattice basis $e_1,\ldots,e_m$, and observe that expansions
\begin{align}\
\beta &= \sum_{i=1}^m \beta ^i e_i, \hspace{.2cm} \beta^i \in \mathbb{Z}, \\
v_j &= \sum_{i=1}^m \kappa_j^i e_i, \hspace{.2cm}
\kappa_j^i\in\mathbb{Z}, \hspace{.2cm}
j=1,\ldots,m,
\end{align}
combine with
\begin{align}
\beta = \sum_{j=1}^m b^j v_j
\end{align}
into the system
\begin{align}
\beta^i = \sum_{j=1}^m \kappa_j^i b^j, \hspace{.2cm} i=1,\ldots,m,
\end{align}
which, by Cramer's rule, forces $b^1,\ldots, b^{m+1}\in\mathbb{Q}$, as claimed.
\end{proof}

\begin{example}\label[example]{simplex ex}
We examine simplex compatible with the Veronese factorization structure $S^mW^*$, whose factorization curve is denoted here by $\psi$. For real numbers $x_1<\cdots<x_{m+1}$, let
\begin{align}\label{simplicial}
\frac{\psi([1:x_r])}{\langle \psi([1:x_r]), \epsilon \rangle},
\hspace{.2cm}
r=1,\ldots,m+1,
\end{align}
generate extremal rays of the cone $\sigma^\vee$ over a simplicial polytope (see \Cref{Ver cone is simplicial}). Clearly, any $\beta\in\text{Int}(\sigma^\vee)$ yields a compatible polytope which is a simplex. \Cref{general Gale} shows that $m$ hyperplanes given by any $m$ vectors out of
\begin{align}\label{simplex normals}
C_r
\frac{\psi([1:x_r])}{\langle \psi([1:x_r]), \epsilon \rangle} \text{ mod } \beta,
\hspace{.2cm}
r=1,\ldots,m+1,
\end{align}
intersect in a vertex, and that all vertices arise this way, where $C_r$ are positive constants.\par
To find if this simplex is rational Delzant, or Delzant, means to determine if all normals belong to a common lattice, or if sets of normals corresponding to simplex's vertices span the same lattice, respectively. We start with two such sets of normals \eqref{simplex normals}, say indexed by $\{1,\ldots,m\}$ and $S:=\{1,\ldots,m+1\}\backslash\{m\}$, and find when they belong to a common lattice. \Cref{common lattice} shows that this happens if and only if there exist $\alpha^1,\ldots,\alpha^{m+1}\in\mathbb{Q}$ such that
\begin{align}
\sum_{r=1}^{m+1}
\alpha^r 
C_r
\frac{\psi([1:x_r])}{\langle \psi([1:x_r]), \epsilon \rangle} \text{ mod } \beta
=0,
\end{align}
which is equivalent with
\begin{align}
\sum_{r=1}^{m+1}
\alpha^r 
C_r
\frac{\psi([1:x_r])}{\langle \psi([1:x_r]), \epsilon \rangle}
\in
\langle \beta \rangle,
\end{align}
and thus $\langle \beta \rangle$ must have a rational point with respect ot the full-rank lattice $\Lambda$ generated by
\begin{align}\label{Lambda generators}
C_r
\frac{\psi([1:x_r])}{\langle \psi([1:x_r]), \epsilon \rangle},
\hspace{.2cm}
r=1,\ldots,m+1,
\end{align}
in $S^mW^*$. Additionally, these two sets of normals generate the same lattice in $S^mW^*/\langle \beta \rangle$ if and only if $\alpha^m, \alpha^{m+1} \in \{\pm1 \}$. Since $\beta$ belongs to the cone, it forces $\alpha^m = \alpha^{m+1} =1 $. Replacing the set $S$ with any other set of hyperplanes corresponding to a vertex results in the same condition of $\langle \beta \rangle$ having a $\Lambda$-rational point. Therefore, the compatible simplex corresponding to $\beta$ is rational Delzant if and only if $\beta$ is rational with respect to $\Lambda$, and an example of a common lattice is $\Lambda/\langle \beta \rangle$. Furthermore, the simplex is Delzant if and only if $\beta$ has coordinates $[1,\ldots,1]$ with respect to \eqref{Lambda generators} and, the corresponding lattice being $\Lambda/\langle \beta \rangle$. Since $C_r$ are arbitrary, we can argue that Veronese-compatible simplex is always Delzant with respect to the appropriate choice of the lattice.
\end{example}

We wish to reinterpret the Veronese identity \eqref{1 identity} in terms of a Veronese-compatible simplex. By fixing a new affine chart, the $\varphi^t$-image of
\begin{align}\label{new chart}
\nonumber
(1,x_1) \otimes \cdots \otimes (1,x_m)
+
(1,x_1) \otimes \cdots \otimes (1,x_{m-1}) \otimes (1,x_{m+1})
+\\ \cdots +
(1,x_2) \otimes \cdots \otimes (1,x_{m+1}),
\end{align}
we find that vectors \eqref{simplicial} in this new chart are exactly summands of the identity \eqref{1 identity}, where the sum \eqref{new chart} goes over each linearly ordered $m$-tuple from $x_1<\cdots<x_{m+1}$. In the basis of $S^mW^*$ consisting of these vectors, the right hand side of \eqref{1 identity} has coordinate expression $[1,\ldots,1]$, and thus, using the example above, the corresponding simplex is Delzant.

\begin{example}
We finish this section with general polytopes compatible with the Veronese factorization structure $S^mW^*$ in the chart $\epsilon$. Let
\begin{align}\label{cyclic}
\frac{\psi([1:x_r])}{\langle \psi([1:x_r]), \epsilon \rangle},
\hspace{.2cm}
r=1,\ldots,m+1+\ell,
\end{align}
generate edges of $\sigma^\vee$, and let
\begin{align}
C_r\frac{\psi([1:x_r])}{\langle \psi([1:x_r]), \epsilon \rangle} \text{ mod } \beta,
\hspace{.2cm}
r=1,\ldots,m+1+\ell,
\end{align}
be normals of the compatible polytope corresponding to $\beta\in\text{Int}(\sigma^\vee)$, where $C_r$ are positive constants. By \Cref{common lattice}, these belong to a common lattice if and only if for each $j\in\{m+1,\ldots,m+\ell\}$ there exist $\alpha^1,\ldots,\alpha^m,\alpha^j\in\mathbb{Q}$ such that
\begin{align}
\sum_{r=1}^m
\alpha^r C_r\frac{\psi([1:x_r])}{\langle \psi([1:x_r]), \epsilon \rangle} \text{ mod } \beta +
\alpha^j C_j\frac{\psi([1:x_jr])}{\langle \psi([1:x_j]), \epsilon \rangle} \text{ mod } \beta
= 0.
\end{align}
For each $j$, this means that $\langle \beta \rangle$ has a rational point with respect the lattice in $S^mW^*$ spanned by
\begin{align}\label{cyclic1}
C_r\frac{\psi([1:x_r])}{\langle \psi([1:x_r]), \epsilon \rangle},
\hspace{.2cm}
r=1,\ldots,m
\end{align}
and
\begin{align}\label{cyclic2}
C_j\frac{\psi([1:x_j])}{\langle \psi([1:x_j]), \epsilon \rangle}.
\end{align}
Now, $\langle \beta \rangle$ has a rational point with respect to each of these $\ell$ lattices if and only if it has a rational point with respect to any lattice containing all \eqref{cyclic2} for $j=1,\ldots,m+1+\ell$, as can be seen from \Cref{common lattice} and its proof. To construct such a common lattice $\Lambda$ we verify assumptions of \Cref{common lattice} by observing that for $x_j\in\mathbb{Q}$, $j=1,\ldots,m+1+\ell$, \eqref{1 identity} provides the needed rational dependences. Therefore, $\Lambda$-rational points of $\sigma^\vee$ yield rational Delzant compatible polytopes.
\end{example}

\bibliography{factorizationstructures}
\bibliographystyle{abbrv}

\end{document}